\documentclass{article}

\usepackage[final]{pdfpages}
\usepackage{stmaryrd}
\usepackage[T1]{fontenc}
\usepackage[utf8]{inputenc}
\usepackage{amsmath,amsfonts}
\usepackage{amsthm}
\usepackage{geometry}
\usepackage{setspace}
\usepackage{systeme}
\usepackage{amssymb}
\usepackage{graphicx}
\usepackage{etoolbox}
\usepackage{pstricks}
\usepackage{pst-solides3d}
\usepackage{marginnote}
\usepackage{xcolor} 
\usepackage{ulem} 
\usepackage[english]{babel}
\usepackage{hyperref}[hidelinks]
\usepackage{amsmath}
\usepackage{dsfont}
\usepackage{booktabs}
\usepackage{tabularx}
\usepackage{url}

\newtheorem{theorem}{Theorem}[section]
\newtheorem{lemma}[theorem]{Lemma}
\newtheorem{e-Proposition}[theorem]{Proposition}
\newtheorem{corollary}[theorem]{Corollary}
\newtheorem{remark}{Remark}
\newtheorem{e-definition}[theorem]{Definition}
\newtheorem{example}{Example}
\theoremstyle{definition}
\newtheorem{assumption}{Assumption}

\newcommand{\e}{\mathbb{E}}
\newcommand{\p}{\mathbb{P}}
\newcommand{\diff}{\, \mathrm{d}}

\newcommand{\un}{\mathds{1}}

\DeclareMathOperator{\tr}{tr}

\title{Quantitative propagation of chaos for particle systems with memory and their long-time behaviour}
\author{Adrienne Le Meur}

\begin{document}

\theoremstyle{definition}

\thispagestyle{empty}

\setcounter{page}{1}

\maketitle

\begin{abstract}
We study a class of interacting diffusion particle systems with memory and their mean-field limit using coupling methods. We first establish quantitative propagation of chaos estimates for this class of models. Our first main result is a uniform-in-time propagation of chaos estimate, under a suitable contraction condition. We also discuss conditions for existence of stationary solutions and show that long-time convergence to a non-stationary measure may also hold. This requires an asymptotic behaviour for the memory interaction and a memory loss condition. As a consequence, our second main result is that the particle system converges asymptotically to the same equilibrium up to the propagation of chaos error, which shows that the large-particle limit and the large-time limit are exchangeable.
\end{abstract}

\section{Introduction}

The purpose of this article is to study a system of interacting particles whose drift depends not only on the empirical distribution at the current time, but also on the past, through a memory kernel. More precisely, we consider the microscopic particle system
\begin{equation}
\label{particles}
    \diff X_t^{i, N} = \sigma \diff B_t^{i , N} -    F(X_t^{i, N}) \diff t + \displaystyle \int_0^t  L\left(t, s, X_t^{i, N}, \frac{1}{N} \sum_{j = 1}^N \delta_{X_s^{j, N}}\right) u_t(\mathrm{d} s) \ \diff t, \quad i = 1, \dots, N,
\end{equation}
where $(X_0^1, \dots, X_0^N) \sim \mu_0 \in \mathcal{P}\left( \left(\mathbb{R}^d\right)^N \right)$ is the initial condition, $(B_t^i)_{1 \leq i \leq N}$ are independent $k$-dimensional Brownian motions independent of $(X_0^1, \dots, X_0^N)$, $\sigma \in \mathcal{M}_{d, k}\left( \mathbb{R} \right)$, $u_t$ denotes a measure on $\left( \mathbb{R}_+, \mathcal{B}\left( \mathbb{R}_+ \right) \right)$ for all $t \geq 0$, $F : \mathbb{R}^d \to \mathbb{R}^d$ is a measurable function, $L: \mathbb{R}_+ \times \mathbb{R}_+ \times \mathbb{R}^d \times \mathcal{P}_p\left( \mathbb{R}^d \right) \to \mathbb{R}^d$ is a measurable function. The memory kernel $L$ makes the current time position interact with the past time marginals, and $u_t$ is a finite measure that encodes the amount of memory retained in the interaction. Typical examples for the measure $u_t$ are the uniform distribution over $[0, t]$, or over $[\max(0, t-s), t]$ for some window time $s > 0$, or exponential distributions. The case $u_t = \delta_t$ corresponds to standard memoryless mean-field interacting particles and their corresponding McKean-Vlasov diffusion processes. The particles also undergo an external force $F$. We also consider the associated nonlinear limit stochastic equation~\eqref{limit}.
\bigbreak

Let us motivate the main features of our model. In many interacting systems, the influence of the past is substantial, and neglecting it leads to a genuine loss of accuracy. This is the case in several classes of models arising from biology and statistical physics. A first important example is chemotaxis, a self-organization mechanism in which individuals interact with chemo-attractant generated by the population itself. At the macroscopic level, this phenomenon is classically described by Keller-Segel type models~\cite{KellerSegel71}. In the doubly parabolic case, the classical Keller-Segel system gives rise to a McKean-Vlasov SDE with memory, see~\cite{TalayTomasevic,TomasevicAAP} and~\cite{JabirTalayTomasevic,FournierTomasevic} for the corresponding particle system with memory. A second example comes from coarse graining in molecular dynamics, where effective models of observable functions of a fine dynamics may exhibit non-Markovian effects, for example with the Volterra-type models derived with the Mori-Zwanzig method~\cite{Zwanzig61,Mori65,GKS04}. These examples motivate the study of stochastic systems with memory as a structurally relevant modelling aspect. This is what is studied in recent works, by considering path-dependent McKean–Vlasov equations from the viewpoint of optimal control~\cite{optimalcontrolpathdependentmckeanvlasov}.

Another key ingredient of our model is the mean-field structure of the interaction. Mean-field interactions provide a description of systems in which each particle feels only the averaged influence of the population, either at the microscopic or at the macroscopic level. The probabilistic study of such limits goes back to the work of McKean~\cite{McKean66} and has since been developed into a vast literature on McKean-Vlasov equations and propagation of chaos. Standard references include the Saint-Flour notes of Sznitman~\cite{Sznitman}, the survey chapter by Méléard~\cite{Meleard96}, and the review~\cite{ChaintronDiez22a, ChaintronDiez22b}.
\bigbreak

Our objective is to quantify the convergence of the microscopic model towards the macroscopic one. In the usual terminology, this is a propagation of chaos problem: As the number of particles tends to infinity, the finite-dimensional marginals asymptotically tensorize, so that independence emerges at the limit~\cite{Sznitman,ChaintronDiez22a, ChaintronDiez22b}. When the speed of propagation of chaos is bounded uniformly in time, this is called uniform-in-time propagation of chaos~\cite{MALRIEU2001109}.

Our first main result Theorem~\ref{unifpropachaos} establishes precisely such a uniform-in-time propagation of chaos estimate, under a so-called contraction condition (in a sense that will be made precise in Theorem~\ref{unifpropachaos}).
An important feature of our criterion is that it only compares the asymptotic strength of memory with the confining effect of the force, which makes the condition both flexible and easy to interpret (see Section~\ref{secexamples}). The propagation of chaos estimate is done in Wasserstein distance, and for the time marginals of the empirical measure.

Our second main result concerns the asymptotic behaviour in time of the nonlinear dynamics and of the particle system, under the same condition as the uniform-in-time propagation of chaos, and under the assumption that, asymptotically in time, the effective memory interaction is stationary. In this case, the nonlinear dynamics converges in the long-time limit to a measure (Proposition~\ref{longtime}), even if there does not exist a stationary solution to the non-Markovian equation. The lack of existence of an invariant measure in all generality is due to the memory effect.
We therefore distinguish between stationary solutions and long-time equilibrium (see Definition~\ref{def:equilibrium}).
\bigbreak

In the Markovian McKean-Vlasov setting, the relation between propagation of chaos, convergence to equilibrium, and stationary solutions is classical~\cite{MALRIEU2001109,BolleyGuillinMalrieu10}. See also~\cite{BaoScheutzowYuan22} for a recent result on equilibrium in the functional McKean-Vlasov framework. Our result shows that a similar long-time convergence exists in the non-Markovian setting.

For systems with instantaneous interactions, the literature on propagation of chaos is by now very rich, both at the qualitative and quantitative levels. General methods were developed: coupling, Wasserstein stability, BBGKY hierarchies, entropy methods, and functional inequalities~\cite{Meleard96,ChaintronDiez22a, ChaintronDiez22b}. In particular, explicit convergence rates are known in several settings, including singular interaction regimes~\cite{LackerLeFlem23,FournierJourdain, olivera2025quantitativeapproximationburgerskellersegel}, and uniform-in-time propagation of chaos has been proved for several classes of instantaneous mean-field systems~\cite{DurmusEberleGuillinZimmer20,LackerLeFlem23,Eberle16,BolleyGuillinMalrieu10}. By contrast, much less seems to be available for genuinely non-Markovian mean-field models with memory, except a finite-time horizon propagation of chaos result~\cite{Bernouparticlemethod}, and Theorem~$2.2.6$ of~\cite{theseMilica}. This lack of quantitative results is one of the main motivations of our work. Let us also stress that allowing singular memory kernels would be a very natural direction, especially in view of the analogy with the two-dimensional Patlak-Keller-Segel equation~\cite{FournierJourdain}. However, even in the Lipschitz setting, obtaining uniform-in-time propagation of chaos already raises difficulties. The only known results for singular interaction are non-quantitative~\cite{JabirTalayTomasevic, FournierTomasevic}.

Indeed, the main difficulty compared with the Markovian setting is that the coupling error no longer satisfies a standard Grönwall inequality, because the interaction depends on the full past trajectory of the empirical measure. This leads naturally to integro-differential estimates. We show that, despite this non-Markovian structure of the memory models, we can still obtain quantitative propagation of chaos estimates, including bounds that are uniform-in-time, under suitable assumptions.
\bigbreak

\paragraph{Plan of the paper.} In Section~\ref{secresults}, we present the results, by first exposing the framework and the main assumptions, then the propagation of chaos results, and finally the results on stationary solutions and long-time convergence. In Section~\ref{secexamples}, we present applications of our work to different models. The proofs of the results on propagation of chaos (Theorem~\ref{bigtheorem} and Theorem~\ref{propachaos}) can be found in Section~\ref{secproofPoC}. Then the proofs of the uniform-in-time results (Lemma~\ref{moments} and Theorem~\ref{unifpropachaos}) are presented in Section~\ref{secproofunif}, and finally the proofs on stationary solutions and long-time behaviour (Proposition~\ref{meanfieldlongtime}, Proposition~\ref{coro}, Corollary~\ref{equilibrium}, Proposition~\ref{longtime} and Corollary~\ref{exchangelimits}) are in Section~\ref{secproofequilibrium}. The proofs of the well-posedness of the particle system and the limit equation can be found in the annex in Section~\ref{secproofwellposedness}.

\section{Results}
\label{secresults}
Here we start by introducing our framework and assumptions in Section~\ref{secframework}. The results on propagation of chaos are presented in Section~\ref{secPoC}, and those on the long-time behaviour are in Section~\ref{secequilibrium}.

\subsection{Framework and main assumptions}
\label{secframework}

We denote by $\mathcal{P}\left( E \right)$ the set of probability measures on the set $E$, and $\mathcal{P}_p\left( \mathbb{R}^d \right)$ the set of probability measures on $\mathbb{R}^d$ that admit a moment of order $p$. For $\mu$ and $\nu \in \mathcal{P}_p\left( \mathbb{R}^d \right)$, we call $\pi \in \mathcal{P}\left( \mathbb{R}^d \times \mathbb{R}^d \right)$ a coupling between $\mu$ and $\nu$ if $\mu(\mathrm{d} x) = \int_y \pi(\mathrm{d} x, \diff y)$ and $\nu(\mathrm{d} y) = \int_x \pi(\mathrm{d} x, \diff y)$. The set of couplings between $\mu$ and $\nu$ is denoted $\Pi\left( \mu, \nu \right)$. Let us define the Wasserstein distance by the formula
\begin{equation*}
        W_p\left( \mu, \nu \right) := \left( \inf_{\pi \in \Pi\left( \mu, \nu \right)} \int_{x, y} \left| x - y \right|^p \; \pi(x, y) \right)^{1/p}.
\end{equation*}

We consider a system of interacting diffusion SDEs whose drift depends not only on the current empirical distribution, but also on the past trajectory of the system through a memory kernel. Our formulation is designed to cover both persistent memory and localized-in-time memory effects, including purely instantaneous interactions as a limiting case.
\bigbreak
    
Consider the system~\eqref{particles}. Let us introduce our main assumptions on the coefficients of the particle system.
\begin{assumption}
\label{assumparticles}
We suppose that
\begin{itemize}
    \item[\boldmath$(1.1)$] \textbf{well-defined memory:} $u_t$ is a finite measure with support in $[0, t]$, for all $t \geq 0$, and $s \mapsto L(t, s, x, \mu)$ is integrable w.r.t. $u_t$ for all $t \geq 0$, $x \in \mathbb{R}^d$ and $\mu \in \mathcal{P}_p\left( \mathbb{R}^d \right)$,
    \item[\boldmath$(1.2)$] \textbf{Lipschitz regularity of the memory kernel:} there exist $p \geq 1$, $h^1 : \mathbb{R}_+ \times \mathbb{R}_+ \to \mathbb{R}_+$ and $h^2 : \mathbb{R}_+ \times \mathbb{R}_+ \to \mathbb{R}_+$ non-negative measurable functions such that for all $t$, $s \in \mathbb{R}_+$, $x$, $y \in \mathbb{R}^d$, and $\mu$, $\nu \in \mathcal{P}_p\left( \mathbb{R}^d \right)$
    \begin{equation}
    \label{memory}
        \begin{cases}
			\quad |L(t, s, x, \mu) - L(t, s, y, \mu)| &\leq h^1(t, s) \, |x - y|, \\
            \quad |L(t, s, x, \mu) - L(t, s, x, \nu)| &\leq h^2(t, s) \, W_p(\mu, \nu),
		 \end{cases}
    \end{equation}
\item[\boldmath$(1.3)$] \textbf{finite total memory weight:} $h^1$ and $h^2$ are such that for all $t \geq 0$
\begin{equation}
\label{H_t}
    H_t^1 := \int_0^t h^1(t, s) \, u_t(\mathrm{d} s) < \infty, \quad H_t^2 := \int_0^t h^2(t, s) \, u_t(\mathrm{d} s) < \infty.
\end{equation}
We call the quantity $H_t^1$ the instantaneous memory weight at time $t$ relative to the position, and $H_t^2$ the instantaneous memory weight at time $t$ relative to the measure.
We suppose that there exists $0 \leq D_T < \infty$ such that for all $0 \leq t \leq T$
    \begin{equation}
    \label{D_T}
        H_t^1 \leq D_T \quad \text{and} \quad H_t^2 \leq D_T.
    \end{equation}
    The quantity $D_T$ can be seen as the supremum memory weight on $[0, T]$,
    \item[\boldmath$(1.4)$] \textbf{regularity-in-time of the memory force:} $t \mapsto \int_0^t L(t, s, x, \mu(s)) \, u_t(\mathrm{d} s)$ is a measurable continuous function for every $x \in \mathbb{R}^d$ and $\mu \in \mathcal{C}\left( \mathbb{R}_+, \mathcal{P}_p\left( \mathbb{R}^d \right) \right)$,
    \item[\boldmath$(1.5)$] \textbf{regularity of the external force:} $F$ is locally Lipschitz and there exists $\kappa \in \mathbb{R}$ such that for all $x$, $y \in \mathbb{R}^d$, we have \begin{equation}
\label{pot}
    \big\langle x-y,  F(x) -  F(y)\big\rangle \geq \kappa |x-y|^2,
\end{equation}
that is to say the drift $-F$ is one-sided Lipschitz.
\end{itemize}
\end{assumption}
\begin{remark}
    If $\kappa > 0$, then it means the external force $F$ confines the particles. But in general, with Assumption~\ref{assumparticles}, the external force $F$ is allowed not to display a confining effect. The function $F$ can be zero for example.
    
    The quantities $H_t^1$ and $H_t^2$ play the role of an effective interaction strength induced by the memory kernel. Their role is central in our main results.
\end{remark}
\begin{example}
Let us take $\psi$ a Lipschitz function, and $\mathcal{D}$ a convex domain. The memory kernel
\begin{equation*}
    L(t, s, x, \mu) = - \frac{x}{\| x \| + 1} \psi\left( \left[ \int dist^p(y, \mathcal{D}) \mu(\diff y) \right]^{1/p} \right),
\end{equation*}
is Lipschitz w.r.t. $W_p$ but not w.r.t. $W_k$ with $k < p$. It penalizes the particles that wander away from $\mathcal{D}$ the domain of acceptable positions. If $p$ is large, even just one particle that escapes creates a stress.
\end{example}
We will drop $N$ in the index for readability. Before studying mean-field limits, we first show that the non-Markovian particle system~\eqref{particles} is well posed under the above hypotheses.
\begin{e-Proposition}
\label{wellposednesspart}
    Let Assumption~\ref{assumparticles} hold, let $m \geq \max(p, 2) \geq p \geq 1$, and $X_0^i$, $i = 1, \dots, N$ be $\mathcal{F}_0$-measurable random variables with a finite $m$-th moment. There exists a unique strong solution to the system of SDEs~\eqref{particles}.
\end{e-Proposition}

We next introduce the corresponding nonlinear McKean–Vlasov equation and establish its well-posedness under the same structural assumptions.
\begin{equation}
\label{limit}
     \diff \bar{X}_t = \displaystyle{\sigma \diff B_t -    F(\bar{X}_t) \diff t + \int_0^t L(t, s, \bar{X}_t, \mu_s) \ u_t(\mathrm{d} s) \ \diff t, \quad \mu_t = Law(\bar{X}_t), \quad t \geq 0}.
\end{equation}
\begin{e-Proposition}
\label{wellposedness}
    Let Assumption~\ref{assumparticles} hold. Let $q \geq p \geq 1$, and $\bar{X}_0$ be $\mathcal{F}_0$-measurable random variable with a finite $q$-th moment. There exists a unique strong solution to the McKean-Vlasov SDE~\eqref{limit}. Moreover, the solution $(\bar{X}_t)_{t \geq 0}$ is such that for all $T > 0$, $\sup_{0 \leq t \leq T} \e\left|\bar{X}_t\right|^q < \infty$.
\end{e-Proposition}
These two results provide the basis for the study of propagation of chaos and long time convergence for the particle system and its mean-field limit.

\begin{table}[h]
\centering
\caption{Important parameters of our framework}
\begin{tabularx}{\textwidth}{l X X X}
\toprule
\textbf{Notation} & \textbf{Interpretation of the parameter} & \textbf{Where to look} \\
\midrule

$\kappa$
& confinement strength of the external force $F$. If $\kappa \leq 0$, it means $F$ is not confining
& (see Assumption~\ref{assumparticles}) \\

$H_t^1$
& memory weight for the position at time $t$ associated to the memory kernel $L$ 
& (see definition in~\eqref{H_t}) \\

$H_t^2$
& memory weight for the measure at time $t$ associated to the memory kernel $L$ 
& (see definition in~\eqref{H_t}) \\

$\overline{D}^1$
& asymptotic memory weight for the spatial fluctuation
& (see Assumption~\ref{contraction}) \\

$\overline{D}^2$
& asymptotic memory weight for the fluctuation in measure
& (see Assumption~\ref{contraction}) \\

$p$
& Wasserstein or moment exponent 
& (see Assumption~\ref{assumparticles}) \\

$q$
& higher integrability exponent needed for the propagation of chaos results 
& (see Section~\ref{secPoC}) \\

$v(N)$
& speed of propagation of chaos 
& (see Section~\ref{secPoC}) \\

\bottomrule
\end{tabularx}
\end{table}

\subsection{Propagation of chaos}
\label{secPoC}

Let $p \geq 1$, $d \geq 1$, $q > p$. In the present paper the following quantity will denote our speed rate:
\begin{align}
\label{v(N)}
\nonumber
    v(N) := C \times & \left( (N^{-1/2} + N^{-(q-p)/q}) \un_{p > d/2, q \neq 2 p} + (N^{-1/2} \log(1 + N) + N^{-(q-p)/q}) \un_{p = d/2, q \neq 2p} \right. \\
    &\left. + (N^{-p/d} + N^{-(q-p)/q}) \un_{0< p < d/2, q \neq d/(d-p)}\right),
    \end{align}
where $C$ is a constant depending only on $p$, $d$, and $q$. This is the convergence rate, in the distance $W_p$ to the power $p$, of the empirical measure of a law, under an assumption on the existence of a $q$-th moment of the estimated law. It appears in~\cite{FournierGuillin}.
\bigbreak

We first address whether the particle system approximates the nonlinear limit on a fixed time horizon $[0,T]$. Because the drift depends on the full past trajectory of the empirical measure, the standard coupling argument produces an integro-differential error inequality rather than a classical Grönwall estimate. We first have a very general result for the finite-time horizon.
\begin{theorem}
\label{bigtheorem}
Let Assumption~\ref{assumparticles} hold. Let $(X_t^i)_{t \geq 0, 1 \leq i \leq N}$, be the solution to~\eqref{particles} with $\mathcal{F}_0$-measurable initial conditions with finite $m$-th moment, for a certain $m \geq \max(p, 2)$, and $\bar{X}_t^i$, $i = 1, \dots, N$ be $N$ solutions to~\eqref{limit} with i.i.d $\mathcal{F}_0$-measurable initial conditions with finite $q$-th moment for a certain $q > p$, each driven by the same $N$ independent Brownian motions $(B_t^{i, N})$ as in the system~\eqref{particles}. If we consider $x : \mathbb{R}_+ \to \mathbb{R}_+$ solution to the system
\begin{equation}
\label{upperbound}
    \left\{ \begin{array}{rcl}
        x'(t) &=& \alpha(t) x(t) + \displaystyle{\int_0^t} \beta(t, s) x(s) \, u_t(\mathrm{d} s) + v(N) C(t), \quad \text{for all $0 \leq t \leq T$}\\
        x(0) &=& \displaystyle \frac{1}{N} \sum_{i=1}^N \frac{\e|X_0^i - \bar{X}_0^i|^p}{p},
    \end{array}\right.
\end{equation} where $v(N)$ is defined in~\eqref{v(N)}, and $\alpha$, $\beta$ and $C$ are defined in~\eqref{parameters}, then we have for all $0 \leq t \leq T$
\begin{equation}
    \label{Gronwall}
        \frac{1}{N} \sum_{i = 1}^N \frac{\e|X_t^i - \bar{X}_t^i|^p}{p} \leq x(t).
    \end{equation}
\end{theorem}
\begin{remark}
    Note that the integrability needed for the initial condition of the McKean-Vlasov solutions $\bar{X}_t^i$ is slightly stronger than what was needed for well posedness (we need $q > p$ instead of $q \geq p$). This is because we need to apply the result of convergence of~\cite{FournierGuillin}. 
\end{remark}
The previous theorem bounds the coupling error with a function $x$, but $x$ is defined implicitly through an auxiliary integro-differential equation. Depending on what can be said on the function $x$, it yields a bound on the error term between the system of particles and the limit equation, and this result includes several interesting cases. The following theorem gives a more readable propagation-of-chaos statement.
\begin{theorem}
\label{propachaos}
    Let us take $X_t^i$, and $\bar{X}_t^i$, $i = 1, \dots, N$ as in Theorem~\ref{bigtheorem}. Under Assumption~\ref{assumparticles}, there exist two positive constants $C_1(T)$ and $C_2(T)$ that depend only on $p$, $q$, $\kappa$, $T$, and $D_T$, such that
\begin{equation}
\label{eqPoC}
        \sup_{0 \leq t \leq T} \frac{1}{N} \sum_{i = 1}^N \frac{\e(|X_t^i - \bar{X}_t^i|^p)}{p} \leq C_1(T) \left( \frac{1}{N} \sum_{i = 1}^N \frac{
        \e(|X_0^i - \bar{X}_0^i|^p)}{p} + v(N) \cdot C_2(T) \right),
\end{equation}
where $v(N)$ is defined in~\eqref{v(N)}.
\end{theorem}
In other words, there is propagation of chaos for time marginals of the empirical measure on a finite time horizon with speed $v(N)$ when the particle system and nonlinear system are started from the same initial data, i.e. for all $i = 1, \dots, N$, $X_0^i = \bar{X}_0^i$.
\bigbreak

Up to this point, we only have finite-time estimates, and no control on the large-time behavior of the particle approximation. To obtain uniform-in-time bounds, we need the confining effect of the external force to dominate the influence of memory, for the particles not to wander too much, at least asymptotically in time.
Let us introduce a \textbf{contraction condition}.
\begin{assumption}
\label{contraction}
    We use the notations of Assumption~\ref{assumparticles}. Denote $\overline{D}^1 := \lim\sup_t H_t^1$ and $\overline{D}^2 := \lim\sup_t H_t^2$.
    We assume that $\kappa > \overline{D}^1 + \overline{D}^2$.
\end{assumption}
\begin{remark}
This condition only involves the asymptotic memory weights $\overline{D}^1$ and $\overline{D}^2$, rather than the supremum values of $H_t^1$ and $H_t^2$. This distinction is important for systems with bursts of memory. See example in Section~\ref{stressexample}.
\end{remark}
We first establish a uniform-in-time bound on the moments of the limit system.
\begin{lemma}
\label{moments}
Let $q > p$ such that $q \geq 2$. Under Assumptions~\ref{assumparticles} and~\ref{contraction}, and if 
\begin{equation*}
    \sup_{t \geq 0} \int_0^t |L(t, s, 0, \delta_0)| \ u_t(\mathrm{d} s) < \infty,
\end{equation*}
the solution $\bar{X}_t$ to~\eqref{limit} associated to the initial condition $\bar{X}_0 \in \mathcal{P}_q\left( \mathbb{R}^d \right)$ is such that 
\begin{equation*}
    \sup_{t \geq 0} \e(|\bar{X}_t|^q) < \infty.
\end{equation*}
\end{lemma}
We use the notations of Assumption~\ref{assumparticles} and denote \begin{equation}
\label{momentuniform}
M_{p, q} := \left( \underset{t \geq 0}{\sup }\; \e(|\bar{X}_t|^q) \right)^{p/q}.
\end{equation}
Under all of the above assumptions, we achieve a uniform-in-time propagation of chaos estimate.
\begin{theorem}
\label{unifpropachaos}
    Let $q > p$ such that $q \geq 2$. Let us assume that Assumption~\ref{assumparticles} and~\ref{contraction} are satisfied.
Let us take $X_t^i$, and $\bar{X}_t^i$, $i = 1, \dots, N$ as in Theorem~\ref{bigtheorem}. Consider $M_{p, q}$ defined in~\eqref{momentuniform}.  If 
\begin{equation*}
    \sup_{t \geq 0} \int_0^t |L(t, s, 0, \delta_0)| \ u_t(\mathrm{d} s) < \infty,
\end{equation*}
there exist constants $C_1$, $C_2$ and $C_3$ depending only on $p$, $q$, $\kappa$, $\overline{D}^1$ and $\overline{D}^2$, where $C_1$ and $C_2$ are positive, such that we have
    \begin{equation}
    \label{eq:PoCunif}
    \sup_{t \geq 0} \frac{1}{N} \sum_{i=1}^N \e( |X_t^i - \bar{X}_t^i|^p) \leq C_1 M_{p, q} v(N) + C_2 \left( \frac{1}{N} \sum_{i=1}^N \e( |X_0^i - \bar{X}_0^i|^p) - C_3 M_{p, q} v(N)\right)_+,
    \end{equation} 
    where $v(N)$ is defined in~\eqref{v(N)}.
\end{theorem}
In particular, when the particle system and nonlinear system are started from the same initial data, i.e. $X_0^i = \bar{X}_0^i$ a.s, for all $i = 1, \dots, N$ with the notations of the corollary, there is uniform-in-time propagation of chaos with speed $v(N)$.

\begin{remark}\label{rem:creationofchaos}
    By analogy with the memoryless McKean-Vlasov situation, it might be tempting to expect that, at the cost of removing the supremum in time in the left-hand side of~\eqref{eq:PoCunif}, under the contraction condition of Assumption~\eqref{contraction}, the inequality actually holds with $C_1 = C_1(t)$ vanishing as $t\rightarrow \infty$. In this situation, even if particles are initially highly correlated so that it is not possible to couple them with i.i.d. particles in order to have $e(0):= \frac{1}{N} \sum_{i=1}^N \e( |X_0^i - \bar{X}_0^i|^p)$ vanishing when $N \to + \infty$,  it would still be true that, for $(t_N)_N$ such that $C_1(t_N) e(0) \xrightarrow[N \to + \infty]{} 0$, we would be able to bound $\frac{1}{N} \sum_{i=1}^N \sup_{t \geq t_N} \e( |X_t^i - \bar{X}_t^i|^p)$
by some vanishing function as $N\rightarrow\infty$.  This would mean that the particles become approximately independent in large time although they were not initially. This is sometimes referred to as \textit{creation of chaos} (see e.g.~\cite{bernou2025creation} and references within). However, with memory, the contraction condition is not enough, as illustrated with the counter-example  in Section~\ref{seccounterexample}.
\end{remark}

\subsection{Equilibrium and long-time convergence}
\label{secequilibrium}

We study the long time asymptotic behavior of the nonlinear dynamics. 
\begin{e-definition}
\label{def:equilibrium}
    We say that $\mu_\infty$ is an equilibrium of~\eqref{limit} if, for all initial distribution, $\mu_t$ converges to $\mu_\infty$ as $t \to + \infty$.
\end{e-definition}
In the usual memoryless McKean-Vlasov case, such an equilibrium is necessarily a stationary solution in the sense that $\bar{X}_t \sim \mu_\infty$ for all $t \geq 0$ if $\bar{X}_0 \sim \mu_\infty$. For such an equilibrium, we need the system to forget its past. We formalize this through the following \textbf{memory loss condition}.
\begin{assumption}
\label{memoryloss}
Assume that for $t_0 \geq 0$ large enough, for all $t \geq t_0$, $H_t^2 > 0$. Assume also that for all $s > 0$, we have
    \begin{equation*}
        \int_0^s \frac{h^2(t, v) \; u_t(\mathrm{d} v)}{\int_0^t h^2(t, w) \; u_t(\mathrm{d} w)} \xrightarrow[t \to + \infty]{} 0.
    \end{equation*}
\end{assumption}
In other words, Assumption~\ref{memoryloss} states that any fixed time window around initial time becomes asymptotically negligible compared with the total effective memory weight.
We refer to Subsection~\ref{seccounterexample} for a comment on this assumption.
\bigbreak
Under contraction and memory loss conditions, solutions to the nonlinear mean-field equation contract asymptotically in time.
\begin{e-Proposition}
\label{meanfieldlongtime}
    Suppose that Assumption~\ref{assumparticles},~\ref{contraction} and~\ref{memoryloss} are satisfied. Let $\bar{X}_t$ and $\bar{Y}_t$ be solutions to the system~\eqref{limit} with the initial conditions respectively $\bar{X}_0$ and $\bar{Y}_0$ and driven by the same Brownian motion. Suppose there exists $q > p$ such that $\bar{X}_0, \bar{Y}_0 \in \mathbb{L}^q\left( \mathbb{R}^d \right)$. Finally, assume that 
    \begin{equation*}
        \sup_{t \geq 0} \int_0^t |L(t, s, 0, \delta_0)| \ u_t(\mathrm{d} s) < \infty.
    \end{equation*}
    There exists $m : \mathbb{R}_+ \to \mathbb{R}_+$ a decreasing function such that $m(t)$ goes to $0$ as $t$ goes to infinity and
    \begin{equation*}
        \e(|\bar{X}_t - \bar{Y}_t|^p) \leq m(t) \e(|\bar{X}_0 - \bar{Y}_0|^p).
    \end{equation*}
\end{e-Proposition}

Since all of the solutions to the system~\eqref{limit} have the same behaviour asymptotically in time, no matter the initial distribution, it is legitimate to ask whether there exists a stationary solution or not, in which case the stationary solution is a candidate for the limit behaviour. This is indeed the case under an assumption of \textbf{stationary effective memory}.
\begin{assumption}
\label{b}
There exists $b : \mathbb{R}^d \times \mathcal{P}_p\left( \mathbb{R}^d \right) \to \mathbb{R}^d$ such that, for all $t \geq 0$, $x \in \mathbb{R}^d$, and $\mu \in \mathcal{P}_p\left( \mathbb{R}^d \right)$, we have \[ \int_0^t L(t, s, x, \mu) \, u_t(\mathrm{d} s) = b(x, \mu).\]
\end{assumption}

\begin{remark}
For example if \[L(t, s, x, \mu) := K(t,s) b(x, \mu) / C,\] and if additionally for all $t \geq 0$
\[\int_0^t K(t, s) \, u_t(\mathrm{d} s) = C,\]
then Assumption~\ref{b} is satisfied. This can be interpreted as a memory kernel with separation of variables, and where the weight of the time component of the memory interaction is constant over time.
\end{remark}

\begin{e-Proposition}
\label{coro}
Suppose that Assumption~\ref{assumparticles} and~\ref{b} are satisfied. Let us take $\mu_\infty \in \mathcal{P}_p\left( \mathbb{R}^d \right)$. Then the following statements are equivalent.
\begin{itemize}
    \item[$(1)$] The measure $\mu_\infty$ is an invariant measure of 
    \begin{equation}
        \label{fixed}
        \diff \bar{X}_t = - F(\bar{X}_t) \diff t + b(\bar{X}_t, \mu_\infty) \diff t + \sigma \diff B_t.
    \end{equation}
    \item[$(2)$] Let $\bar{X}_t$ be the solution to the McKean-Vlasov SDE 
    \begin{equation}
    \label{changing}
        \diff \bar{X}_t = - F(\bar{X}_t) \diff t + b(\bar{X}_t, \mu_t) \diff t + \sigma \diff B_t, \quad Law\left( \bar{X}_t \right) = \mu_t,
    \end{equation}
    with initial condition $\bar{X}_0 \sim \mu_\infty$. Then $\mu_t = \mu_\infty$ for all $t \geq 0$.
    \item[$(3)$] Let $\bar{X}_t$ be the solution to the McKean-Vlasov SDE 
    \begin{equation}
        \label{fullmemory}
    \diff \bar{X}_t = - F(\bar{X}_t) \diff t + \int_0^t L(t, s, \bar{X}_t, \mu_s) \; u_t(\mathrm{d} s) \ \diff t + \sigma \diff B_t, \quad Law\left( \bar{X}_t \right) = \mu_t,
    \end{equation}
    with initial condition $\bar{X}_0 \sim \mu_\infty$. Then $\mu_t = \mu_\infty$ for all $t \geq 0$.
    \end{itemize}
\end{e-Proposition}
This shows that the equilibrium of the full memory equation~\eqref{fullmemory}, which here is a stationary solution thanks to Assumption~\ref{b}, can be characterized through a memoryless McKean–Vlasov dynamics.

\begin{corollary}
\label{equilibrium}
Under Assumption~\ref{assumparticles},~\ref{contraction} and~\ref{b}, and if 
\begin{equation*}
    \sup_{t \geq 0} \int_0^t |L(t, s, 0, \delta_0)| \ u_t(\mathrm{d} s) < \infty,
\end{equation*}
there exists a unique equilibrium to~\eqref{fullmemory}, which is the unique stationary solution.
\end{corollary}

Note that Assumption~\ref{b} is restrictive and there could be an equilibrium even if there is no stationary solution. This is the case for example in Section~\ref{secexKellerSegel}. A consequence of Proposition~\ref{meanfieldlongtime} is the long time convergence of the limit trajectory defined by~\eqref{limit} to the equilibrium, under a weaker assumption than Assumption~\ref{b} that, asymptotically in time, the effective memory has a stationary behaviour.
\begin{e-Proposition}
\label{longtime}
    Suppose that there exists $b : \mathbb{R}^d \times \mathcal{P}_p\left( \mathbb{R}^d \right) \to \mathbb{R}^d$ such that, for all $\mu \in \mathcal{P}_p\left( \mathbb{R}^d \right)$, we have \begin{equation}
    \label{asympb}
        \int_{\mathbb{R}^d} \left| \int_0^t L(t, s, x, \mu) \, u_t(\mathrm{d} s) - b(x, \mu) \right|^p \mu(\mathrm{d} x) \xrightarrow[t \to + \infty]{} 0.
    \end{equation} 
    Suppose also that Assumption~\ref{assumparticles}, ~\ref{contraction} and~\ref{memoryloss} are satisfied, and that
    \begin{equation*}
        \sup_{t \geq 0} \int_0^t |L(t, s, 0, \delta_0)| \ u_t(\mathrm{d}s) < \infty.
    \end{equation*}
    Then, there exists a unique $\mu_\infty \in \mathcal{P}_p\left( \mathbb{R}^d \right)$ such that, if $\bar{X}_t$ is the solution to the McKean-Vlasov SDE 
    \begin{equation}
        \label{McKeanmemory}
        \diff \bar{X}_t = - F(\bar{X}_t) \diff t + \int_0^t L(t, s, \bar{X}_t, \mu_s) \; u_t(\mathrm{d} s) \ \diff t + \sigma \diff B_t, \quad Law\left( \bar{X}_t \right) = \mu_t,
    \end{equation} 
    with initial condition $\bar{X}_0 \sim \mu_0 \in \mathcal{P}_p\left( \mathbb{R}^d \right)$, then $\mu_t \xrightarrow[t \to +\infty]{} \mu_\infty$. The measure $\mu_\infty$ is the unique stationary solution to memoryless McKean-Vlasov equation
    \begin{equation*}
        \diff \bar{X}_t = - F(\bar{X}_t) \diff t + b(\bar{X}_t, \mu_t) \diff t + \sigma \diff B_t, \quad Law\left( \bar{X}_t \right) = \mu_t,
    \end{equation*}
    which can be seen as the stationary version of~\eqref{McKeanmemory}.
\end{e-Proposition}
\begin{remark}
See Section~\ref{stressexample} for an example where~\eqref{asympb} is satisfied.
\end{remark}
By combining the uniform-in-time propagation of chaos result (Theorem~\ref{unifpropachaos}), with the previous result on the long-time convergence of the limit equation (Proposition~\ref{longtime}), we get the long-time behaviour of the particles.
\begin{corollary}
\label{exchangelimits}
Under Assumptions~\ref{assumparticles},~\ref{contraction},~\ref{memoryloss}, and~\ref{b}, and if 
\begin{equation*}
    \sup_{t \geq 0} \int_0^t |L(t, s, 0, \delta_0)| \ u_t(\mathrm{d} s) < \infty,
\end{equation*}
there exist a decreasing function $m : \mathbb{R}_+ \to \mathbb{R}_+$, and a constant $C > 0$ such that for all $t \geq 0$, and $N \in \mathbb{N}^*$, we have \begin{equation*}
    W_p(Law(X_t^{1, N}) \, , \, \mu_\infty) \leq C v(N)^{1/p} + m(t),
\end{equation*}
with $X_t^{i, N}$, for $i = 1, \dots, N$ solution to~\eqref{particles} with i.i.d. initial condition with a moment of order $q > p$.
\end{corollary}
\begin{remark}
    This shows that the limit in $N \to + \infty$ and the limit in $t \to + \infty$ are exchangeable.
\end{remark}

\section{Examples, applications and perspectives}
\label{secexamples}
In this section, we present applications of our work to models from existing papers, or to toy models that illustrate the relevance of the mathematical setting and the assumptions.
\subsection{A non-contractive memory model}
\label{counterexcontractioncondition}
We introduce a simple non-contractive model showing that uniform-in-time propagation of chaos may fail as soon as the contraction condition of Assumption~\ref{contraction} is not satisfied.
Let us fix $D^1$ and $D^2$. We then take $\kappa := D^1 + D^2$.
Let us take $\alpha > 0$ and consider the particle model
\begin{equation}
\label{Eq:partnoncontractive}
    \diff X_t^i = - \kappa X_t^i \diff t + \left[ D^1 X_t^i + D^2 \frac{\alpha}{N} \sum_{j = 1}^N \int_0^t X_s^j \, e^{- \alpha (t - s)} \diff s \right] \diff t + \diff B_t^i, \quad i = 1, \dots, N.
\end{equation}
\begin{remark}
    If $\kappa = D^2$ and $D^1 = 0$, the drift of this model can be interpreted as the particle $i$ being attracted to the mean over time and particles of positions.
\end{remark}
Let us denote
\begin{equation*}
    F(x) := \kappa x, \quad L(t, s, x, \mu) := D^1 x + D^2 \int y \diff \mu(y), \quad \text{and } u_t(\mathrm{d} s) := \alpha \, e^{- \alpha (t - s)} \un_{[0, t]}(s) \diff s.
\end{equation*}
Then Assumption~\ref{assumparticles} is satisfied with $p = 1$, $h^1 \equiv D^1$, and $h^2 \equiv D^2$, so that $\kappa = \overline{D}^1 + \overline{D}^2 := \limsup H_t^1 + \limsup H_t^2$. Additionally we have
\begin{equation*}
    \int_0^t |L(t, s, 0, \delta_0)| \ u_t(\mathrm{d} s) = D^2,
\end{equation*}
which is bounded uniformly in time.
We fall exactly in the limit case of Assumption~\ref{contraction}. Let us prove that propagation of chaos is not uniform-in-time. This will show that the contraction condition $\kappa > \overline{D}^1 + \overline{D}^2$ is sharp. We will do this in two steps.

\paragraph{Step 1. We determine the order in $t$ of $\e\left( |X_t^i|^2 \right)$.}
Denote the memory variable
\begin{equation*}
    Z_t := \frac{\alpha}{N} \sum_{j = 1}^N \int_0^t X_s^j \, e^{- \alpha (t - s)} \diff s.
\end{equation*}
Then, because $\kappa - D^1 = D^2$, we get
\begin{equation*}
    \left\{
    \begin{array}{rcl}
        \diff X_t^i &=& - D^2 X_t^i \diff t + D^2 Z_t \diff t + \diff B_t^i, \quad i = 1, \dots, N, \\
        \diff Z_t &=& \alpha \left( \frac{1}{N} \sum_{i = 1}^N X_t^i - Z_t \right) \diff t.
    \end{array}
    \right.
\end{equation*}
It is then natural to introduce the quantity $Y_t := \frac{1}{N} \sum_{i = 1}^N X_t^i$. Since $\frac{1}{N} \displaystyle{\sum_{i = 1}^N} B_t^i$ is a combination of independent Brownian motions, it is a Gaussian process, of variance function $\frac{t}{N}$. Therefore there exists $W_t$ a Brownian motion such that $\frac{1}{N} \displaystyle{\sum_{i = 1}^N} B_t^i = \frac{1}{\sqrt{N}} W_t$. This yields the following system on $(Y_t, Z_t)$
\begin{equation*}
     \left\{
    \begin{array}{rcl}
        \diff Y_t &=& \left[ - D^2 Y_t + D^2 Z_t \right] \diff t + \frac{1}{\sqrt{N}} \diff W_t, \\
        \diff Z_t &=& \alpha \left( Y_t - Z_t \right) \diff t.
    \end{array}
    \right.
\end{equation*}
We then transform the system into
\begin{equation*}
    \left\{
    \begin{array}{rcl}
        \diff \left( \alpha Y_t + D^2 Z_t \right) &=& \frac{\alpha}{\sqrt{N}} \diff W_t, \\
        \diff \left( Y_t - Z_t \right) &=& - \left( D^2 + \alpha \right) \left( Y_t - Z_t \right) \diff t + \frac{1}{\sqrt{N}} \diff W_t,
    \end{array}
    \right.
\end{equation*}
so that
\begin{equation*}
    \left\{
    \begin{array}{rcl}
        \alpha Y_t + D^2 Z_t &=& \alpha Y_0 + D^2 Z_0 + \frac{\alpha}{\sqrt{N}} W_t, \\
        Y_t - Z_t &=& \left( Y_0 - Z_0 \right) e^{- \left( D^2 + \alpha \right) t} + \frac{1}{\sqrt{N}} \displaystyle \int_0^t e^{- \left( D^2 + \alpha \right) (t - s)} \diff W_s.
    \end{array}
    \right.
\end{equation*}
Let us remark that $Z_0 = 0$ and use this to deduce the expression of $Y_t$
\begin{equation*}
        Y_t = \frac{\alpha}{D^2 + \alpha} Y_0  + \frac{D^2}{D^2 + \alpha} Y_0 e^{- \left( D^2 + \alpha \right) t} + \frac{\alpha}{\left( D^2 + \alpha \right) \sqrt{N}} W_t + \frac{D^2}{\left( D^2 + \alpha \right) \sqrt{N}} \int_0^t e^{- \left( D^2 + \alpha \right) (t - s)} \diff W_s.
\end{equation*}
Thanks to the Itô isometry, we have
\begin{equation*}
    \e\left[ \left| \int_0^t e^{- (D^2 + \alpha) (t - s)} \diff W_s \right|^2 \right] = \int_0^t e^{- 2 (D^2 + \alpha) (t - s)} \diff s \underset{t \to + \infty}{=} O(1),
\end{equation*}
and 
\begin{equation*}
    \e\left[ W_t \int_0^t e^{- (D^2 + \alpha) (t - s)} \diff W_s \right] = \int_0^t e^{- (D^2 + \alpha) (t - s)} \diff s \underset{t \to + \infty}{=} O(1)
\end{equation*}
so that we can compute the variance of $Y_t$ thanks to the Itô isometry
\begin{equation*}
    Var\left( Y_t\right) \underset{t \to + \infty}{=} O(1) + \frac{\alpha^2 t}{(D^2 + \alpha)^2 N}.
\end{equation*}
We also establish that $U_t^i := X_t^i - Y_t$ satisfies
\begin{equation*}
    \diff U_t^i = - D^2 U_t^i \diff t + \diff B_t^i - \frac{1}{\sqrt{N}} \diff W_t,
\end{equation*}
so that it has bounded-in-time variance, by similar arguments to the ones used before. It can additionally be expressed by the formula
\begin{equation*}
    U_t^i = U_0^i e^{- D^2 t} + \sigma \int_0^t e^{- D^2 (t - s)} \diff \beta_s^i.
\end{equation*}
where $\beta_t^i$ is a Brownian motion, and \begin{equation*}
    \sigma := \sqrt{\left( 1 - \frac{1}{N} \right)^2 + \frac{N - 1}{N^2}} = \sqrt{1 - \frac{1}{N}}.
\end{equation*}
By Itô isometry again, $Var\left( U_t^i \right) \underset{t \to + \infty}{=} O(1)$. Additionally,
\begin{align*}
    \langle \beta^i, W \rangle_t &= \left\langle B^i - \frac{1}{\sqrt{N}} W, W \right\rangle_t \\
    &= \frac{t}{\sqrt{N}} - \frac{t}{\sqrt{N}} = 0,
\end{align*}
so that
\begin{align*}
    \e\left( Y_t U_t^i \right) \underset{t \to + \infty}{=}& O(1) + \frac{\sigma \alpha}{(D^2 + \alpha) \sqrt{N}} \e\left[ W_t \int_0^t e^{- D^2 (t - s)} \diff \beta_s^i \right] \\
    &+ \frac{\sigma (D^2)^2}{(D^2 + \alpha) \sqrt{N}} \e\left[ \int_0^t e^{- (D^2 + \alpha) (t - s)} \diff W_s \int_0^t e^{- D^2 (t - s)} \diff \beta_s^i \right] \\
    \underset{t \to + \infty}{=}& O(1) + \frac{\sigma \alpha}{(D^2 + \alpha) \sqrt{N}} \int_0^t e^{- D^2 (t - s)} \diff \langle \beta^i, W \rangle_s \\
    &+ \frac{\sigma (D^2)^2}{(D^2 + \alpha) \sqrt{N}} \int_0^t e^{- (2 D^2 + \alpha) (t - s)} \diff \langle \beta^i, W \rangle_s \underset{t \to + \infty}{=} O(1).
\end{align*}
Therefore $X_t^i$ has variance
\begin{align*}
    Var\left( X_t^i \right) &= Var\left( Y_t \right) + 2 Cov\left( Y_t, U_t^i \right) + Var\left( U_t^i \right) \\
    &\underset{t \to + \infty}{=} O(1) + \frac{\alpha^2 t}{(D^2 + \alpha)^2 N}.
\end{align*}

\paragraph{Step 2. We study the variance of the limit equation.}
Let us define $\bar{X}_t$ the non-linear limit process satisfying
\begin{equation*}
    \diff \bar{X}_t = - D^2 \left[ \bar{X}_t - \alpha \int_0^t \mu_s \, e^{- \alpha (t - s)} \diff s \right] \diff t + \diff B_t, \quad \e\left( \bar{X}_t \right) = \mu_t.
\end{equation*} 
By introducing $z_t := \alpha \int_0^t \mu_s \, e^{- \alpha (t - s)} \diff s$, we get the linear system
\begin{equation*}
    \left\{
    \begin{array}{rcl}
        \mu_t' &=& - D^2 ( \mu_t - z_t ) \\
        z_t' &=& \alpha \mu_t - \alpha z_t,
    \end{array}
    \right.
\end{equation*} 
so that by choosing $\mu_0 = z_0 = 0$, we get that $\mu \equiv 0$. Therefore $\bar{X}_t$ satisfies
\begin{equation*}
    \diff \bar{X}_t = - D^2 \bar{X}_t \diff t + \diff B_t.
\end{equation*} 
That is why 
\begin{equation*}
    \bar{X}_t = \bar{X}_0 e^{- D^2 t} + \int_0^t e^{- D^2 (t - s)} \diff B_s, \quad \text{and } Var\left( \bar{X}_t \right) \underset{t \to + \infty}{=} O(1).
\end{equation*} 

\paragraph{Conclusion.} We can conclude that there is no uniform-in-time propagation of chaos thanks to the following lemma.
\begin{lemma}
If there is uniform-in-time propagation of chaos of the particle system $(X_t^i)$ to the limit process $\bar{X}_t$, and if $\e\left( |\bar{X}_t|^2 \right)$ is bounded uniformly in $t$, then 
\begin{equation*}
    \sup_{t \geq 0} \left| \sqrt{\e\left( |X_t^i|^2 \right)} - \sqrt{\e\left( |\bar{X}_t| \right)} \right| \xrightarrow[N \to + \infty]{} 0.
\end{equation*}
\end{lemma}
\begin{proof}
    By the reverse triangular inequality, we get that
    \begin{equation*}
        \left| \sqrt{\e\left( |X_t^i|^2 \right)} - \sqrt{\e\left( |\bar{X}_t| \right)} \right| \leq \sqrt{\e\left( \left| X_t^i - \bar{X}_t \right|^2 \right)},
    \end{equation*}
    so that by optimizing on the couplings, we get that 
    \begin{equation*}
        \left| \sqrt{\e\left( |X_t^i|^2 \right)} - \sqrt{\e\left( |\bar{X}_t| \right)} \right| \leq W_2\left( Law\left( X_t^i \right), Law\left( \bar{X}_t \right) \right).
    \end{equation*}
    Since the right-hand side converges to $0$ uniformly-in-time, this concludes.
\end{proof}
In the case of the particle system~\eqref{Eq:partnoncontractive}, we have
\begin{equation*}
    \left| \sqrt{\e\left( |X_t^i|^2 \right)} - \sqrt{\e\left( |\bar{X}_t| \right)} \right| \underset{t \to + \infty}{=} O(1) + \frac{\alpha^2 t}{(1 + \alpha)^2 N},
\end{equation*}
so the uniform-in-time propagation of chaos fails when $\kappa = \overline{D}^1 + \overline{D}^2$, that is to say when the contraction condition of Assumption~\ref{contraction} is not satisfied. This indeed shows that the condition $\kappa > \overline{D}^1 + \overline{D}^2$ required in Theorem~\ref{unifpropachaos} is necessary whatever the values of $\overline{D}^1$ and $\overline{D}^2$.

\subsection{Chemotaxis modelling and Keller-Segel equation}
\label{secexKellerSegel}
A motivating example comes from the Keller-Segel equation. The later is a model for the biological phenomenon called chemotaxis responsible for self-organizing behaviour of particles due to the interaction with a chemo-attractant they can produce. The usual parabolic-parabolic model is unattainable due to the singular interaction (see~\cite{FournierTomasevic, JabirTalayTomasevic}). That is why we turn to a regularized version of the Keller-Segel equation
\begin{equation*}
\left\{
    \begin{array}{rcl}
        \partial_t u(t,x) &=& \frac{1}{2} \Delta u(t, x) - \nabla \cdot \left( u(t, x) \left[ \xi \nabla h(t, x) - \nabla V(x) \right] \right) \\
        \frac{1}{\gamma} \partial_t h(t, x) &=& \frac{1}{2} \Delta h(t, x) - \alpha h(t, x) + \beta \int u(t, z) g(t - x) \diff z,
    \end{array}
    \right.
\end{equation*}
that is studied in~\cite{budhiraja2017uniformtimeinteractingparticle}.
Amarjit Budhiraja and Wai-Tong (Louis) Fan study the macroscopic model
\begin{equation*}
\left\{
    \begin{array}{rcl}
        \diff \bar{X}_t &=& \diff B_t - \nabla V(\bar{X}_t) \diff t + \chi \nabla c(t, \bar{X}_t) \diff t, \\
        \frac{1}{\gamma} \partial_t c(t, x) &=& \frac{1}{2} \Delta c(t, x) - \alpha c(t, x) + \beta \int_{\mathbb{R}^d} g(z - x) \diff \mu_t(z), \quad x \in \mathbb{R}^d, \\
        Law\left( \bar{X}_t \right) &=& \mu_t,
    \end{array}
    \right.
\end{equation*}
associated to the particle system
\begin{equation*}
    \left\{
    \begin{array}{rcl}
        \diff X_t^{i, N} &=& \diff B_t - \nabla V(X_t^{i, N}) \diff t + \chi \nabla c_N(t, X_t^{i, N}) \diff t, \quad i = 1, \dots, N \\
        \frac{1}{\gamma} \partial_t c_N(t, x) &=& \frac{1}{2} \Delta c_N(t, x) - \alpha c_N(t, x) + \frac{\beta}{N} \sum_{i = 1}^N g(X_t^{i, N} - x), \quad x \in \mathbb{R}^d.
    \end{array}
    \right.
\end{equation*}
They obtain in Theorem~$3.4$ the uniform-in-time propagation of chaos of the particle system to the parabolic-parabolic equation. It holds only for $\alpha > 0$ and their contraction condition (Assumption~$2.4$) blows up when $\alpha \to 0$.

Let us suppose for the sake of simplicity that $\beta = 1$. This model falls into our framework, with
\begin{align*}
    \sigma &:= 1, \quad F := \nabla V \\
    u_t(\mathrm{d} s) &:= \delta_0(\mathrm{d} s) + \un_{(0, t]}(\mathrm{d} s), \quad \quad \quad \quad \quad \quad \quad \quad \quad \quad \; \text{for all $t \geq 0$} \\
    L(t, s, x, \mu) &:= \begin{cases}
Q_t c_0(x), \quad \text{if $s = 0$} \\
Q_{t-s}\left( \int g(y - \cdot) \diff \mu(y) \right)(x), \quad \text{if $s > 0$},
\end{cases} \text{for all $t \geq 0$, $x \in \mathbb{R}^d$, and $\mu \in \mathcal{P}_1\left( \mathbb{R}^d \right)$}
\end{align*} with $g$, $c_0 \in \mathcal{C}_b^2\left( \mathbb{R}^d \right)$ and the associated semigroup
\begin{equation*}
    Q_t f := e^{- \gamma \alpha t} \nabla g_t \ast f := \frac{e^{- \gamma \alpha t}}{(2 \pi \gamma t)^{d/2}} \nabla \exp\left( - \frac{|\cdot|^2}{2 t \gamma}\right) \ast f.
\end{equation*}
Note that this example motivates the need for a non-homogeneous modelling of the memory.

\paragraph{Uniform-in-time propagation of chaos, case $\alpha > 0$.} The memory kernel $L$ satisfies Assumption~\ref{assumparticles}, with $p = 1$,
\begin{align*}
    h^1(t, s) := e^{- \gamma \alpha t} \| \mbox{Hess } c_0 \|_\infty \delta_{0, s} + e^{- \gamma \alpha (t - s)} \| \mbox{Hess } g \|_\infty \un_{s > 0},
\end{align*}
and
\begin{equation*}
    h^2(t, s) := e^{- \gamma \alpha (t - s)} \| \mbox{Hess } g \|_\infty \un_{s > 0}.
\end{equation*}
Let us detail the verification of the Lipschitz hypothesis in measure~\eqref{memory}. For all $t \geq 0$, $x \in \mathbb{R}^d$, $\mu \in \mathcal{P}_p\left( \mathbb{R}^d \right)$, $\nu \in \mathcal{P}_p\left( \mathbb{R}^d \right)$ and $s \neq 0$, and all $\pi$ a coupling between $\mu$ and $\nu$, we have the following computation
\begin{align*}
    | L(t, s, x, \mu) - L(t, s , x, \nu) | =& e^{- \gamma \alpha (t - s)} \left| \nabla g_{t - s} \ast g \ast ( \mu - \nu )(x) \right| \\
    \leq& e^{- \gamma \alpha (t - s)} \left| \int_y \int_{y'} g_{t - s} \ast \nabla g (x - y) \diff \pi(y, y') - \int_y \int_{y'} g_{t - s} \ast \nabla g (x - y') \diff \pi(y, y') \right| \\
    \leq& e^{- \gamma \alpha (t - s)} \| g_{t - s} \ast \mbox{Hess } g \|_\infty \int_{y, y'} | y - y' | \diff \pi(y, y').
\end{align*}
By optimization in the coupling $\pi$, and because $\| g_{t - s} \ast \mbox{Hess } g \|_\infty \leq \| \mbox{Hess } g \|_\infty$, it yields that 
\begin{equation*}
    | L(t, s, x, \mu) - L(t, s , x, \nu) | \leq h^2(t, s) W_1(\mu, \nu),
\end{equation*}
which proves the desired Lipschitz regularity~$(1.2)$ of Assumption~\ref{assumparticles}.

Additionally, we have for all $t \geq 0$
\begin{equation*}
    \int_0^t L(t, s, x, \mu) \, u_t(\mathrm{d} s) = e^{- \gamma \alpha t} g_t \ast \nabla c_0(x) + \int_0^t e^{- \gamma \alpha s} g_s \ast \nabla g \ast \mu(x) \diff s.
\end{equation*}
Therefore, $t \mapsto \int_0^t L(t, s, x, \mu) \, u_t(\mathrm{d} s)$ is continuous on $\mathbb{R}_+$ as required in~$(1.3)$ of Assumption~\ref{assumparticles}.

Let us now compute $H_t^1$ and $H_t^2$ defined in~\eqref{H_t}. We have
\begin{equation*}
    H_t^1 = e^{- \gamma \alpha t} \| \mbox{Hess } c_0 \|_\infty + \frac{1 - e^{- \gamma \alpha t}}{\gamma \alpha} \| \mbox{Hess } g \|_\infty \quad \text{and} \quad H_t^2 = \frac{1 - e^{- \gamma \alpha t}}{\gamma \alpha} \| \mbox{Hess } g \|_\infty.
\end{equation*}
This makes~$(1.3)$ of Assumption~\ref{assumparticles} satisfied. We also get that $\overline{\lim} H_t^1 +  \overline{\lim} H_t^2 = \frac{2}{\gamma \alpha} \| \mbox{Hess } g \|_\infty$. Plus, for all $s \geq 0$, we have \begin{equation*}
    \frac{\int_0^s h^2(t, u) u_t(\diff u)}{\int_0^t h^2(t, v) u_t(\diff v)} = \frac{\int_{t - s}^t e^{- \gamma \alpha u} \diff u}{\int_0^t e^{- \gamma \alpha v} \diff v} \xrightarrow[t \to + \infty]{} 0.
\end{equation*}
Therefore, thanks to our Theorem~\ref{unifpropachaos}, for all $\kappa > 2 \| \mbox{Hess } g \|_\infty / (\gamma \alpha)$, then we have the uniform-in-time propagation of chaos. This contraction condition is sharper than the one obtained in Theorem~$3.4$.

\paragraph{Uniform-in-time propagation of chaos, case $\alpha = 0$.}
The contraction condition in Assumption $2.4$ of~\cite{budhiraja2017uniformtimeinteractingparticle} blows up when $\alpha \to 0$. But we can obtain uniform-in-time propagation of chaos for $\alpha = 0$, that is to say even when the chemical signal does not vanish. Let us restrict the smoothing function $g$ to be the gaussian function $g_\varepsilon(x) := \exp\left( - \frac{|x|^2}{\gamma \alpha \varepsilon} \right) / \left( 2 \pi \gamma \varepsilon \right)^{d/2}$, and treat only the case $d = 1$ for computational simplicity. By choosing 
\begin{equation*}
    h^1(t, s) := \| \mbox{Hess } c_0 \|_\infty \delta_{0, s} + \| \mbox{Hess } g_\varepsilon \ast g_{t - s} \|_\infty \un_{s > 0} \quad \text{and} \quad h^2(t, s) := \| \mbox{Hess } g_\varepsilon \ast g_{t - s} \|_\infty \un_{s > 0}.
\end{equation*}
Also, we have
\begin{align*}
    \mbox{Hess } g_\varepsilon \ast g_{t - s}(x) = \mbox{Hess } g_{t - s + \varepsilon}(x) = \frac{1}{\sqrt{\pi} \gamma^{3/2} (t - s + \varepsilon)^{3/2}} \left( \frac{1}{\sqrt{2}} e^{- \frac{x^2}{2 (t - s + \varepsilon)
    \gamma}} + \sqrt{2} \frac{x^2}{2 (t - s + \varepsilon) \gamma} e^{- \frac{x^2}{2 (t - s + \varepsilon) \gamma}} \right),
\end{align*}
therefore 
\begin{equation*}
    \| \mbox{Hess } g_\varepsilon \ast g_{t - s}(x) \|_\infty = \frac{1}{\sqrt{\pi} \gamma^{3/2} (t - s + \varepsilon)^{3/2}} \left( \frac{1}{\sqrt{2}} + Cst \ \sqrt{2} \right),
\end{equation*}
where $Cst$ is the infinity norm of $|\cdot|^2 e^{- |\cdot|^2}$. Therefore we have
\begin{equation*}
    H_t^1 = \| \mbox{Hess } c_0 \|_\infty + \frac{1}{\sqrt{\pi} \gamma^{3/2}} \left( \frac{1}{\sqrt{2}} + Cst \ \sqrt{2} \right) \left( \frac{2}{\sqrt{\varepsilon}} - \frac{2}{\sqrt{t + \varepsilon}} \right)
\end{equation*}
\begin{equation*}
    H_t^2 = \frac{1}{\sqrt{\pi} \gamma^{3/2}} \left( \frac{1}{\sqrt{2}} + Cst \ \sqrt{2} \right) \left( \frac{2}{\sqrt{\varepsilon}} - \frac{2}{\sqrt{t + \varepsilon}} \right),
\end{equation*}
and recover uniform-in-time propagation of chaos for all
\begin{equation*}
    \kappa > \| \mbox{Hess } c_0 \|_\infty + \frac{4}{\sqrt{\pi \varepsilon} \gamma^{3/2}} \left( \frac{1}{\sqrt{2}} + Cst \ \sqrt{2} \right).
\end{equation*}
\begin{remark}
    Note that for small $\alpha$, this threshold is better than the one found in the previous paragraph, but depends on the initial condition.
\end{remark}

\paragraph{Application of our long-time results.} We restrict here to $\alpha > 0$ for simplicity. Suppose that $\kappa > 2 \| \mbox{Hess } g \|_\infty / (\gamma \alpha)$, so that Assumption~\ref{contraction} holds.

We have
\begin{equation*}
    \int_0^t |L(t, s, 0, \delta_0)| \ u_t(\mathrm{d} s) = e^{- \gamma \alpha t} g_t \ast |\nabla c_0|(0) + \int_0^t e^{- \gamma \alpha s} g_s \ast |\nabla g|(0) \diff s \leq \|\nabla c_0 \|_\infty + \frac{\|\nabla g \|_\infty}{\gamma \alpha},
\end{equation*}
so that $\sup_t \int_0^t |L(t, s, 0, \delta_0)| \ u_t(\mathrm{d} s) < \infty$. 

Additionally,
\begin{align*}
    \frac{\int_0^s h^2(t, u) u_t(\diff u)}{\int_0^t h^2(t, v) u_t(\diff v)} = \frac{e^{- \gamma \alpha ( t - s)} - e^{- \gamma \alpha t}}{1 - e^{- \gamma \alpha t}} \xrightarrow[t \to + \infty]{} 0,
\end{align*}
so that Assumption~\ref{memoryloss} holds as well.

Let us denote 
\begin{equation*}
    b(x, \mu) := \int_0^{+ \infty} e^{- \gamma \alpha s} g_s \ast \nabla g \ast \mu(x) \diff s.
\end{equation*}
By triangular inequality and then Hölder inequality for the convolutions, for all $\mu \in \mathcal{P}_p\left( \mathbb{R}^d \right)$, we have
\begin{align*}
    \int_{\mathbb{R}^d} \left| \int_0^t L(t, s, x, \mu) \ u_t(\mathrm{d} s) - b( x, \mu) \right| \ \mu(\mathrm{d} x) =& \int_{\mathbb{R}^d} \left| e^{- \gamma \alpha t} g_t \ast \nabla c_0 (x) - \int_t^{+ \infty} e^{- \gamma \alpha s} g_s \ast \nabla g \ast \mu(x) \diff s \right| \ \mu(\mathrm{d} x) \\
    \leq& e^{- \gamma \alpha t} \| \nabla c_0 \|_\infty + \int_t^{+ \infty} e^{- \gamma \alpha s} \| \nabla g \|_\infty \diff s \xrightarrow[t \to + \infty]{} 0.
\end{align*}
Therefore we obtain by Proposition~\ref{longtime} the long-time convergence towards a measure $\mu_\infty$. Additionally, if $e^{- V}$ is integrable, this measure should have a density $\rho_\infty$ which we define implicitly by a fixed point:
\begin{equation}
\label{rho_infty}
    \rho_\infty = \frac{1}{Z_{\rho_\infty}} \exp\left( - 2 V + 2 G \ast \rho_\infty \right), \quad Z_{\rho_\infty} := \int_{\mathbb{R}^d} \exp\left( - 2 V + 2 G \ast \rho_\infty \right),
\end{equation}
with 
\begin{equation*}
    G(x) := \int_0^{+ \infty} e^{- \gamma \alpha s} g_s \ast \nabla g (x) \diff s.
\end{equation*}
The quantity $Z_{\rho_\infty}$ is well defined because $\rho_\infty$ is a probability density so that
\begin{equation*}
    Z_{\rho_\infty} \leq e^{2 \| G \|_\infty} \int_{\mathbb{R}^d} e^{- 2 V} < \infty.
\end{equation*}

Let us prove that such fixed point indeed exists. We define the function
\begin{equation*}
    \Phi : \mu \mapsto \frac{1}{Z_\mu} \exp\left( - 2 V + 2 G \ast \mu \right)
\end{equation*}
on the non empty convex set
\begin{equation*}
    K := \left\{ \mu \in \mathcal{P}_1\left( \mathbb{R}^d \right), \, \mu(\mathrm{d} x) = \rho(x) \diff x, \quad 0 \leq \rho(x) \leq \frac{e^{4 \| G \|_\infty}}{\int_{\mathbb{R}^d} e^{- 2 V}} e^{- 2 V(x)} \right\}.
\end{equation*} 
The set $K$ is uniformly tight because for all $A \subset \mathbb{R}^d$ compact, and $\nu \in K$, $\nu(A^c) \leq \int_{A^c} \frac{e^{2 \| G \|_\infty}}{\int_{\mathbb{R}^d} e^{- 2 V}} e^{- 2 V(x)} \diff x$, which can be made arbitrarily small uniformly in $\nu$. Therefore $K$ is compact for the weak topology. 

We can also directly check that $\Phi\left( K \right) \subset K$ because for all $\mu \in K$
\begin{equation*}
    0 \leq \Phi(\mu)(x) \leq \frac{\exp\left( - 2 V(x) + 2 G \ast \mu(x) \right)}{\int_{\mathbb{R}^d} \exp\left( - 2 V(y) + 2 G \ast \mu(y) \right) \diff y} \leq e^{4 \| G \|_\infty} \frac{\exp\left( - 2 V(x) \right)}{\int_{\mathbb{R}^d} \exp\left( - 2 V(y) \right) \diff y}.
\end{equation*}

Moreover, let us prove that $\Phi$ is continuous. Let us take $(\mu_n)_n$ a weakly convergent sequence in $K$. For each $x \in \mathbb{R}^d$, we can easily prove that $y \mapsto G(x - y) \in \mathcal{C}_b\left( \mathbb{R}^d \right)$, so that $G \ast \mu_n (x) \xrightarrow[n \to + \infty]{} G \ast \mu (x)$. Plus, we have the uniform in $n$ bound
\begin{equation*}
    \left| \exp\left( - 2 V(x) + 2 G \ast \mu_n(x) \right) \right| \leq e^{2 \| G \|_\infty} e^{- 2 V(x)},
\end{equation*}
so that by dominated convergence we have $Z_{\mu_n} \xrightarrow[n \to + \infty]{} Z_\mu$. Similarly, for all test function $\phi \in \mathcal{C}_b\left( \mathbb{R}^d \right)$, we have
\begin{equation*}
    \int_{\mathbb{R}^d} \phi(x) \exp\left( - 2 V(x) + 2 G \ast \mu_n(x) \right) \diff x \xrightarrow[n \to + \infty]{} \int_{\mathbb{R}^d} \phi(x) \exp\left( - 2 V(x) + 2 G \ast \mu(x) \right) \diff x.
\end{equation*}
Therefore $\Phi(\mu_n) \xrightarrow[n \to + \infty]{} \Phi(\mu)$ weakly.

Let us prove now that $\rho_\infty$ actually is the equilibrium. By application of Schauder theorem, there exists a fixed point, so that $\rho_\infty$ defined by~\eqref{rho_infty} is well defined and is a candidate for the equilibrium.

By application of Proposition in the Markovian case, by taking $u_t = \delta_t$ and $L(t, s, x, \mu) = b(x, \mu)$, and by taking an optimal coupling for $\overline{X}_0$ and $\overline{Y}_0$, we get that
\begin{equation*}
    W_p^p\left( Law\left( \bar{X}_t \right), Law\left( \bar{Y}_t \right) \right) \leq m(t) W_p^p\left( Law\left( \bar{X}_0 \right), Law\left( \bar{Y}_0 \right) \right).
\end{equation*}
Therefore, for $t$ large enough, this means the function $\Phi_t$ that maps an initial law $\mu_0$ to $\mu_t$ the law of the solution to~\eqref{fixed} is a contraction for $W_p$. Therefore it has a unique fixed point that we will call $m_t$.
Since we are in the Markovian case, $\Phi$ is a semigroup, therefore for every $s \geq 0$
\begin{equation*}
    \Phi_s(m_t) = \Phi_s\left( \Phi_t(m_t) \right) = \Phi_{t + s}(m_t) = \Phi_t\left( \Phi_s(m_t) \right).
\end{equation*}
Therefore $\Phi_s(m_t)$ is a fixed point of $\Phi_t$ which has a unique fixed point, so $\Phi_s(m_t) = m_t$. By uniqueness of the invariant measure of~\eqref{fixed} and characterization of Proposition~\ref{coro}, then $\rho_\infty$ is indeed the density of the unique equilibrium.

\subsection{System of particles under a spike of stress}
\label{stressexample}
We now present an example inspired by time-dependent memory kernels that appear in models of aging viscoelastic solids (see~\cite{conti2016modelviscoelasticitytimedependentmemory}). Note that this example is not an application of their work, since the equations are not the same. This example illustrates Assumption~\ref{contraction}, by showing that large bursts of memory do not prevent uniform-in-time results, as long as the asymptotic memory weight is small enough.

Consider
\begin{equation*}
\diff X_t^{i,N} = \sigma \diff B_t^i - \nabla V(X_t^{i,N}) \diff t + \left[ \int_0^t \chi(t)e^{-\lambda(t-s)} \frac{1}{N} \sum_{j=1}^N F(X_t^{i,N}-X_s^{j,N}) \diff s \right] \diff t,
\end{equation*}
where $F:\mathbb R^d\to\mathbb R^d$ is globally Lipschitz, with Lipschitz constant $L_F$, and where
\begin{equation*}
\chi(t) = \chi_\infty + \chi_1 e^{-\beta t}, \qquad \chi_\infty, \, \chi_1, \, \text{and } \beta > 0.
\end{equation*}
This type of temporal modulation $\chi$ depending on the running time $t$ models a transitional regime with a loss of memory asymptotically in time. It can model a system of particles subject to an environmental stress that makes them more reactive to the past of the system, and then relaxes to an asymptotic regime of interaction with the past.
Let us take $u_t(ds) = \mathbf \un_{[0,t]}(s) \diff s$ and
\begin{equation*}
L(t, s, x, \mu) = \chi(t) e^{-\lambda(t-s)} F \ast \mu (x),
\end{equation*}
with $\lambda > 0$.
Then Assumption~\ref{assumparticles} is satisfied with
\begin{equation*}
    h^1(t, s) = h^2(t, s) = L_F \,\chi(t) e^{-\lambda(t-s)},
\end{equation*}
so that
\begin{equation*}
    H_t^1 = H_t^2 = \int_0^t h^1(t,s) \diff s = \frac{L_F}{\lambda} \chi(t) (1 - e^{-\lambda t}).
\end{equation*}
Therefore
\begin{equation*}
    \limsup_{t\to+\infty} H_t^1 = \limsup_{t\to+\infty} H_t^2 = \frac{L_F}{\lambda}\chi_\infty.
\end{equation*}
With $L_F = \chi_\infty = \chi_1 = \beta = 1$ and $\lambda = 3$, we have
\begin{equation*}
    H_{0.5}^1 \simeq 0.416 > \frac{2}{3} = \limsup_{t \geq 0} H_t^1.
\end{equation*}
So the condition Assumption~\ref{contraction} is strictly less demanding than having to assume 
\begin{equation*}
    \sup_{t \geq 0} \left\{ H_t^1 + H_t^2 \right\} < \kappa.
\end{equation*}
This motivates the choice of the more general contraction condition Assumption~\ref{contraction} compared to the condition $\sup_t \{ H_t^1 + H_t^2 \} < \kappa$. 
\begin{remark}
    In short, it means that bursts of memory are acceptable if the long-run memory weight is weak enough. Uniform-in-time estimates and long-time convergence results are more influenced by the asymptotic weight of the memory than by the supremum memory weight.
\end{remark}

\subsection{First impression model}
\label{seccounterexample}
Let us comment on Assumption~\ref{memoryloss}. Consider a case where 
    \begin{equation*}
        u_t(\mathrm{d} s) = \delta_0(\mathrm{d} s), \quad L(t, s, x, \mu) := \int y \diff \mu(y).
    \end{equation*}
    Assumption~\ref{assumparticles} is satisfied with $p = 1$, $h^1 \equiv 0$ and $h^2 \equiv 1$. Therefore we have for all $s > 0$ and $t \geq 0$
    \begin{equation*}
        \frac{\int_0^s h^2(t, v) \, u_t(\mathrm{d} v)}{\int_0^t h^2(t, w) \, u_t(\mathrm{d} w)} = \frac{h^2(t, 0)}{h^2(t, 0)} = 1,
    \end{equation*}
    and Assumption~\ref{memoryloss} is exactly not satisfied. Indeed, we have quite the opposite: all the memory weight is concentrated at the initial time. Consider then the particle system
    \begin{equation}
    \label{eq:XtiN}
        \diff X_t^{i, N} = -\kappa X_t^{i, N} \diff t + \frac{1}{N} \sum_{j = 1}^N X_0^{j, N} \diff t + \diff B_t^i, \quad i = 1, \dots, N,
    \end{equation}
    the initial conditions $X_0^{i, N}$ i.i.d., distributed according to the law $\mu_0 \in \mathcal{P}_2\left( \mathbb{R}^d \right)$ with a non-zero variance. The contraction condition of Assumption~\ref{contraction} is satisfied as soon as $\kappa>1$, so that uniform-in-time propagation of chaos is achieved, by Theorem~\ref{unifpropachaos}.
    Let us consider as well the limit system
    \begin{equation}
    \label{eq:XbarGauss}
        \diff \bar{X}_t^{i, N} =  -\kappa \bar{X}_t^{i, N} \diff t + \e[\bar{X}_0^{1, N}] \diff t + \diff B_t^i, \quad i = 1, \dots, N,
    \end{equation}
    driven by the exact same Brownian motions as the particles, and with $\left( \bar{X}_0^{1, N}, \dots, \bar{X}_0^{N, N} \right) \sim \mu_0^{\otimes N}$ independent from $(X_0^{i, N})_i$. 
    \begin{remark}
        The equation~\eqref{eq:XbarGauss} already shows that, for any initial distribution, $Law\left( \bar{X}_t^{i, N} \right)$ converges, when $t \to + \infty$, to a Gaussian law of mean $\e\left( \bar{X}_0 \right)/\kappa$. It makes the existence of an equilibrium fail, since we defined it to be independent of the initial distribution. This also shows that the contraction of Proposition~\ref{meanfieldlongtime} fails.
    \end{remark}

    Let us prove that in this case, there is no creation of chaos, in the sens introduced in Remark~\ref{rem:creationofchaos}.  If we take 
    \begin{equation*}
        (X_0^{1, N}, \dots, X_0^{N, N}) \sim \frac{1}{2} \delta_{(-1, \dots, -1)} + \frac{1}{2} \delta_{(1, \dots, 1)},
    \end{equation*}
    and $\bar{X}_0^i$ i.i.d. of law $\frac{1}{2} \delta_{- 1} + \frac{1}{2} \delta_1$, we do not even have
    \begin{equation*}
        \e\left( \left| \frac{1}{N} \sum_{i = 1}^N X_0^{i, N} - \e[\bar{X}_0^{1, N}] \right|^2 \right) \xrightarrow[N \to + \infty]{} 0,
    \end{equation*}
    so \textit{a fortiori} we do not have
    \begin{equation*}
        \frac{1}{N} \sum_{i = 1}^N \e\left( \left| X_0^{i, N} - \e[\bar{X}_0^{1, N}] \right|^2 \right) \xrightarrow[N \to + \infty]{} 0,
    \end{equation*}
    by Jensen inequality. In this case, there is no creation of chaos, since, for all $i = 1$ to $N$, conditionally to the initial condition, we have
    \begin{equation*}
        X_t^{i, N} = \left[ 1 + \frac{1 - e^{- \kappa t}}{\kappa} \right] X_0^{i, N} + \int_0^t e^{- \kappa (t - s)} \diff B_t^i,
    \end{equation*}
    with $\int_0^t e^{- \kappa (t - s)} \diff B_t^i$ being independent Brownian motions, so that
    \begin{equation*}
        (X_t^{1, N}, \dots, X_t^N) \sim \frac{1}{2} \mathcal{N} \left( m_t^-, \sigma_t \right)^{\otimes N} + \frac{1}{2} \mathcal{N} \left( m_t^+, \sigma_t \right)^{\otimes N},
    \end{equation*}
    with \begin{equation*}
        m_t^+ := 1 + \frac{1 - e^{- \kappa t}}{\kappa}, \quad m_t^- := - m_t^+, \quad \text{and } \sigma_t := \frac{1 - e^{- \kappa t}}{2 \kappa}.
    \end{equation*}
    Therefore we can compute the correlation function of $X_\cdot^{1, N}$ and $X_\cdot^{1, N}$
    \begin{equation*}
        \e\left[ X_t^{1, N} X_t^{2, N} \right] = \frac{1}{2} (m_t^+)^2 + \frac{1}{2} (m_t^-)^2 = \left( 1 + \frac{1 - e^{- \kappa t}}{\kappa} \right)^2 \xrightarrow[t \to + \infty]{} 1.
    \end{equation*}
    Therefore there is no creation of chaos.

\section{Propagation of chaos}
\label{secproofPoC}

Let us now prove that the propagation of chaos is achieved by adapting the calculations of~\cite{MALRIEU2001109}. We will first prove Theorem~\ref{bigtheorem}, and then use it to prove the propagation of chaos on a finite time horizon.

\subsection{Proof of Theorem~\ref{bigtheorem}}
Let us follow the approach of~\cite{theseMilica}, Chapter $2$, section $2.2$, with particles subject to an additional force $F$. Here, the adapted distance is not $W_2$ or the coarser (in expectation) distance $ \frac{1}{N} \sum_{i=1}^N |X_t^i - \bar{X}_t^i|^2$, but $\frac{1}{N} \sum_{i=1}^N |X_t^i - \bar{X}_t^i|^p$, to be compared with $W_p$, since $W_2$ is too coarse here to enable us to work properly with the $W_p$-Lipschitz memory kernel $L$. Let us now show that the propagation of chaos is achieved by adapting the calculations of~\cite{MALRIEU2001109}.
\begin{proof}
By differentiation we have
\begin{align}
\label{split}
\nonumber
    \frac{\diff}{\diff t} &\left[ \frac{1}{N} \sum_{i = 1}^N \frac{
    |X_t^i - \bar{X}_t^i|^p}{p}\right] 
    = - \frac{1}{N} \sum_{i = 1}^N |X_t^i - \bar{X}_t^i|^{p-2} \left( X_t^i - \bar{X}_t^i \right) \cdot \left( F(X_t^i) - F(\bar{X}_t^i) \right) \\
    \nonumber
    &+ \frac{1}{N} \sum_{i = 1}^N |X_t^i - \bar{X}_t^i|^{p-2} \left( X_t^i - \bar{X}_t^i \right) \cdot \left[ \int_0^t L\Huge(t, s, \bar{X}_t^i, \frac{1}{N} \sum_{j=1}^N \delta_{\bar{X}_s^j} \Huge) u_t(\mathrm{d} s) - \int_0^t L\Huge(t, s, \bar{X}_t^i, \mu_s \Huge) u_t(\mathrm{d} s) \right] \\
    \nonumber
    &+ \frac{1}{N} \sum_{i = 1}^N |X_t^i - \bar{X}_t^i|^{p-2} \left( X_t^i - \bar{X}_t^i \right) \cdot \left[ \int_0^t L\Huge(t, s, X_t^i, \frac{1}{N} \sum_{j = 1}^N \delta_{X_s^j}\Huge) u_t(\mathrm{d} s) - \int_0^t L\Huge(t, s, \bar{X}_t^i, \frac{1}{N} \sum_{j=1}^N \delta_{\bar{X}_s^j} \Huge) u_t(\mathrm{d} s) \right] \\
    &=: A_1 + A_2 + A_3.
\end{align}
The differentiation works for $p > 1$ because by parallel coupling $t \mapsto X_t^i - \bar{X}_t^i$ satisfies an ODE because the Brownian motions cancel out. The case $p = 1$ can be treated by a regularization trick we will not perform here. Refer to~\cite{FournierJourdain} for an example of this trick.

From now on the proof will be separated in five steps. The three first steps \textbf{Step 1.}, \textbf{Step 2.} and \textbf{Step~3.} will focus on each terms of~\eqref{split}. Then, \textbf{Step 4.} will put all the previous calculations together to get a Gronwall-type inequality. Finally, \textbf{Step 5.} will deduce the desired estimate.

\paragraph{Step 1. Let us start by treating the first term of~\eqref{split}.}
The first term can be treated easily thanks to the hypothesis on $F$. We get the following upper bound \begin{equation}
\label{A1}
    A_1 \leq - p \kappa \frac{1}{N} \sum_{i = 1}^N \frac{|X_t^i - \bar{X}_t^i|^p}{p}.
\end{equation}

\paragraph{Step 2. Let us treat the second term of~\eqref{split}.}
The second term can be treated in a similar manner to in~\cite{theseMilica}, by using successively Cauchy-Schwarz inequality and then the Lipschitz hypothesis~\eqref{memory} on the memory kernel $L$. We have
    \begin{align*}
        A_2 &=\frac{1}{N} \sum_{i = 1}^N |X_t^i - \bar{X}_t^i|^{p-2} \left( X_t^i - \bar{X}_t^i \right) \cdot \left[ \int_0^t L\Huge(t, s, \bar{X}_t^i, \frac{1}{N} \sum_{j=1}^N \delta_{\bar{X}_s^j} \Huge) u_t(\mathrm{d} s) - \int_0^t L\Huge(t, s, \bar{X}_t^i, \mu_s \Huge) u_t(\mathrm{d} s) \right] \\
        &\leq \frac{1}{N} \sum_{i=1}^N |X_t^i - \bar{X}_t^i|^{p-1} \cdot \left| \int_0^t L\Huge(t, s, \bar{X}_t^i, \frac{1}{N} \sum_{j=1}^N \delta_{\bar{X}_s^j} \Huge) u_t(\mathrm{d} s) - \int_0^t L\Huge(t, s, \bar{X}_t^i, \mu_s \Huge) u_t(\mathrm{d} s) \right| \\
        &\leq \frac{1}{N} \sum_{i=1}^N \int_0^t h^2(t, s) \, |X_t^i - \bar{X}_t^i|^{p-1} \, W_p\left(\frac{1}{N} \sum_{j=1}^N \delta_{\bar{X}_s^j}, \mu_s \right) \, u_t(\mathrm{d} s).
    \end{align*}
    Let us take the expectation and apply Young inequality with an arbitrary $\varepsilon > 0$ for $\frac{p - 1}{p} + \frac{1}{p} = 1$. Note that here, the introduction of the parameter $\varepsilon$ is overdone, since it is not needed to play with this parameter to conclude. For here, think of $\varepsilon$ as $1$, but these calculations will be helpful for the uniform-in-time proof of Theorem~\ref{unifpropachaos}. By taking the supremum in the second inequality, we get
    \begin{align}
    \label{temp}
    \nonumber
        \e(A_2) &\leq \varepsilon \frac{p - 1}{p} H_t^2 \cdot \frac{1}{N} \sum_{i=1}^N \e|X_t^i - \bar{X}_t^i|^p +  \frac{1}{p \varepsilon^{p - 1}} \int_0^t h^2(t, s) \, \e\left[ W_p^p\left( \frac{1}{N} \sum_{i=1}^N \delta_{\bar{X}_s^i} , \mu_s \right) \right] \, u_t(\mathrm{d} s) \\
        &\leq \varepsilon \frac{p - 1}{p} H_t^2 \cdot \frac{1}{N} \sum_{i=1}^N \e|X_t^i - \bar{X}_t^i|^p + \frac{H_t^2}{p \varepsilon^{p - 1}} \sup_{0 \leq s \leq T} \e\left[ W_p^p\left(\frac{1}{N} \sum_{i=1}^N \delta_{\bar{X}_s^i}, \mu_s \right) \right],
    \end{align} 
    Let us treat the term $\e\left[ W_p^p\left(\frac{1}{N} \sum_{i=1}^N \delta_{\bar{X}_s^i}, \mu_s \right) \right]$. The article~\cite{FournierGuillin} yields
    \begin{equation*}
        \e\left[ W_p^p\left( \frac{1}{N} \sum_{i=1}^N \delta_{\bar{X}_s^i}, \mu_s \right) \right] \leq v(N) \, \e(|\bar{X}_s|^q)^{p/q},
    \end{equation*}
    with $C$ a constant provided by the article depending only on $p$, $d$, and $q$, and $v(N)$ a speed function that decreases to $0$
    \begin{align*}
    v(N) := C &\times\left((N^{-1/2} + N^{-(q-p)/q}) \un_{p > d/2, q \neq 2 p} + (N^{-1/2} \log(1 + N) \right. \\
    &\left.+ N^{-(q-p)/q}) \un_{p = d/2, q \neq 2p} + (N^{-p/d} + N^{-(q-p)/q}) \un_{0< p < d/2, q \neq d/(d-p)}\right).
    \end{align*}
    Let us go back to~\eqref{temp}. We have
\begin{equation}
\label{A2}
    \e(A_2) \leq \varepsilon (p - 1) H_t^2 \; \frac{1}{N} \sum_{i=1}^N \frac{\e (|X_t^i - \bar{X}_t^i|^p)}{p} + v(N) \cdot \frac{H_t^2}{p \varepsilon^{p - 1}} \sup_{0 \leq s \leq T} \e(|\bar{X}_s^1|^q)^{p/q}.
\end{equation} 
Note that 
\begin{equation*}
    \sup_{0 \leq s \leq T} \e(|\bar{X}_s^1|^q)^{p/q} < \infty,
\end{equation*}
thanks to Proposition~\ref{wellposedness} (see the remark at the end of the proof).

\paragraph{Step 3. Let us treat the third term of~\eqref{split}.}    
    For the last term, let us use successively Cauchy-Schwarz inequality along with triangular inequality, and the hypotheses on $L$ in Assumption~\ref{assumparticles}
    \begin{align*}
        A_3 =\frac{1}{N} & \sum_{i=1}^N |X_t^i - \bar{X}_t^i|^{p-2} \left( X_t^i - \bar{X}_t^i \right) \cdot \left[ \int_0^t L\Huge(t, s, X_t^i, \frac{1}{N} \sum_{j = 1}^N \delta_{X_s^j}\Huge) u_t(\mathrm{d} s) - \int_0^t L\Huge(t, s, \bar{X}_t^i, \frac{1}{N} \sum_{j=1}^N \delta_{\bar{X}_s^j} \Huge) u_t(\mathrm{d} s) \right] \\
        \leq& \frac{1}{N} \sum_{i=1}^N |X_t^i - \bar{X}_t^i|^{p-1}\times \left( \left| \int_0^t L\Huge(t, s, X_t^i, \frac{1}{N} \sum_{j = 1}^N \delta_{X_s^j}\Huge) u_t(\mathrm{d} s) - \int_0^t L\Huge(t, s, \bar{X}_t^i, \frac{1}{N} \sum_{j = 1}^N \delta_{X_s^j}\Huge) u_t(\mathrm{d} s) \right| \right. \\
        &\left. + \left| \int_0^t L\Huge(t, s, \bar{X}_t^i, \frac{1}{N} \sum_{j = 1}^N \delta_{X_s^j}\Huge) u_t(\mathrm{d} s) - \int_0^t L\Huge(t, s, \bar{X}_t^i, \frac{1}{N} \sum_{j = 1}^N \delta_{\bar{X}_s^j}\Huge) u_t(\mathrm{d} s) \right| \right) \\
        \leq& \frac{1}{N} \sum_{i=1}^N |X_t^i - \bar{X}_t^i|^{p-1} \times \left( \int_0^t h^1(t, s) |X_t^i - \bar{X}_t^i| u_t(\mathrm{d} s) + \int_0^t h^2(t, s) W_p\left(\frac{1}{N} \sum_{j = 1}^N \delta_{X_s^j} \, , \, \frac{1}{N} \sum_{j = 1}^N \delta_{\bar{X}_s^j}\right) u_t(\mathrm{d} s) \right) \\
        \leq& \frac{1}{N} \sum_{i=1}^N \left(H_t^1 \, |X_t^i - \bar{X}_t^i|^p + \int_0^t h^2(t, s) W_p\left(\frac{1}{N} \sum_{j = 1}^N \delta_{X_s^j} \, , \, \frac{1}{N} \sum_{j = 1}^N \delta_{\bar{X}_s^j}\right) u_t(\mathrm{d} s) \cdot |X_t^i - \bar{X}_t^i|^{p-1} \right).
    \end{align*}
    Let us take the expectation. Young inequality with $\frac{p-1}{p} + \frac{1}{p} = 1$ and a parameter $\Tilde{\varepsilon}$ yields
    \begin{align}
    \label{A3}
    \nonumber
        \e(A_3) \leq& \frac{1}{N} \sum_{i=1}^N H_t^1 \e(|X_t^i - \bar{X}_t^i|^p) + \int_0^t h^2(t, s) \frac{1}{p \Tilde{\varepsilon}^{p - 1}} \e\left[W_p^p\left(\frac{1}{N} \sum_{j = 1}^N \delta_{X_s^j} \, , \, \frac{1}{N} \sum_{j = 1}^N \delta_{\bar{X}_s^j}\right)\right] \, u_t(\mathrm{d} s) \\
    \nonumber
        &+ \frac{(p-1)\Tilde{\varepsilon}}{p} \int_0^t h^2(t, s) \frac{1}{N} \sum_{i = 1}^N \e(|X_t^i - \bar{X}_t^i|^p) \, u_t(\mathrm{d} s) \\
    \nonumber
        \leq& \left( p H_t^1 + (p-1) \Tilde{\varepsilon} H_t^2 \right) \frac{1}{N} \sum_{i=1}^N \frac{\e(|X_t^i - \bar{X}_t^i|^p)}{p} \\
        \nonumber
        &+ \frac{1}{p \Tilde{\varepsilon}^{p - 1}} \int_0^t h^2(t, s) \e \left[ W_p^p\left(\frac{1}{N} \sum_{j = 1}^N \delta_{X_s^j} \, , \, \frac{1}{N} \sum_{j = 1}^N \delta_{\bar{X}_s^j}\right) \right] u_t(\mathrm{d} s) \\
        \leq& \left( p H_t^1 + (p-1) \Tilde{\varepsilon} H_t^2 \right) \frac{1}{N} \sum_{i=1}^N \frac{\e(|X_t^i - \bar{X}_t^i|^p)}{p} + \frac{1}{\Tilde{\varepsilon}^{p - 1}} \int_0^t h^2(t, s) \frac{1}{N} \sum_{i=1}^N \frac{\e(|X_s^i - \bar{X}_s^i|^p)}{p} \, u_t(\mathrm{d} s),
    \end{align} because $\e \left[ W_p^p\left(\frac{1}{N} \sum_{j = 1}^N \delta_{X_s^j} \, , \, \frac{1}{N} \sum_{j = 1}^N \delta_{\bar{X}_s^j}\right) \right] \leq \frac{1}{N} \sum_{i=1}^N \e(|X_s^i - \bar{X}_s^i|^p).$ Again, the parameter $\Tilde{\varepsilon}$ is useful only for the later proof of Theorem~\ref{unifpropachaos}, but for now, one can think of it as $\Tilde{\varepsilon} = 1$.

\paragraph{Step 4. Let us go back to~\eqref{split} and deduce a Gronwall-type estimate.}
Assemble~\eqref{A1},~\eqref{A2}, and~\eqref{A3}, and then rearrange them.
For all $0 \leq t \leq T$
\begin{align}
\label{estimatepropa}
\nonumber
    \frac{\diff}{\diff t} \left[ \frac{1}{N} \sum_{i = 1}^N \frac{
        \e(|X_t^i - \bar{X}_t^i|^p)}{p}\right] \leq& - p \kappa \frac{1}{N} \sum_{i = 1}^N \frac{\e(|X_t^i - \bar{X}_t^i|^p)}{p} \\
        \nonumber
        &+ \varepsilon (p - 1) H_t^2 \frac{1}{N} \sum_{i=1}^N \frac{\e (|X_t^i - \bar{X}_t^i|^p)}{p} + v(N) \frac{H_t^2}{p \varepsilon^{p - 1}} \sup_{0 \leq s \leq T} \e(|\bar{X}_s^1|^q)^{p/q} \\
        \nonumber
        &+ \left( p H_t^1 + (p-1) \Tilde{\varepsilon} H_t^2 \right) \frac{1}{N} \sum_{i=1}^N \frac{\e(|X_t^i - \bar{X}_t^i|^p)}{p} \\
        \nonumber
        &+ \frac{1}{\Tilde{\varepsilon}^{p - 1}} \int_0^t h^2(t, s) \frac{1}{N} \sum_{i=1}^N \frac{\e(|X_s^i - \bar{X}_s^i|^p)}{p} \, u_t(\mathrm{d} s) \\
        \nonumber
        \leq& \left( p (H_t^1 - \kappa) + H_t^2 (\Tilde{\varepsilon} + \varepsilon) (p-1) \right) \frac{1}{N} \sum_{i=1}^N \frac{\e(|X_t^i - \bar{X}_t^i|^p)}{p} \\
        \nonumber
        &+ \frac{1}{\Tilde{\varepsilon}^{p - 1}} \int_0^t h^2(t, s) \frac{1}{N} \sum_{i = 1}^N \frac{\e \left( |X_s^i -\bar{X}_s^i|^p\right)}{p} u_t(\mathrm{d} s) \\
        &+ v(N) \cdot \frac{H_t^2}{p \varepsilon^{p - 1}} \sup_{0 \leq s \leq T} \e(|\bar{X}_s^1|^q)^{p/q}.
\end{align}
Let us denote
\begin{equation}
\label{parameters}
    \alpha(t) := p (H_t^1 - \kappa) + H_t^2 (\varepsilon + \Tilde{\varepsilon}) (p - 1), \quad \beta(t, s) := \frac{1}{\Tilde{\varepsilon}^{p - 1}} h^2(t, s), \quad \text{and} \quad C(t) := \frac{H_t^2}{p \varepsilon^{p - 1}} \sup_{0 \leq t \leq T} \e(|\bar{X}_t^1|^q)^{p/q},
\end{equation}
and
\begin{equation*}
    e_t := \frac{1}{N} \sum_{i=1}^N \frac{\e|X_t^i - \bar{X}_t^i|^p}{p}.
\end{equation*}
Then we have \begin{equation*}
\left\{
\begin{array}{rcl}
    e_t' & \leq & \alpha(t) e_t + \displaystyle \int_0^t \beta(t, s) e_s \; u_t(\mathrm{d} s) + v(N) C(t) \\
    e_0 & = & \displaystyle \frac{1}{N} \sum_{i = 1}^N \frac{\e|X_0^i - \bar{X}_0^i|^p}{p}.
\end{array}
\right.
\end{equation*}

\paragraph{Step 5. Deduce an estimate for $e_t$.}
Consider $x_\eta : \mathbb{R}_+ \to \mathbb{R}_+$ solution to the system
\begin{equation*}
    \left\{ \begin{array}{rcl}
        x_\eta'(t) &=& \alpha(t) x_\eta(t) + \displaystyle \int_0^t \beta(t, s) x_\eta(s) \, u_t(\mathrm{d} s) + v(N) C(t), \quad \text{for all $0 \leq t \leq T$}\\
        x_\eta(0) &=& \displaystyle \frac{1}{N} \sum_{i=1}^N \frac{\e|X_0^i - \bar{X}_0^i|^p}{p} + \eta,
    \end{array}\right.
\end{equation*}
for all $\eta \geq 0$. The function $x_\eta$ is a candidate for an upper bound for $e_t$.

Let us denote $y_t := x_\eta(t) - e_t$. Then we have 
\begin{equation*}
    y_t' \leq \alpha(t) y_t + \int_0^t \beta(t, s) y_s \ u_t(\mathrm{d} s),
\end{equation*}
with $y_0 = \eta$. Therefore, by denoting $z_t := \exp\left( - \int_0^t \alpha(s) \diff s \right) y_t$ and $\Tilde{\beta}(t, s) = \frac{\exp\left( - \int_0^t \alpha(r) \diff r \right)}{\exp\left( - \int_0^s \alpha(r) \diff r \right)} \beta(t, s)$, we get
\begin{equation*}
    z_t' \geq \int_0^t \Tilde{\beta}(t, s) z_s \ u_t(\mathrm{d} s).
\end{equation*}
Let us remark that, for all $0 \leq s \leq T$, we have
\begin{equation*}
    \int_0^s \Tilde{\beta}(s, r) \ u_s(\mathrm{d} r) \leq \frac{1}{\Tilde{\varepsilon}^{p - 1}} \frac{\sup_{0 \leq t \leq T} \exp\left( - \int_0^t \alpha(r) \diff r \right)}{\inf_{0 \leq t \leq T} \exp\left( - \int_0^t \alpha(r) \diff r \right)} D_T =: \gamma,
\end{equation*}
so that, by integrating, we have, for all $0 \leq t \leq T$
\begin{align*}
    z_t &\geq \eta + \int_0^t \int_0^s \Tilde{\beta}(s, r) z_r \ u_s(\mathrm{d} r) \diff s \\
    &\geq \eta + \gamma T \min\left\{ 0, \inf_{0 \leq t \leq T} z_t \right\}.
\end{align*}
If $\min\left\{ 0, \inf_{0 \leq t \leq T} z_t \right\} = 0$, then we can already conclude that $e_t \leq x_\eta(t)$ for all $0 \leq t \leq T$. Suppose this is not the case, and denote $m := \inf_{0 \leq t \leq T} z_t$. Then 
\begin{equation*}
    m \geq \eta + \gamma T m. 
\end{equation*}
By taking $T$ small enough such that $1 - \gamma T > 0$, we get a contradiction.

To extend this local result to $[0, T]$ with fixed $T > 0$, we can use a subdivision of $[0, T]$ into sub intervals $[T_k, T_{k + 1}[$ of length strictly smaller than $1/\gamma$.
Suppose that $z_t \geq 0$ on $[0, T_k[$. Then for all $0 \leq l \leq k$, for all $T_l \leq t \leq T_{l + 1}$
\begin{align*}
    z_t &\geq \eta + \int_0^{T_l} \int_0^s \Tilde{\beta}(s, r) z_r \ u_s(\mathrm{d} r) \diff s + \int_{T_l}^t \int_0^s \Tilde{\beta}(s, r) z_r \ u_s(\mathrm{d} r) \diff s \\
    &\geq \eta + \inf_{0 \leq t \leq T_{k + 1}} z_t \ \gamma \left( T_{l + 1} - T_l \right),
\end{align*}
so that 
\begin{equation*}
    \inf_{0 \leq t \leq T_{k + 1}} z_t \geq \eta + \inf_{0 \leq t \leq T_{k + 1}} z_t \ \gamma \min_{0 \leq l \leq k} \left( T_{l + 1} - T_l \right).
\end{equation*}
Contradiction. We conclude that for all $0 \leq t \leq T$
\begin{equation*}
    e_t \leq x_\eta(t).
\end{equation*}
\bigbreak
Since this is true for all $\eta > 0$ arbitrarily small, we can make $\eta$ tend to $0$.

Let us prove that for all $t \geq 0$, $x_\eta(t) \xrightarrow[\eta \to 0]{} x_0(t)$. This follows from a stability argument. For all $\eta$, $\nu$, we have
\begin{align*}
    |x_\eta(t) - x_\nu(t)| &= \left| x_\eta(0) - x_\nu(0) + \alpha(t) \int_0^t (x_\eta(s) - x_\nu(s)) \diff s + \int_0^t \int_0^s \beta(s, v) (x_\eta(v) - x_\nu(v)) \; u_s(\mathrm{d} v) \diff s \right| \\
    &\leq |\eta - \nu| + \int_0^t \left( |\alpha(t)| + \int_0^s \beta(s, v) \; u_s(\mathrm{d} v) \right) \sup_{0 \leq v \leq s} |x_\eta(v) - x_\nu(v)| \diff s.
\end{align*}
By taking the supremum in the left hand side and then applying Gronwall, it gives 
\begin{equation*}
    \sup_{0 \leq t \leq T} |x_\eta(t) - x_\nu(t)| \leq |\eta - \nu| \exp\left( \int_0^T \left( |\alpha(t)| + \int_0^t \beta(t, s) \; u_t(\mathrm{d} s) \right) \diff t \right).
\end{equation*}
\textit{A fortiori} we have the simple convergence of $x_\eta$ to $x_0$, so that for all $0 \leq t \leq T$
\begin{equation*}
    e_t \leq x_0(t).
\end{equation*}
\end{proof}

\subsection{Proof of the propagation of chaos on a finite time horizon (Theorem~\ref{propachaos})}
For this proof, we won't need to use the variability of the parameters of $\varepsilon$ and $\Tilde{\varepsilon}$, so let us start from the study of $x : \mathbb{R}_+ \to \mathbb{R}_+$ defined in~\eqref{upperbound}, where we take $\varepsilon = 1$ and $\Tilde{\varepsilon} = 1$. The goal is then to use~\eqref{Gronwall} to get the desired result~\eqref{eqPoC}. For all $0 \leq s \leq t \leq \Tilde{T} \leq T$, we have
\begin{align*}
    x'(s) =& \left( p (H_s^1 - \kappa) + 2 H_s^2 (p-1) \right) x(s) + \int_0^s h^2(s, v) x(v) \; u_t(\mathrm{d} v) + \frac{v(N) H_s^2}{p} \sup_{0 \leq v \leq T} \e(|\bar{X}_v^1|^q)^{p/q} \\
    \leq& \left( p (D_T - \kappa)_+ + 2 D_T (p-1) \right) \sup_{0 \leq v \leq s} x(v) + D_T \sup_{0 \leq v \leq s} x(v) + \frac{v(N) D_T}{p} \sup_{0 \leq v \leq T} \e(|\bar{X}_v^1|^q)^{p/q},
\end{align*}
by using the definition of $D_T$ in~\eqref{D_T}.
Integrate from $0$ to $t \in [0, \Tilde{T}]$. We get
\begin{align*}
    x(t) \leq& \frac{1}{N} \sum_{i = 1}^N \frac{\e(|X_0^i - \bar{X}_0^i|^p)}{p} + \left( p (D_T - \kappa)_+ + 2 D_T (p-1) + D_T \right) \int_0^{\Tilde{T}} \sup_{0 \leq v \leq s} x(v) \diff s \\
    &+ v(N) \cdot \frac{T D_T}{p} \sup_{0 \leq v \leq T} \e(|\bar{X}_v^1|^q)^{p/q}.
\end{align*}
Then, if we take the supremum in time again for the left hand side, we get that, for all $\Tilde{T} \in [0, T]$, we have
\begin{align*}
    \sup_{0 \leq t \leq \Tilde{T}} x(t) \leq& \frac{1}{N} \sum_{i = 1}^N \frac{
        \e(|X_0^i - \bar{X}_0^i|^p)}{p} + \left( p (D_T - \kappa)_+ + 2 D_T (p-1) + D_T \right) \int_0^{\Tilde{T}} \sup_{0 \leq s \leq t} x(s) \diff t \\
        &+ v(N) \cdot \frac{T D_T}{p} \sup_{0 \leq t \leq T} \e(|\bar{X}_t^1|^q)^{p/q}.
\end{align*}
We end up with a classical Gronwall formulation, so that Gronwall's lemma yields
\begin{equation*}
        \sup_{0 \leq t \leq T} x(t) \leq C_1(p, \kappa, T) \left( \frac{1}{N} \sum_{i = 1}^N \frac{
        \e(|X_0^i - \bar{X}_0^i|^p)}{p} + v(N) \cdot \frac{T D_T}{p} \sup_{0 \leq t \leq T} \e(|\bar{X}_t^1|^q)^{p/q}\right),
\end{equation*}
where $C_1(p, \kappa, T) := \exp\left( p T (D_T - \kappa)_+ + 2 T D_T (p-1) + T D_T \right)$. Note that Proposition~\ref{wellposedness} yields
\begin{equation*}
    \sup_{0 \leq s \leq T} \e(|\bar{X}_s^1|^q)^{p/q} < \infty.
\end{equation*}
Therefore, we may denote $C_2(p, q, \kappa, T) := \frac{T D_T}{p} \sup_{0 \leq s \leq T} \e(|\bar{X}_s^1|^q)^{p/q}$, and end up with the desired result.

\section{Uniform-in-time results}
\label{secproofunif}
Let us prove the uniform-in-time results. We will start by proving the uniform-in-time boundedness of the moment of the limit equation, and then prove the uniform-in-time propagation of chaos.

\subsection{Proof of the uniform-in-time moment bound (Lemma~\ref{moments})}
The proof will be separated in five steps. In \textbf{Step 1.}, we will derive a moment estimate, where a special term remains to be treated. This is what we will do in \textbf{Step 2.}, leading to a more convenient estimate of the $q$-th moment. In \textbf{Step 3.} we use a variant of Gronwall's lemma that is suitable for the integro-differential estimate that we get in the previous step.

\paragraph{Step 1. Application of Itô formula.} 
Let us take $q$ such that $q \geq 2$ and $q > p$. By application of the Itô formula with the $\mathcal{C}^2$ function $x \mapsto |x|^q$, we get 
\begin{align*}
|\bar{X}_T|^q =& |\bar{X}_0|^q - \int_0^T q |\bar{X}_t|^{q-2} \bar{X}_t \cdot F(\bar{X}_t) \ \diff t + \int_0^T q |\bar{X}_t|^{q-2} \ \bar{X}_t \cdot \int_0^t L(t, s, \bar{X}_t, \mu_s) u_t(\mathrm{d} s) \ \diff t \\
&+ \int_0^T q |\bar{X}_t|^{q-2} \bar{X}_t \cdot \sigma \diff B_t + \frac{1}{2} \sum_{i, j = 1}^d \int_0^T \left[ q(q-2) |\bar{X}_t|^{q-4} \bar{X}_t^i \bar{X}_t^j + q |\bar{X}_t|^{q-2} \delta_{i, j}\right] \diff [\bar{X}_t^i, \bar{X}_t^j].
\end{align*}
Let us take the expectation of the previous equation. Since $T \mapsto \displaystyle \int_0^T q |\bar{X}_t|^{q-2} \ \bar{X}_t \cdot \sigma \diff B_t$ is a martingale thanks to Proposition~\ref{wellposedness}, then the expectation of this term vanishes
\begin{align*}
\e|\bar{X}_T|^q =& \e|\bar{X}_0|^q - \int_0^T q \e\left[ |\bar{X}_t|^{q-2} \bar{X}_t \cdot F(\bar{X}_t) \right] \ \diff t + \int_0^T q \left[ |\bar{X}_t|^{q-2} \ \bar{X}_t \cdot \int_0^t L(t, s, \bar{X}_t, \mu_s) u_t(\mathrm{d} s) \right] \ \diff t \\
&+ \frac{1}{2} \sum_{i, j = 1}^d (\sigma ^t \sigma)_{i, j} \int_0^T \left[ q(q-2) \e\left( |\bar{X}_t|^{q-4} \bar{X}_t^i \bar{X}_t^j \right) + q \e\left( |\bar{X}_t|^{q-2} \right) \delta_{i, j}\right] \diff t.
\end{align*}
Then, let us derive
\begin{align*}
\frac{\diff}{\diff t} \e\left(|\bar{X}_t|^q \right) =& - q \e\left[ |\bar{X}_t|^{q-2} \bar{X}_t \cdot F(\bar{X}_t) \right] + q \e\left[ |\bar{X}_t|^{q-2} \ \bar{X}_t \cdot \int_0^t L(t, s, \bar{X}_t, \mu_s) u_t(\mathrm{d} s) \right] \\
&+ \frac{1}{2} \sum_{i, j = 1}^d (\sigma ^t \sigma)_{i, j} \left[ q(q-2) \e\left(|\bar{X}_t|^{q-4} \bar{X}_t^i \bar{X}_t^j \right) + q \e\left(|\bar{X}_t|^{q-2}\right) \delta_{i, j} \right].
\end{align*}
Then, let us transform this formula in order to use the hypotheses on $F$ and $L$ in Assumption~\ref{assumparticles}
\begin{align*}
\frac{\diff}{\diff t} \e\left(|\bar{X}_t|^q \right) =& - q \e\left[ |\bar{X}_t|^{q-2} \left( \bar{X}_t - 0 \right) \cdot \left( F(\bar{X}_t) -   F(0) \right) \right] - q \e\left[ |\bar{X}_t|^{q-2} \ \bar{X}_t \cdot F(0) \right] \\
&+ q \e\left[ |\bar{X}_t|^{q-2} \ \left( \bar{X}_t - 0 \right) \cdot \int_0^t \left[ L(t, s, \bar{X}_t, \mu_s) - L(t, s, 0, \mu_s)\right] u_t(\mathrm{d} s) \right] \\
&+ q \e\left[ |\bar{X}_t|^{q-2} \ \bar{X}_t \cdot \int_0^t L(t, s, 0, \mu_s) u_t(\mathrm{d} s) \right] \\
&+ \frac{1}{2} \sum_{i, j = 1}^d (\sigma ^t \sigma)_{i, j}\left[ q(q-2) \e\left( |\bar{X}_t|^{q-4} \bar{X}_t^i \bar{X}_t^j \right) + q \left( |\bar{X}_t|^{q-2} \right) \delta_{i, j} \right] \\
\leq& - \kappa q \e\left( |\bar{X}_t|^q \right)  + q \ |F(0)| \ \e\left( |\bar{X}_t|^{q-1} \right) \\
&+ q \e\left[ |\bar{X}_t|^{q-2} |\bar{X}_t| \ \int_0^t h^1(t, s) \, |\bar{X}_t| \; u_t(\mathrm{d} s) \right] + q \e\left[ |\bar{X}_t|^{q-1} \left|\int_0^t L(t, s, 0, \mu_s) u_t(\mathrm{d} s)\right| \right] \\
&+ \frac{1}{2} \sum_{i, j = 1}^d (\sigma ^t \sigma)_{i, j}\left[ q(q-2) \e\left( |\bar{X}_t|^{q-4} \bar{X}_t^i \bar{X}_t^j \right) + q \e\left( |\bar{X}_t|^{q-2} \right) \delta_{i, j}\right],
\end{align*}
where the last inequality uses the Cauchy-Schwarz inequality. Let us use Young inequality with the parameter $\varepsilon > 0$ and with $\frac{1}{q} + \frac{q - 1}{q} = 1$ and the inequality $|\bar{X}_t^i| |\bar{X}_t^j| \leq \sum_{k, l = 1}^d |\bar{X}_t^k| |\bar{X}_t^l|$
\begin{align*}
    \frac{\diff}{\diff t} \e\left(|\bar{X}_t|^q \right) \leq& - \kappa q \e\left( |\bar{X}_t|^q \right)  + (q-1) \varepsilon \e\left(|\bar{X}_t|^q \right) + \frac{1}{\varepsilon^{q - 1}} |F(0)|^q \\
&+ q H_t^1 \e\left( |\bar{X}_t|^q \right) +  q \e\left( |\bar{X}_t|^{q - 1} \right) \left|\int_0^t L(t, s, 0, \mu_s) u_t(\mathrm{d} s)\right| \\
&+ \frac{1}{2} \sum_{i, j = 1}^d (\sigma ^t \sigma)_{i, j} \left[ q(q-2) \e\left(|\bar{X}_t|^{q-4} \sum_{k, l = 1}^d |\bar{X}_t^k| |\bar{X}_t^l|\right) + q \e\left( |\bar{X}_t|^{q-2} \right) \delta_{i, j}\right] \\
\leq& - \kappa q \e\left( |\bar{X}_t|^q \right)  + (q-1)\varepsilon \e\left( |\bar{X}_t|^q \right) + \frac{1}{\varepsilon^{q - 1}} |F(0)|^q \\
&+ q H_t^1 \e\left( |\bar{X}_t|^q \right) +  q \e\left( |\bar{X}_t|^{q - 1} \right) \left|\int_0^t L(t, s, 0, \mu_s) u_t(\mathrm{d} s)\right| \\
&+ \frac{1}{2} \sum_{i, j = 1}^d (\sigma ^t \sigma)_{i, j} \left[ d q (q-2) \e\left( |\bar{X}_t|^{q-2} \right) + q \e\left( |\bar{X}_t|^{q-2} \right) \delta_{i, j}\right],
\end{align*}
where in the last inequality we used the Cauchy-Schwarz inequality $\sum_{k, l \leq d} |\bar{X}_t^k||\bar{X}_t^l| \leq d \cdot \sum_k |\bar{X}_t^k|^2$. Then let use Young inequality twice with $\frac{q-2}{q} + \frac{2}{q} = 1$
\begin{align}
\label{Eq:EstimMoment}
\nonumber
    \frac{\diff}{\diff t} \e\left(|\bar{X}_t|^q \right) \leq& - \kappa q \e\left( |\bar{X}_t|^q \right)  + (q-1) \varepsilon \e\left( |\bar{X}_t|^q \right) + \frac{1}{\varepsilon^{q - 1}} |F(0)|^q \\
    \nonumber
&+ q H_t^1 \e\left( |\bar{X}_t|^q \right) +  q \e\left( |\bar{X}_t|^{q - 1} \right) \left|\int_0^t L(t, s, 0, \mu_s) u_t(\mathrm{d} s)\right| \\
\nonumber
&+ \frac{1}{2} \sum_{i, j = 1}^d (\sigma ^t \sigma)_{i, j} \ d q (q-2) \e\left( |\bar{X}_t|^{q-2} \right) + \frac{q}{2} \tr\left(\sigma ^t \sigma\right) \e\left( |\bar{X}_t|^{q-2} \right) \\
\nonumber
\leq& - \kappa q \e\left( |\bar{X}_t|^q \right)  + (q-1) \varepsilon \e\left( |\bar{X}_t|^q \right) + \frac{1}{\varepsilon^{q - 1}} |F(0)|^q \\
\nonumber
&+ q H_t^1 \e\left( |\bar{X}_t|^q \right) +  q \e\left( |\bar{X}_t|^{q - 1} \right) \left|\int_0^t L(t, s, 0, \mu_s) u_t(\mathrm{d} s)\right| \\
\nonumber
&+ \frac{\varepsilon}{2} \left[ d (q-2)^2 \sum_{i, j = 1}^d (\sigma ^t \sigma)_{i, j} + (q - 2) \tr\left(\sigma ^t \sigma\right) \right] \e\left( |\bar{X}_t|^q \right) \\
&+ d (q-2) \frac{1}{\varepsilon^{q/2 - 1}} \sum_{i, j = 1}^d (\sigma ^t \sigma)_{i, j} + \frac{1}{\varepsilon^{q/2 - 1}} \tr\left(\sigma ^t \sigma\right).
\end{align}
\paragraph{Step 2. We treat the term $\e\left( |\bar{X}_t|^{q - 1} \right) \left|\int_0^t L(t, s, 0, \mu_s) u_t(\mathrm{d} s)\right|$.} 
By triangular inequality and Lipschitz hypothesis in $\mu$ of $L$, we get \begin{align*}
    |L(t, s, 0, \mu_s)| \leq& |L(t, s, 0, \delta_0)| + |L(t, s, 0, \mu_s) - L(t, s, 0, \delta_0)| \\
    \leq& |L(t, s, 0, \delta_0)| + h^2(t, s) \e(|\bar{X}_s|^q)^{1/q},
\end{align*} since $W_p(\delta_0, \mu_s) \leq \e(|\bar{X}_s|^p)^{1/p} \leq \e(|\bar{X}_s^q|)^{1/q}$. Therefore we have
\begin{align*}
    \e\left( |\bar{X}_t|^{q - 1} \right) & \left|\int_0^t L(t, s, 0, \mu_s) u_t(\mathrm{d} s)\right| \\
    \leq& \e\left( |\bar{X}_t|^{q - 1} \right) \left[ \int_0^t |L(t, s, 0, \delta_0)| \ u_t(\mathrm{d} s) + \int_0^t |L(t, s, 0, \mu_s) - L(t, s, 0, \delta_0)| \ u_t(\mathrm{d} s) \right] \\
    \leq& \varepsilon \frac{q - 1}{q} \e\left( |\bar{X}_t|^q \right) + \frac{1}{q \varepsilon^{q - 1}} \left( \int_0^t |L(t, s, 0, \delta_0)| \ u_t(\mathrm{d} s) \right)^q + \e\left( |\bar{X}_t|^{q - 1} \right) \int_0^t h^2(t, s) \e\left( |\bar{X}_s|^q \right)^{1/q} \ u_t(\mathrm{d} s),
\end{align*}
by Young inequality with $\frac{q - 1}{q} + \frac{1}{q} = 1$ with the parameter $\varepsilon > 0$. Then by Young inequality with $\frac{q - 1}{q} + \frac{1}{q} = 1$ and parameter $\alpha > 0$ and then by Jensen inequality, we get
\begin{align}
\label{Eq:term}
\nonumber
    \e\left( |\bar{X}_t|^{q - 1} \right) & \left|\int_0^t L(t, s, 0, \mu_s) u_t(\mathrm{d} s)\right| \\
    \nonumber
    \leq& \varepsilon \frac{q - 1}{q} \e\left( |\bar{X}_t|^q \right) + \frac{1}{q \varepsilon^{q - 1}} \left( \int_0^t |L(t, s, 0, \delta_0)| \ u_t(\mathrm{d} s) \right)^q \\
    \nonumber
    &+ \alpha \frac{q - 1}{q} H_t^2 \e\left( |\bar{X}_t|^q \right) + \frac{1}{q \alpha^{q - 1}} H_t^2 \left( \int_0^t \frac{h^2(t, s)}{\int_0^t h^2(t, s) \ u_t(\mathrm{d} s)} \e\left( |\bar{X}_s|^q \right)^{1/q} \ u_t(\mathrm{d} s) \right)^q \\
    \nonumber
    \leq& \varepsilon \frac{q - 1}{q} \e\left( |\bar{X}_t|^q \right) + \frac{1}{q \varepsilon^{q - 1}} \left( \int_0^t |L(t, s, 0, \delta_0)| \ u_t(\mathrm{d} s) \right)^q \\
    &+ \alpha \frac{q - 1}{q} H_t^2 \e\left( |\bar{X}_t|^q \right) + \frac{1}{q \alpha^{q - 1}} \int_0^t h^2(t, s) \e\left( |\bar{X}_s|^q \right) \ u_t(\mathrm{d} s).
\end{align}
Therefore we have, by plugging~\eqref{Eq:term} into~\eqref{Eq:EstimMoment} and reorganizing the terms
\begin{align}
\label{Eq:estimatemoment}
    \frac{\diff}{\diff t} \e\left(|\bar{X}_t|^q \right)
\leq& K_1(t) - K_2(t) \e\left( |\bar{X}_t|^q \right) + \frac{1}{\alpha^{q - 1}} \int_0^t h^2(t, s) \e\left( |\bar{X}_s|^q \right) \ u_t(\mathrm{d} s),
\end{align}
with
\begin{equation*}
    K_1(t) := \frac{1}{\varepsilon^{q - 1}} \left( \int_0^t |L(t, s, 0, \delta_0)| \ u_t(\mathrm{d} s) \right)^q + d (q-2) \frac{1}{\varepsilon^{q/2 - 1}} \sum_{i, j = 1}^d (\sigma ^t \sigma)_{i, j} + \frac{1}{\varepsilon^{q/2 - 1}} \tr\left(\sigma ^t \sigma\right) + \frac{1}{\varepsilon^{q - 1}} |F(0)|^q \leq K_1
\end{equation*}
and
\begin{equation*}
    K_2(t) := \kappa q - q H_t^1  - \alpha (q - 1) H_t^2 - 2 \varepsilon (q-1) - \frac{d}{2} \varepsilon (q-2)^2 \sum_{i, j = 1}^d (\sigma ^t \sigma)_{i, j} - \varepsilon (q - 2) \frac{\tr\left(\sigma ^t \sigma\right)}{2},
\end{equation*}
since $\sup_{t \geq 0} \int_0^t |L(t, s, 0, \delta_0)| \ u_t(\mathrm{d} s) < \infty$.
Let us fix $D^1 > \overline{D}^1$ and $D^2 > \overline{D}^2$ such that $\kappa > D^1 + D^2$. There exists $t_0$ such that for all $t \geq t_0$, we have
\begin{equation}
\label{asympt}
    H_t^1 \leq D^1 \quad \text{and} \quad H_t^2 \leq D^2.
\end{equation}
Then for all $t \geq t_0$, by plugging~\eqref{asympt} into~\eqref{Eq:estimatemoment}, and using Proposition~\ref{wellposedness}, we have
\begin{align}
\frac{\diff}{\diff t} \e\left(|\bar{X}_t|^q \right)
\leq& K_1 + \frac{D^2}{\alpha^{q - 1}} \sup_{0 \leq s \leq t_0} \e\left( |\bar{X}_s|^q \right) - K_2 \e\left( |\bar{X}_t|^q \right) + \frac{D^2}{\alpha^{q - 1}} \int_{t_0}^t \frac{h^2(t, s)}{\int_0^t h^2(t, s) \ u_t(\mathrm{d} s)} \e\left( |\bar{X}_s|^q \right) \ u_t(\mathrm{d} s),
\end{align}
where
\begin{align*}
    K_2 := q \left( \kappa - D^1 \right) - \alpha (q - 1) D^2 - 2 \varepsilon (q - 1) - \frac{\varepsilon}{2} \left[ d (q - 2)^2 \sum_{i, j = 1}^d (\sigma ^t \sigma)_{i, j} + (q - 2) \tr\left( \sigma ^t \sigma \right) \right].
\end{align*}

\paragraph{Step 3. Constant optimization and application of a generalized integro-differential Gronwall's lemma.}
For the proof of Lemma~\ref{moments}, we will need the following lemma which is a Gronwall-type lemma for integro-differential inequations and is inspired by Lemma~$2.4$ of~\cite{deraynal2021reducingexittimesdiffusionsrepulsive}.
\begin{lemma}
\label{Lem:boundedgronwall}
    Let $0 \leq b < a$, and $c \in \mathbb{R}$. If for all $t \geq t_0$
    \begin{equation*}
        x'(t) \leq - a x(t) + b \int_{t_0}^t \frac{h^2(t, s)}{\int_0^t h^2(t, v) \; u_t(\mathrm{d} v)} x(s) \; u_t(\mathrm{d} s) + c,
    \end{equation*}
    then we have 
    \begin{equation*}
        x(t) \leq \frac{c}{a - b} + \left( x(t_0) - \frac{c}{a - b} \right)_+.
    \end{equation*}
\end{lemma}
\begin{proof}
Denote $A := \frac{c}{a - b}$.

First, let us show that if $x(t_0) \leq A$, then for all $t \geq t_0$, $x(t) \leq A$. Let us take $B > A$. Suppose $x(t_0) \leq B$. Denote $t^\ast := \inf\{ s \geq t_0, \; x(s) > B \}$. Suppose by contradiction that $t^\ast < \infty$.

By continuity, $x(t^\ast) = B$, and for all $t_0 \leq s \leq t^\ast$, $x(s) \leq B$. Therefore
\begin{equation*}
    x'(t^\ast) \leq - a B + b B + c = - (a - b)(B  - A) < 0.
\end{equation*}
So there exists $\gamma > 0$ such that for all $u \in [0, 1]$, $x(t^\ast + u \gamma) \leq B$. Contradiction.

So $x(t) \leq B$, for all $t \geq t_0$. Since this is true for all $B > A$, we get that if $x(t_0) \leq A$, then $x(t) \leq A$ for all $t \geq t_0$.
\bigbreak

Then, let us prove the result if $x(t_0) > A$. Denote $g(t) := \frac{x(t) - A}{x(t_0) - A}$. This function is well defined thanks to the previous case, and it satisfies $g(t) \leq 1$ for all $t \geq t_0$ as well, so
\begin{equation*}
    x(t) \leq A + \left( x(t_0) - A \right). 
\end{equation*}
\end{proof}
In order to apply Lemma~\ref{Lem:boundedgronwall}, we need the inequality $K_2 > D^2 / \alpha^{q - 1}$ to be satisfied, that is to say
\begin{equation*}
    q\left( \kappa - \left( \frac{q - 1}{q} \alpha + \frac{1}{\alpha^{q - 1} q} \right) D^2 - D^1 \right) - 2 \varepsilon (q - 1) - \frac{d}{2} \varepsilon (q - 2)^2 \sum_{i, j = 1}^d (\sigma ^t \sigma)_{i, j} - \varepsilon (q - 2) \frac{\tr\left(\sigma ^t \sigma \right)}{2} > 0
\end{equation*}
Let us first take $\alpha := 1$ that minimizes the function $\alpha \in \mathbb{R}_+^* \mapsto (q-1) \alpha + 1/\alpha^{q - 1}$ and therefore optimizes our bound. We end up with the condition 
\begin{equation*}
    q\left( \kappa - D^1 - D^2 \right) - 2 \varepsilon (q - 1) - \frac{d}{2} \varepsilon q (q - 2)^2 \sum_{i, j = 1}^d (\sigma^2)_{i, j} - \varepsilon q (q - 2) \frac{\tr(\sigma^2)}{2} > 0.
\end{equation*}
Provided $\varepsilon$ is small enough, the condition is satisfied.
By applying Lemma~\ref{Lem:boundedgronwall} with $a := K_2$, $b := D^2 / \alpha^{q - 1}$, $c := K_1 + \frac{D^2}{\alpha^{q - 1}} \sup_{0 \leq s \leq t_0} \e\left( |\bar{X}_s|^q \right)$, and $x(t) := \e\left( |\bar{X}_t|^q \right)$. Assuming $\bar{X}_0 \in \mathcal{P}_q\left( \mathbb{R}^d \right)$, we get that
\begin{equation*}
    \sup_{t \geq t_0} \e|\bar{X}_t|^q < \infty.
\end{equation*}
Since Proposition~\ref{wellposedness} yields
\begin{equation*}
    \sup_{0 \leq t \leq t_0} \e|\bar{X}_t|^q < \infty,
\end{equation*}
this concludes the proof.

\subsection{Proof of the uniform-in-time propagation of chaos (Theorem~\ref{unifpropachaos})} 
Thanks to the contraction hypothesis $\kappa > \overline{D}^1 + \overline{D}^2$, we can choose parameters $\varepsilon$ and $\Tilde{\varepsilon}$ that allow the contraction to happen in the Gronwall-type estimate and get a uniform-in-time estimate.
\begin{proof}
Let us fix $D^1 > \overline{D}^1$ and $D^2 > \overline{D}^2$ such that $\kappa > D^1 + D^2$. There exists $t_0$ such that for all $t \geq t_0$, we have
\begin{equation}
\label{asym}
    H_t^1 \leq D^1 \quad \text{and} \quad H_t^2 \leq D^2.
\end{equation}
Define for all $T \geq t_0$
\begin{equation}
\label{x}
    \left\{ \begin{array}{rcl}
        x'(t) &=& - \left( p (\kappa - H_t^1) - H_t^2 (\varepsilon + \Tilde{\varepsilon}) (p - 1) \right) x(t) + \displaystyle \frac{1}{\Tilde{\varepsilon}^{p - 1}} \int_0^t h^2(t, s) x(s) \, u_t(\mathrm{d} s) + \frac{H_t^2 v(N)}{p \varepsilon^{p - 1}} \sup_{t_0 \leq s \leq T} \e(|\bar{X}_s|^q)^{p/q} \\
        x(t_0) &=& \displaystyle \frac{1}{N} \sum_{i=1}^N \frac{\e|X_{t_0}^i - \bar{X}_{t_0}^i|^p}{p}, \quad \text{for all $t_0 \leq t \leq T$}.
    \end{array}\right.
\end{equation}
It is slightly different from~\eqref{upperbound}, but the only difference is that we start from $t_0$. Note that $x$ is such that for all $t \geq t_0$
\begin{equation*}
    \frac{1}{N} \sum_{i = 1}^N \frac{\e(|X_t^i - \bar{X}_t^i|^p)}{p} \leq x(t),
\end{equation*}
thanks to Theorem~\ref{bigtheorem}.
By using Lemma~\ref{moments} and~\eqref{asym}, we get that for all $t \geq t_0$, we have
\begin{equation*}
    x'(t) \leq - \left( p (\kappa - D^1) - D^2 (\varepsilon + \Tilde{\varepsilon}) (p - 1) \right) x(t) + \frac{D^2}{\Tilde{\varepsilon}^{p - 1}} \int_0^t \frac{h^2(t, s)}{\int_0^t h^2(t, s) \ u_t(\mathrm{d} s)} x(s) \; u_t(\mathrm{d} s) + \frac{D^2 v(N)}{p \varepsilon^{p - 1}} M_{p, q},
\end{equation*}
where $M_{p, q}$ is defined in~\eqref{momentuniform}.
In order to apply Lemma~\ref{Lem:boundedgronwall}, we want the condition
\begin{equation*}
    p \left( \kappa - D^1 \right) - D^2 (\Tilde{\varepsilon} + \varepsilon) (p-1) > \frac{D^2}{\Tilde{\varepsilon}^{p - 1}}
\end{equation*}
to be satisfied. It is enough to take $\varepsilon$ small if 
\begin{equation*}
   p \left( \kappa - D^1 \right) - \Tilde{\varepsilon} D^2 (p-1) > \frac{D^2}{\Tilde{\varepsilon}^{p - 1}}
\end{equation*}
is satisfied. As in the previous proof, the function $\Tilde{\varepsilon} \mapsto \Tilde{\varepsilon} (p - 1) + 1/\Tilde{\varepsilon}^{p - 1}$ is minimal at $\Tilde{\varepsilon} := 1$, so that the condition boils down to $\kappa > D^1 + D^2$, which we supposed.

By applying Lemma~\ref{Lem:boundedgronwall} with all of the choices of parameters that make the assumption satisfied, we get that for all $t \geq t_0$
\begin{equation*}
    \frac{1}{N} \sum_{i=1}^N \e( |X_t^i - \bar{X}_t^i|^p) \leq C(p, q, \kappa, D^1, D^2, \varepsilon) v(N) + \left( \frac{1}{N} \sum_{i=1}^N \e( |X_{t_0}^i - \bar{X}_{t_0}^i|^p) - C(p, q, \kappa, D^1, D^2, \varepsilon) v(N)\right)_+,
\end{equation*}
with 
\begin{equation*}
    C(p, q, \kappa, D^1, D^2, \varepsilon) := \frac{D^2 M_{p, q}}{p \varepsilon^{p - 1} \left( p \left( \kappa - D^1 - D^2\right) - \varepsilon D^2 (p - 1) \right)}.
\end{equation*}
But then, by plugging in~\eqref{eqPoC}, we get the wanted estimate for all $t \geq t_0$. We get a similar estimate for all $0 \leq t \leq t_0$, so this concludes.
\end{proof}

\section{Equilibrium and long-time convergence}
\label{secproofequilibrium}
Let us prove the long-time behaviour results, by starting with the proof of the long-time contraction for the mean-field model, then the characterization of the equilibrium for the stationary effective memory case as well as the existence of a unique stationary solution in this case, the result on the existence of an equilibrium in the case of asymptotic-in-time stationary effective memory, and finally the result on exchangeability of the limits.

\subsection{Proof of the long-time contraction for the mean-field model (Proposition~\ref{meanfieldlongtime}).} 
Let $X_t^{i, N}$, for $i = 1, \dots, N$ be solutions to the system~\eqref{particles}, with i.i.d. initial conditions $\bar{X}_0^i$, driven by $(B_t^i)_{1 \leq i \leq N}$ $N$ independent Brownian motions. Let $\bar{Y}_t^i$, for $i = 1, \dots, N$ be $N$ independent solutions to~\eqref{limit}, each driven by $B_t^i$ the exact same Brownian motion as for $X_t^i$, and with i.i.d. initial conditions $\bar{Y}_0^i$.

Let us fix $D^1 > \overline{D}^1$ and $D^2 > \overline{D}^2$ such that $\kappa > D^1 + D^2$. There exists $t_0$ such that for all $t \geq t_0$, we have
\begin{equation}
\label{asym_}
    H_t^1 \leq D^1 \quad \text{and} \quad H_t^2 \leq D^2.
\end{equation}
Define for all $T \geq t_0$
\begin{equation}
\label{x_}
    \left\{ \begin{array}{rcl}
        x'(t) &=& - \left( p (\kappa - H_t^1) - H_t^2 (\varepsilon + \Tilde{\varepsilon}) (p - 1) \right) x(t) + \displaystyle \int_0^t \frac{h^2(t, s)}{\Tilde{\varepsilon}^{p - 1}} x(s) \, u_t(\mathrm{d} s) + \frac{H_t^2 v(N)}{p \varepsilon^{p - 1}} \sup_{t_0 \leq s \leq T} \e(|\bar{X}_s|^q)^{p/q} \\
        x(t_0) &=& \displaystyle \frac{1}{N} \sum_{i=1}^N \frac{\e|X_{t_0}^i - \bar{Y}_{t_0}^i|^p}{p}, \quad \text{for all $t_0 \leq t \leq T$}.
    \end{array}\right.
\end{equation}
Note that $x$ is such that for all $t \geq t_0$
\begin{equation*}
    \frac{1}{N} \sum_{i = 1}^N \frac{\e(|X_t^i - \bar{X}_t^i|^p)}{p} \leq x(t),
\end{equation*}
thanks to Theorem~\ref{bigtheorem}. By using Lemma~\ref{wellposedness}, we get that for all $t \geq t_0$, we have
\begin{equation*}
    x'(t) \leq \kappa \frac{v(N)}{p \varepsilon^{p - 1}} M_{p, q} - \left[  p ( \kappa - D^1) - D^2 (\varepsilon + \Tilde{\varepsilon}) (p - 1) \right] x(t) + \frac{D^2}{\Tilde{\varepsilon}^{p - 1}} \int_0^t \frac{h^2(t, s)}{\int_0^t h^2(t, v) \; u_t(\mathrm{d} v)} x(s) \; u_t(\mathrm{d} s),
\end{equation*}
where $M_{p, q}$ is defined in~\eqref{momentuniform}.

Again, we want the condition
\begin{equation*}
    p ( \kappa - D^1) - D^2 (\varepsilon + \Tilde{\varepsilon}) (p - 1) > \frac{D^2}{\Tilde{\varepsilon}^{p - 1}}
\end{equation*}
to be satisfied. See the proof of Theorem~\ref{unifpropachaos} for the choice of parameters that satisfy this.

Instead of using Lemma~\ref{Lem:boundedgronwall}, we will need a sharper estimate, using Assumption~\ref{memoryloss}, that is this time exactly Lemma~$2.4$ of~\cite{deraynal2021reducingexittimesdiffusionsrepulsive}, but we added its proof for the sake of completeness.
\begin{lemma}
\label{Lem:decreasingGronwall}
    Let $0 \leq b < a$. If for all $t \geq t_0$,
\begin{equation*}
    x'(t) \leq - a x(t) + b \int_{t_0}^t \frac{h^2(t, s)}{\int_0^t h^2(t, v) \; u_t(\mathrm{d} v)} x(s) \; u_t(\mathrm{d} s) + c,
\end{equation*}
then under Assumption~\ref{memoryloss}, there exists $m : \mathbb{R}_+ \to \mathbb{R}_+$ a decreasing function such that for all $t \geq t_0$, we have
\begin{equation*}
    x(t) \leq \frac{c}{a - b} + m(t) \left( x(t_0) -  \frac{c}{a - b} \right)_+.
\end{equation*}
\end{lemma}

\begin{proof}
Denote $A := \frac{c}{a - b}$. Show that there exists $m : \mathbb{R}_+ \to \mathbb{R}_+$ such that $x(t) \leq A + m(t) \left( x(t_0) - A \right)_+$.

First, let us show that if $x(t_0) \leq A$, then for all $t \geq t_0$, $x(t) \leq A$.

Let us take $B > A$. Suppose $x(t_0) \leq B$. Denote $t^\ast := \inf\{ s, \; x(s) > B \}$. Suppose by contradiction that $t^\ast < \infty$.

By continuity, $x(t^\ast) = B$, and for all $0 \leq s \leq t^\ast$, $x(s) \leq B$. Therefore
\begin{equation*}
    x'(t^\ast) \leq - a B + b B + c = - (a - b)(B  - A) < 0.
\end{equation*}
So there exists $\eta > 0$ such that for all $u \in [0, 1]$, $x(t^\ast + u \eta) \leq B$. Contradiction.

So $x(t) \leq B$, for all $t \geq t_0$. Since this is true for all $B > A$, we get that if $x(t_0) \leq A$, then $x(t) \leq A$ for all $t \geq t_0$.
\bigbreak

Then, let us prove the result if $x(t_0) > A$. Denote $g(t) := \frac{x(t) - A}{x(t_0) - A}$. This function is well defined thanks to the previous case, and it satisfies $g(t) \leq 1$ for all $t \geq t_0$ as well. Then we have the following inequality
\begin{equation*}
    g'(t) \leq - a g(t) + b \int_{t_0}^t \frac{h^2(t, s)}{\int_0^t h^2(t, v) \; u_t(\mathrm{d} v)} g(s) \; u_t(\mathrm{d} s) \leq - a g(t) + b.
\end{equation*}
By a similar reasoning as in \textbf{Step $5$.} of the proof of Theorem~\ref{bigtheorem}, we get that for all $t \geq t_0$
\begin{equation*}
    g(t) \leq m_0(t), \quad \text{where} \quad m_0(t) := e^{- a (t - t_0)} + \frac{b}{a} (1 - e^{- a (t - t_0)}), \quad t \geq t_0
\end{equation*}
That is to say $g$ is smaller than $m_0$ solution to the ODE $y' = - a y + b$ on $[t_0, +\infty)$ with initial condition $y(t_0) = 1$. Let us take $\gamma := \sqrt{b/a}$. Remark that $0 < b/a < \gamma < 1$. Set $\Tilde{t}_0 := t_0$.

Suppose by induction that for a fixed $n \in \mathbb{N}$, we have constructed $m_n : [t_, +\infty) \to \mathbb{R}_+$ and $\Tilde{t}_n \geq n$, and that $m_n(t) \xrightarrow[t \to + \infty]{} \gamma^{n + 2}$.

Choose $t_{n + 1}' \geq \Tilde{t}_n + 1$ such that for all $t \geq t_{n + 1}'$, we have
\begin{equation}
\label{m_nbound}
    m_n(t) \leq \frac{\gamma^{n + 1}}{2} \leq \gamma^{n + 1}.
\end{equation}
Then we have 
\begin{equation*}
    \int_{t_0}^t \frac{h^2(t, s)}{\int_0^t h^2(t, v) \; u_t(\mathrm{d} v)} m_n(s) \; u_t(\mathrm{d} v) = \int_{t_0}^{t_{n + 1}'} \frac{h^2(t, s)}{\int_0^t h^2(t, v) \; u_t(\mathrm{d} v)} m_n(s) \; u_t(\mathrm{d} v) + \int_{t_{n + 1}'}^t \frac{h^2(t, s)}{\int_0^t h^2(t, v) \; u_t(\mathrm{d} v)} m_n(s) \; u_t(\mathrm{d} v).
\end{equation*}
We already know that 
\begin{equation}
\label{one}
    \int_{t_{n + 1}'}^t \frac{h^2(t, s)}{\int_0^t h^2(t, v) \; u_t(\mathrm{d} v)} m_n(s) \; u_t(\mathrm{d} v) \leq \frac{\gamma^{n + 1}}{2}.
\end{equation}
By hypothesis, there exists $\Tilde{t}_{n + 1} \geq t_{n + 1}'$ such that for all $t \geq \Tilde{t}_{n + 1}$, we have
\begin{equation}
\label{two}
    \int_0^{t_{n + 1}'} \frac{h^2(t, s)}{\int_0^t h^2(t, v) \; u_t(\mathrm{d} v)} m_n(s) \; u_t(\mathrm{d} v) \leq \frac{\gamma^{n + 1}}{2}.
\end{equation}
Therefore, by adding~\eqref{one} and~\eqref{two}, we have that for all $t \geq \Tilde{t}_{n + 1}$
\begin{equation*}
    \int_{t_0}^t \frac{h^2(t, s)}{\int_0^t h^2(t, v) \; u_t(\mathrm{d} v)} m_n(s) \; u_t(\mathrm{d} v) \leq \gamma^{n + 1}.
\end{equation*}
Define then $m_{n + 1}(t) := m_n(t)$ if $t < \Tilde{t}_{n + 1}$, and 
\begin{equation*}
    m_{n + 1}(t) := e^{- a (t - \Tilde{t}_{n + 1})} \gamma^{n + 1} + \left( 1 - e^{- a (t - \Tilde{t}_{n + 1})} \right) \gamma^{n + 3},
\end{equation*}
if $t \geq \Tilde{t}_{n + 1}$, we get that $m_{n + 1}(t) \xrightarrow[t \to + \infty]{} \gamma^{n + 3}$. This concludes the construction by induction.

Let us now prove by induction that $g(t) \leq m_n(t)$ for all $t \geq t_0$. We already treated the case $n = 0$ previously. Let us suppose that, for a fixed $n \in \mathbb{N}$, $g(t) \leq m_n(t)$ for all $t \geq 0$.

Then, for all $0 \leq t < \Tilde{t}_{n + 1}$, we have $g(t) \leq m_{n + 1}(t) = m_n(t)$, by hypothesis of induction. Then, if $t \geq \Tilde{t}_{n + 1}$, we have 
\begin{equation*}
    g'(t) \leq - a g(t) + \int_0^t \frac{h^2(t, s)}{\int_0^t h^2(t, v) \; u_t(\mathrm{d} v)} m_n(s) \; u_t(\mathrm{d} s) \leq - a g(t) + b \gamma^{n + 1},
\end{equation*}
by construction of $\Tilde{t}_{n + 1}$. Therefore for all $t \geq \Tilde{t}_{n + 1}$, we have 
\begin{align*}
    g(t) &\leq g(\Tilde{t}_{n + 1}) e^{- a (t - \Tilde{t}_{n + 1})} + \left( 1 - e^{ - a (t - \Tilde{t}_{n + 1})}\right) \frac{b}{a} \gamma^{n + 1} \\
    &\leq \gamma^{n + 1} e^{- a (t - \Tilde{t}_{n + 1})} + \left( 1 - e^{ - a (t - \Tilde{t}_{n + 1})}\right) \gamma^{n + 3} = m_{n + 1}(t),
\end{align*}
because $g(\Tilde{t}_{n + 1}) \leq m_n(\Tilde{t}_{n + 1})$ by induction hypothesis, and because of~\eqref{m_nbound} and because of the definition of $\gamma$. By induction, this concludes the proof that $g(t) \leq m_n(t)$, for all $t \geq t_0$ and all integer $n$.

Finally, define $m : [ t_0, + \infty ) \to \mathbb{R}_+$ a function such that $m(t) := \gamma^n$ on $[\Tilde{t}_n, \Tilde{t}_{n + 1})$. This defines a function on $[t_0, + \infty)$ since $n \leq \Tilde{t}_n \to + \infty$. Therefore, by~\eqref{m_nbound}, we know that $g(t) \leq m(t)$ for all $t \geq t_0$, and $m(t) \xrightarrow[t \to + \infty]{} 0$.

This concludes, by using the definition of $g$.
\end{proof}
By application of the lemma to $x$ defined in~\eqref{x}, we get that for all $t \geq t_0$
\begin{equation*}
    x(t) \leq C v(N) + m(t) \left( \frac{1}{N} \sum_{i = 1}^N \e(|\bar{X}_{t_0}^i - \bar{Y}_{t_0}^i|^p) - C v(N) \right)_+.
\end{equation*}
By applying~\eqref{eqPoC}, we then get that for all $t \geq t_0$
\begin{equation*}
    \frac{1}{N} \sum_{i = 1}^N \e(|X_t^{i, N} - \bar{Y}_t^i|^p) \leq C_1 v(N) + m(t) \left( C_2 \frac{1}{N} \sum_{i = 1}^N \e(|\bar{X}_0^i - \bar{Y}_0^i|^p) - C_3 v(N) \right)_+.
\end{equation*}
Since $Law\left( X_t^i - \bar{Y}_t^i \right) = Law\left( X_t^1 - \bar{Y}_t^1 \right)$, and $Law\left( \bar{X}_0^i - \bar{Y}_0^i \right) = Law\left( X_0^1 - \bar{Y}_0^1 \right)$ for all $i = 1, \dots, N$, then we have for all $t \geq t_0$
\begin{equation*}
    \e(|X_t^{1, N} - \bar{Y}_t^1|^p) \leq C_1 v(N) + m(t) \left[ C_2 \e(|\bar{X}_0^1 - \bar{Y}_0^1|^p) - C_3 v(N) \right]_+.
\end{equation*}
By letting $N \to + \infty$, we get
\begin{equation*}
    \e(|\bar{X}_t^1 - \bar{Y}_t^1|^p) \leq C_2 m(t) \e(|\bar{X}_0^1 - \bar{Y}_0^1|^p),
\end{equation*}
where $\bar{X}_t^1$ is solution to the limit system associated with the particle system $X_t^i$.

Note that by optimization on all the couplings between $Law(\bar{X}_t^1)$ and $Law(\bar{Y}_t^1)$, and the couplings between $Law(\bar{X}_0^1)$ and $Law(\bar{Y}_0^1)$, we get
\begin{equation}
    \label{contractionMFL}
    W_p\left( Law(\bar{X}_t^1), Law(\bar{Y}_t^1) \right) \leq C_2 ^{1/p} m(t)^{1/p} W_p\left( Law(\bar{X}_0^1), Law(\bar{Y}_0^1) \right). 
\end{equation}

\subsection{Proof of Proposition~\ref{coro}.}
\begin{itemize}
    \item Suppose the condition ($1$) is satisfied. Let $\bar{X}_t$ be the solution to~\eqref{fixed} with initial condition $\mu_\infty$. Since $\mu_\infty$ is invariant, we have $Law\left( \bar{X}_t\right) = \mu_\infty$ for all $t \geq 0$. So $\bar{X}_t$ is also solution to the McKean-Vlasov equation~\eqref{changing}. By strong uniqueness, it means $\bar{X}_t$ is the only (in law) solution to~\eqref{changing}) with initial condition $\mu_\infty$. So the solution to~\eqref{changing} with initial law $\mu_\infty$ has a constant law equal to $\mu_\infty$. Therefore the condition ($1$) implies the condition ($2$).
    \item Suppose that the condition ($2$) is satisfied. Consider $\bar{X}_t$ the solution to~\eqref{changing} with $\bar{X}_0 \sim \mu_\infty$. It satisfies $\mu_t := Law\left(\bar{X}_t\right) = \mu_\infty$ for all $t \geq 0$, so that, under Assumption~\ref{b}, 
    \begin{equation*}
        \int_0^t L(t, s, \bar{X}_t, \mu_s) \; u_t(\mathrm{d} s) = b(\bar{X}_t, \mu_\infty) = b(\bar{X}_t, \mu_t).
    \end{equation*}
    So $\bar{X}_t$ is also solution to~\eqref{fullmemory}. By uniqueness in law of the solution of~\eqref{fullmemory} given by Proposition~\ref{wellposedness}, it is the only solution to~\eqref{fullmemory} with initial distribution $\mu_\infty$. Therefore the condition ($3$) is satisfied as well.
    \item Suppose now that condition ($3$) is satisfied. Consider $\bar{X}_t$ solution to~\eqref{fullmemory} with initial distribution $\mu_\infty$. It is such that $\mu_t = \mu_\infty$ for all $t \geq 0$. Under Assumption~\ref{b}, we have
    \begin{equation*}
        \int_0^t L(t, s, \bar{X}_t, \mu_s) \; u_t(\mathrm{d} s) = b(\bar{X}_t, \mu_\infty),
    \end{equation*}
    so that $\bar{X}_t$ is a solution to~\eqref{fixed}. Since the law of $\bar{X}_t$ is constant, it is an invariant measure of~\eqref{fixed}, so that the condition ($1$) is satisfied.
\end{itemize}

\subsection{Proof of Corollary~\ref{equilibrium} on the existence of an equilibrium.}
It is enough to prove the existence and uniqueness of an equilibrium for~\eqref{changing}. Let us denote $\Phi_t$ the application mapping $\mu_0$ to $\mu_t$, where $\mu_t$ is the law of the solution to~\eqref{changing} at time $t$, starting from the initial distribution $\mu_0$. Let us take $\mu_0$, $\nu_0 \in \mathcal{P}_p\left( \mathbb{R}^d \right)$, and $X_t$, $Y_t$ solutions to~\eqref{changing} with initial distributions respectively $\mu_0$ and $\nu_0$. We denote $\mu_t$ the law of $X_t$ and $\nu_t$ the law of $Y_t$. Then we have
\begin{align*}
    \frac{\diff}{\diff t} \frac{|X_t - Y_t|^p}{p} =& - |X_t - Y_t|^{p - 2} \left( F(X_t) - F(Y_t) \right) \cdot \left( X_t - Y_t \right) \\
    &+ |X_t - Y_t|^{p - 2} \left( b(X_t, \mu_t) - b(X_t, \nu_t) \right) \cdot \left( X_t - Y_t \right) + |X_t - Y_t|^{p - 2} \left( b(X_t, \nu_t) - b(Y_t, \nu_t) \right) \cdot \left( X_t - Y_t \right) \\
    \leq& - p \kappa \frac{|X_t - Y_t|^p}{p} + H_t^2 |X_t - Y_t|^{p - 1} W_p\left( \mu_t, \nu_t \right) + H_t^1 |X_t - Y_t|^{p - 1} \, |X_t - Y_t|.  
\end{align*}
Let us now apply Young inequality with $\frac{1}{p} + \frac{p - 1}{p} = 1$.
\begin{equation}
\label{orange}
    \frac{\diff}{\diff t} \frac{|X_t - Y_t|^p}{p}
    \leq - p \kappa \frac{|X_t - Y_t|^p}{p} + \varepsilon (p - 1) H_t^2 \frac{|X_t - Y_t|^p}{p} + \frac{1}{p \varepsilon^{p - 1}} H_t^2 W_p^p\left( \mu_t, \nu_t \right) + p H_t^1 \frac{|X_t - Y_t|^p}{p}.  
\end{equation}
Let us fix $D^1 > \overline{D}^1$ and $D^2 > \overline{D}^2$ such that $\kappa > D^1 + D^2$. There exists $t_0$ such that for all $t \geq t_0$, we have
\begin{equation}
\label{asy}
    H_t^1 \leq D^1 \quad \text{and} \quad H_t^2 \leq D^2.
\end{equation}
By taking the expectation, using $W_p^p\left( \mu_t, \nu_t \right) \frac{1}{N} \sum_{i = 1}^N \e|X_t - Y_t|^p$ and using~\eqref{asy}, we get that for all $t \geq t_0$
\begin{equation*}
    \frac{\diff}{\diff t} \frac{\e(|X_t - Y_t|^p)}{p}
    \leq p \left(- \kappa + \left(\varepsilon \frac{p - 1}{p} + \frac{1}{p \varepsilon^{p - 1}} \right) D^2 + D^1 \right) \frac{\e(|X_t - Y_t|^p)}{p}.
\end{equation*}
Then, by taking $\varepsilon := 1$ that optimizes the bound, we get that for all $t \geq t_0$
\begin{equation*}
    \frac{\diff}{\diff t} \frac{\e(|X_t - Y_t|^p)}{p} \leq p \left(- \kappa + D^1 + D^2 \right) \frac{\e(|X_t - Y_t|^p)}{p}.
\end{equation*}
We have, by Gronwall's lemma, for all $t \geq t_0$
\begin{equation}
\label{asymptotic}
    \e(|X_t - Y_t|^p) \leq C e^{- \lambda t} \e(|X_{t_0} - Y_{t_0}|^p),
\end{equation}
for certain positive constants $C$ and $\lambda$.

Additionally, by taking the expectation of~\eqref{orange} and using $D_{t_0}$ defined in~\eqref{D_T}, we have for all $t \in [0, t_0]$
\begin{equation*}
    \frac{\diff}{\diff t} \frac{\e(|X_t - Y_t|^p)}{p}
    \leq p \left(- \kappa + \left( 1 + \varepsilon \frac{p - 1}{p} + \frac{1}{p \varepsilon^{p - 1}} \right) D_{t_0} \right) \frac{\e(|X_t - Y_t|^p)}{p}.
\end{equation*}
We choose again $\varepsilon = 1$, and denote $\lambda' := - p \kappa + 2 p D_{t_0}$. By Gronwall's lemma, we get that
\begin{equation}
\label{finitetime}
    \e(|X_{t_0} - Y_{t_0}|^p) \leq C' e^{\lambda' t_0} \e(|X_0 - Y_0|^p),
\end{equation}
for certain positive constants $C'$ and $\lambda'$. Then by plugging~\eqref{finitetime} into~\eqref{asymptotic}, we get that for all $t \geq t_0$
\begin{equation*}
     \e(|X_t - Y_t|^p) \leq \Tilde{C} e^{- \lambda t} \e(|X_0 - Y_0|^p).
\end{equation*}
By optimizing on all of the couplings between $\mu_t = \Phi_t(\mu_0)$ and $\nu_t = \Phi_t(\nu_0)$ and all of the couplings between $\mu_0$ and $\nu_0$, we get that
\begin{equation*}
    W_p^p\left( \Phi_t(\mu_0), \Phi_t(\nu_0) \right) \leq \Tilde{C} e^{- \lambda t} W_p^p\left( \mu_0, \nu_0 \right),
\end{equation*}
which means that for $\Phi_t$ is a contraction on $\mathcal{P}_p\left( \mathbb{R}^d \right)$, so it has a unique fixed point that we denote $m_t$. Then, for all $s \geq 0$, we have 
\begin{equation*}
    \Phi_s(m_t) = \Phi_s(\Phi_t(m_t)) = \Phi_{t + s}(m_t) = \Phi_t(\Phi_s(m_t)).
\end{equation*}
So $\Phi_s(m_t)$ and $m_t$ are both fixed point of $\Phi_t$ which has only one fixed point. So $\Phi_s(m_t) = m_t$, and $m_t$ does not depend on $t$, in fact. 
Let us therefore denote it $m_\infty$ for the rest of the proof. 
We have for all $t \geq 0$, $\Phi_t(m_\infty) = m_\infty$, 
which exactly means that $m_\infty$ is an equilibrium for~\eqref{changing}. This concludes for the existence and uniqueness of an equilibrium for~\eqref{fullmemory}, thanks to Proposition~\ref{coro}.

\subsection{Proof of Proposition~\ref{longtime}.}
The proof breaks into four steps. In \textbf{Step 1.}, we establish some Lipschitz regularity result on $b$. Then in \textbf{Step 2.}, we deduce that there exists an equilibrium for~\eqref{changing}. In \textbf{Step 3.}, we compare the solutions to~\eqref{changing} and~\eqref{fullmemory}. Eventually, in \textbf{Step 4.}, we prove that solutions to~\eqref{fullmemory} all tend to $\mu_\infty$, no matter what the initial condition is.
\paragraph{Step 1. Let us find some regularity property on $b$.}
Let us deduce some property on $b$ from Assumption~\ref{assumparticles}. Let us take $\mu$, $\nu \in \mathcal{P}_p\left( \mathbb{R}^d \right)$, such that their $p$-th moment is smaller than $M > 0$, and $\pi$ a coupling between $\mu$ and $\nu$. Let us fix $\varepsilon > 0$. Let us denote
\begin{equation*}
    \varepsilon_t(m) := \left( \int_x \left| b(x, m) - \int_0^t L(t, s, x, m) \; u_t(\mathrm{d} s) \right|^p \; m(\mathrm{d} x) \right)^{1/p}.
\end{equation*}
By applying Minkowski inequality twice $\mathbb{L}^p(\pi)$
\begin{align*}
    \left( \int_x \int_y |b(x, \mu) - b(y, \nu)|^p \; \pi(\mathrm{d} x, \diff y) \right)^{1/p} \leq& \left( \int_x \left| b(x, \mu) - \int_0^t L(t, s, x, \mu) \; u_t(\mathrm{d} s) \right|^p \; \mu(\mathrm{d} x) \right)^{1/p} \\
    &+ \left( \int_x \int_y \left| \int_0^t \left( L(t, s, x, \mu) - L(t, s, y, \nu) \right) \; u_t(\mathrm{d} s) \right|^p \; \pi(\mathrm{d} x, \diff y) \right)^{1/p} \\
    &+ \left( \int_y \left| \int_0^t L(t, s, x, \mu) \; u_t(\mathrm{d} s) - b(y, \nu) \right|^p \; \nu(\mathrm{d} y) \right)^{1/p} \\
    \leq& \varepsilon_t(\mu) + \varepsilon_t(\nu) + \left( \int_x \int_y \left( H_t^1 |x - y| + H_t^2 W_p\left( \mu, \nu \right) \right)^p \; \pi(\mathrm{d} x, \diff y) \right)^{1/p} \\
    \leq& \varepsilon_t(\mu) + \varepsilon_t(\nu) + H_t^1 \left( \int_x \int_y \left| x - y \right|^p \; \pi(\mathrm{d} x, \diff y) \right)^{1/p} + H_t^2 W_p\left( \mu, \nu \right).
\end{align*}
Then, we get
\begin{equation*}
    \e\left( | b(X, \mu) - b(Y, \nu) |^p \right)^{1/p} \leq \varepsilon_t(\mu) + \varepsilon_t(\nu) + H_t^1 \e(|X - Y|^p)^{1/p} + H_t^2 W_p\left( \mu, \nu \right),
\end{equation*}
with $X \sim \mu$ and $Y \sim \nu$. Since the left-hand side of the inequality does not depend on $t$, we may pass to the $\limsup$ on the right-hand side. So we have
\begin{equation}
    \label{lipb}
    \e\left( | b(X, \mu) - b(Y, \nu) |^p \right)^{1/p} \leq \overline{D}^1 \e(|X - Y|^p)^{1/p} +  \overline{D}^2 W_p\left( \mu, \nu \right).
\end{equation}

\paragraph{Step 2. Let us prove that there exists an equilibrium for~\eqref{changing} under our assumption on $b$.}
Let us denote $\Phi_t$ the application mapping $\mu_0$ to $\mu_t$, where $\mu_t$ is the law of the solution to~\eqref{changing} at time $t$, starting from the initial distribution $\mu_0$, and with the drift $b$ for which we determined the regularity property~\eqref{lipb}. Let us take $\mu_0$, $\nu_0 \in \mathcal{P}_p\left( \mathbb{R}^d \right)$, and $X_t$, $Y_t$ solutions to~\eqref{changing} with initial distributions respectively $\mu_0$ and $\nu_0$. We denote $\mu_t$ the law of $X_t$ and $\nu_t$ the law of $Y_t$. Then we have
\begin{align*}
    \frac{\diff}{\diff t} \frac{|X_t - Y_t|^p}{p} =& - |X_t - Y_t|^{p - 2} \left( F(X_t) - F(Y_t) \right) \cdot \left( X_t - Y_t \right) + |X_t - Y_t|^{p - 2} \left( b(X_t, \mu_t) - b(Y_t, \nu_t) \right) \cdot \left( X_t - Y_t \right) \\
    \leq& - p \kappa \frac{|X_t - Y_t|^p}{p} + |X_t - Y_t|^{p - 1} \, |b(X_t, \mu_t) - b(Y_t, \nu_t)|.  
\end{align*}
Let us now take the expectation of the previous inequality and apply Hölder inequality with $\frac{1}{p} + \frac{p - 1}{p} = 1$, and then use~\eqref{lipb}.
\begin{align*}
    \frac{\diff}{\diff t} \frac{\e|X_t - Y_t|^p}{p}
    \leq& - \kappa \e(|X_t - Y_t|^p) + \e(|X_t - Y_t|^p)^{\frac{p - 1}{p}} \e\left(|b(X_t, \mu_t) - b(Y_t, \mu_t)|^p\right)^{1/p} \\
    \leq& - \kappa \e(|X_t - Y_t|^p) + \overline{D}^1 \, \e(|X_t - Y_t|^p) + \overline{D}^2 \, W_p(\mu, \nu) \e\left(|X_t - Y_t|^p\right)^{\frac{p - 1}{p}}.  
\end{align*}
Let us then apply Young inequality with $\frac{1}{p} + \frac{p - 1}{p} = 1$ and using $W_p^p(\mu_t, \nu_t) \leq \e(|X_t - Y_t|^p)$. We get
\begin{align*}
    \frac{\diff}{\diff t} \frac{\e|X_t - Y_t|^p}{p} \leq& - \kappa \e(|X_t - Y_t|^p) + \overline{D}^1 \e(|X_t - Y_t|^p) + \overline{D}^2 \frac{1}{p} W_p^p(\mu_t, \nu_t) + \overline{D}^2 \frac{(p - 1)}{p} \e(|X_t - Y_t|^p) \\
    \leq& \left( - \kappa + \overline{D}^1 + \overline{D}^2 \right) \e(|X_t - Y_t|^p).
\end{align*}
Denote $\lambda := \kappa - \overline{D}^1 - \overline{D}^2$. By Gronwall's lemma, we get that
\begin{equation*}
    \e(|X_t - Y_t|^p) \leq e^{- \lambda t} \e(|X_0 - Y_0|^p).
\end{equation*}

As in the proof of Corollary~\ref{equilibrium}, this yields the existence and uniqueness of an equilibrium for~\eqref{changing}, thanks to a semi-group argument. Denote it $\mu_\infty$.

\paragraph{Step 3. We compare the solutions to~\eqref{fullmemory} and~\eqref{changing}, both starting at $\mu_\infty$.}
Let $X_t$ be a solution to~\eqref{fullmemory} with initial distribution $\mu_\infty$, of law denoted $\nu_t$, and $Y_t$ be a solution to~\eqref{changing} with initial distribution $\mu_\infty$. We know by the previous step that $Law(Y_t) = \mu_\infty$, for all $t \geq 0$. By parallel coupling, Assumption~\ref{assumparticles} and Cauchy-Schwartz inequality, we get that
\begin{align*}
    \frac{\diff}{\diff t} & \frac{\left| X_t - Y_t \right|^p}{p} \\
    &= - \left| X_t - Y_t \right|^{p - 2} \left( X_t - Y_t \right) \cdot \left( F(X_t) - F(Y_t) \right) + \left| X_t - Y_t \right|^{p - 2} \left( X_t - Y_t \right) \cdot \left[ \int_0^t L(t, s, X_t, \nu_s) \; u_t(\mathrm{d} s) - b(Y_t, \mu_\infty) \right] \\
    &\leq - \kappa \left| X_t - Y_t \right|^p + \left| X_t - Y_t \right|^{p - 1} \left| \int_0^t L(t, s, X_t, \nu_s) \; u_t(\mathrm{d} s) - b(Y_t, \mu_\infty) \right|.
\end{align*}
By taking the expectation and using Hölder inequality, we get
\begin{equation}
\label{eqHolder}
    \frac{\diff}{\diff t} \frac{\e\left| X_t - Y_t \right|^p}{p} \leq - \kappa \e\left| X_t - Y_t \right|^p + \e(\left| X_t - Y_t \right|^p)^{\frac{p - 1}{p}} \e\left(\left| \int_0^t L(t, s, X_t, \nu_s) \; u_t(\mathrm{d} s) - b(Y_t, \mu_\infty) \right|^p\right)^{1/p}.
\end{equation}
Let us study more specifically the term $\e\left(\left| \int_0^t L(t, s, X_t, \nu_s) \; u_t(\mathrm{d} s) - b(Y_t, \mu_\infty) \right|^p\right)^{1/p}$. We have by Minkowski inequality
\begin{align*}
    \e&\left(\left| \int_0^t L(t, s, X_t, \nu_s) \; u_t(\mathrm{d} s) - b(Y_t, \mu_\infty) \right|^p\right)^{1/p} \\
    &\leq \e\left(\left| \int_0^t \left[ L(t, s, X_t, \nu_s) - L(t, s, Y_t, \mu_\infty) \right] \; u_t(\mathrm{d} s) \right|^p\right)^{1/p} + \e\left(\left| \int_0^t L(t, s, X_t, \mu_\infty) \; u_t(\mathrm{d} s) - b(X_t, \mu_\infty) \right|^p\right)^{1/p} \\
    &\leq \e\left(\left| \int_0^t \left[ h^1(t, s) | X_t - Y_t | + h^2(t, s) W_p(\mu_\infty, \nu_s) \right] \; u_t(\mathrm{d} s) \right|^p\right)^{1/p} + \varepsilon_t(\mu_\infty).
\end{align*}
By applying Minkowski another time and then Jensen inequality, we get that
\begin{align*}
    \e&\left(\left| b(X_t, \mu_\infty) - \int_0^t L(t, s, Y_t, \nu_s) \; u_t(\mathrm{d} s) \right|^p\right)^{1/p} \\
    &\leq \varepsilon_t(\mu_\infty) + H_t^1 \e(|X_t - Y_t|^p)^{1/p} + \left| \int_0^t h^2(t, s) W_p(\mu_\infty, \nu_s) \; u_t(\mathrm{d} s) \right| \\
    &\leq \varepsilon_t(\mu_\infty) + H_t^1 \e(|X_t - Y_t|^p)^{1/p} + H_t^2 \left( \int_0^t \frac{h^2(t, s)}{\int_0^t h^2(t, w) \; u_t(\mathrm{d} w)} W_p^p(\mu_\infty, \nu_s) \; u_t(\mathrm{d} s) \right)^{1/p}.
\end{align*}
If we plug this into inequality~\eqref{eqHolder}, and apply Young inequality with $\frac{1}{p} + \frac{p - 1}{p} = 1$ and parameter $\alpha > 0$, we have
\begin{align*}
    \frac{\diff}{\diff t} \frac{\e\left| X_t - Y_t \right|^p}{p} \leq& - \kappa \e(\left| X_t - Y_t \right|^p) + \varepsilon_t(\mu_\infty) \e(\left| X_t - Y_t \right|^p)^{\frac{p - 1}{p}} + H_t^1 \e(|X_t - Y_t|^p) \\
    &+ H_t^2 \left( \int_0^t \frac{h^2(t, s)}{\int_0^t h^2(t, v) \; u_t(\mathrm{d} v)} W_p^p(\mu_\infty, \nu_s) \; u_t(\mathrm{d} s) \right)^{1/p} \e(\left| X_t - Y_t \right|^p)^{\frac{p - 1}{p}} \\
    \leq& - \kappa \e(\left| X_t - Y_t \right|^p) + \frac{\varepsilon_t(\mu_\infty)^p}{p \alpha^{p - 1}} + \frac{\alpha (p - 1)}{p}  \e(\left| X_t - Y_t \right|^p) + H_t^1 \e(\left| X_t - Y_t \right|^p) \\
    &+ \frac{H_t^2}{p \beta^{p - 1}} \int_0^t \frac{h^2(t, s)}{\int_0^t h^2(t, v) \; u_t(\mathrm{d} v)} W_p^p(\mu_\infty, \nu_s) \; u_t(\mathrm{d} s) + H_t^2 \frac{\beta (p - 1)}{p} \e(\left| X_t - Y_t \right|^p).
\end{align*}
Let us rearrange the terms. We get, by using the inequality $W_p^p(\mu_\infty, \nu_s) \leq \e|X_s - Y_s|^p$, that for all $t \geq 0$ 
\begin{align*}
    \frac{\diff}{\diff t} \frac{\e\left| X_t - Y_t \right|^p}{p} \leq& \frac{\varepsilon_t(\mu_\infty)^p}{p \alpha^{p - 1}} + p \left[ - \kappa + \frac{\alpha (p - 1)}{p} + H_t^1 + H_t^2 \frac{\beta (p - 1)}{p} \right] \frac{\e|X_t - Y_t|^p}{p} \\
    &+ \frac{H_t^2}{\beta^{p - 1}} \int_0^t \frac{h^2(t, s)}{\int_0^t h^2(t, v) \; u_t(\mathrm{d} s)} \frac{\e|X_s - Y_s|^p}{p} \; u_t(\mathrm{d} s). 
\end{align*}
Let us fix $D^1 > \overline{D}^1$ and $D^2 > \overline{D}^2$ such that $\kappa > D^1 + D^2$. There exists $t_0$ such that for all $t \geq t_0$, we have
\begin{equation}
\label{asy_}
    H_t^1 \leq D^1 \quad \text{and} \quad H_t^2 \leq D^2.
\end{equation}
Therefore, for all $t \geq t_0$
\begin{align*}
    \frac{\diff}{\diff t} \frac{\e\left| X_t - Y_t \right|^p}{p} \leq& \frac{\varepsilon_t(\mu_\infty)^p}{p \alpha^{p - 1}} + p \left[ - \kappa + \alpha \frac{p - 1}{p} + D^1 + D^2 \frac{\beta (p - 1)}{p} \right] \frac{\e|X_t - Y_t|^p}{p}  \\
    & + \frac{D^2}{\beta^{p - 1}} \int_0^t \frac{h^2(t, s)}{\int_0^t h^2(t, v) \; u_t(\mathrm{d} s)} \frac{\e|X_s - Y_s|^p}{p} \; u_t(\mathrm{d} s). 
\end{align*}
We want the condition 
\begin{equation*}
    \kappa - \alpha \frac{p - 1}{p} - D^1 - D^2 \left( \beta \frac{p - 1}{p} + \frac{1}{p \beta^{p - 1}} \right) > 0
\end{equation*}
to be satisfied in order to be able to apply Lemma~\ref{Lem:decreasingGronwall}. 
The function $\beta \mapsto \beta \frac{p - 1}{p} + \frac{1}{p \beta^{p - 1}}$ is optimal when $\beta = 1$. It is then enough to choose $\alpha$ small enough, since $\kappa > D^1 + D^2$.

Then, by Lemma~\ref{Lem:decreasingGronwall}, with the previous choice of parameters, we get that there exist $m : \mathbb{R}_+ \to \mathbb{R}_+$ a decreasing function, and a certain positive constant $C$ such that, for all $t \geq t_0$
\begin{equation*}
    \e|X_t - Y_t|^p \leq C \varepsilon_t(\mu_\infty) + m(t) \left( \e|X_{t_0} - Y_{t_0}|^p - C \varepsilon_t(\mu_\infty) \right).
\end{equation*}
By optimization on all the couplings between $Law(X_t)$ and $Law(Y_t)$, we get that for all $t \geq t_0$
\begin{equation}
\label{step3_}
    W_p^p(\nu_t, \mu_\infty) \leq C \varepsilon_t(\mu_\infty) + m(t) \left( \e|X_{t_0} - Y_{t_0}|^p - C \varepsilon_t(\mu_\infty) \right).
\end{equation}

\paragraph{Step 4. Let us prove that whatever the initial distribution, solutions to~\eqref{fullmemory} tend to $\mu_\infty$.}
Let us now take $Z_t$ solution to~\eqref{fullmemory} with initial distribution $\mu_0$ and law at time $t$ denoted $\mu_t$. By triangular inequality, and by plugging in~\eqref{step3_} and~\eqref{contractionMFL}, we get that there exists $m : \mathbb{R}_+ \to \mathbb{R}_+$ a decreasing function such that for all $t \geq t_0$
\begin{align*}
    W_p(\mu_t, \mu_\infty) &\leq W_p(\mu_t, \nu_t) + W_p(\nu_t, \mu_\infty) \\
    &\leq m(t)^{1/p} W_p(\mu_0, \mu_\infty) + \left[ C \varepsilon_t(\mu_\infty) + m(t) \left( \e|X_{t_0} - Y_{t_0}|^p - C \varepsilon_t(\mu_\infty) \right) \right]^{1/p} \xrightarrow[t \to + \infty]{} 0.
\end{align*}

\subsection{Proof of Corollary~\ref{exchangelimits}.}
Let us introduce $X_t^{i, N}$ for $i = 1, \dots, N$ solution to the system~\eqref{particles} with initial condition $(X_0)_{i = 1, \dots, N}$, and $\bar{X}_t^{i, N}$ be $N$ independent copies of solutions to~\eqref{limit}, each driven by the same Brownian motions as for the system, and with the same initial conditions.
By triangular inequality we get that
\begin{equation*}
    W_p(Law(X_t^{1, N}) \, , \, \mu_\infty) \leq W_p(Law(X_t^{1, N}) \, , \, Law(\bar{X}_t^{1, N})) + W_p(Law(\bar{X}_t^{1, N}) \, , \, \mu_\infty).
\end{equation*}
By Theorem~\ref{unifpropachaos}, we know 
that $W_p(Law(X_t^{1, N}) \, , \, Law(\bar{X}_t^{1, N}))$ 
can be bounded above uniformly in $t$ by the speed function $v(N)$ introduced in~\eqref{v(N)} to the power $1/p$, up to a multiplicative constant.
By Proposition~\ref{longtime}, we know that $W_p(Law(\bar{X}_t^{1, N}) \, , \, \mu_\infty)$ 
can be bounded above uniformly in $N$ by a function in $t$ that tends to $0$ when $t \to + \infty$.

\section{Appendix}
\label{secproofwellposedness}

We need to prove the well-posedness of both the particles system and the limit equation. Since we will need to perform a parallel coupling, we need one of them to be strongly well-posed. It is indeed the case of both.
\subsection{Proof of well-posedness of the particle system (Proposition~\ref{wellposednesspart})}
We will divide the proof into four steps. \textbf{Step 1}. focuses on uniqueness, while \textbf{Step 2.}, \textbf{Step 3.}, and \textbf{Step 4.} focus on existence. In \textbf{Step 2.}, we define truncated equations for which we need to prove the well-posedness because of the memory term in the drift. Then we define a process by patching the truncated solutions in \textbf{Step 3.}. In \textbf{Step 4.}, we prove that the process defined in \textbf{Step 3.} defines indeed a global solution for our system of particles, thanks to an estimate.
\paragraph{Step 1. Strong uniqueness.}
We first prove uniqueness for~\eqref{particles}. Let $X_t^i$ and $Y_t^i$, $i=1, \dots, N$ be two a.s. continuous adapted processes that are strong solutions to~\eqref{particles} with the same initial condition and same Brownian motions. Denote $\tau_n := \inf \{ t, \, \exists i \in \llbracket 1, N \rrbracket, |X_t^i| \geq n \text{ or } |Y_t^i| \geq n \}$. By parallel coupling, $t \mapsto X_t^i - Y_t^i$ obeys to an ODE because the diffusion terms cancel out. Therefore, for $p > 1$, we have
\begin{align*}
    &\frac{1}{N}  \sum_{i = 1}^N \frac{|X_{T \wedge \tau_n}^i - Y_{T \wedge \tau_n}^i|^p}{p} \\
    =& - \frac{1}{N} \sum_{i = 1}^N \int_0^{T \wedge \tau_n} |X_t^i - Y_t^i|^{p-2} \left( X_t^i - Y_t^i \right) \cdot \left( F(X_t^i) - F(Y_t^i) \right) \diff t \\
    &+ \frac{1}{N} \sum_{i = 1}^N \int_0^{T \wedge \tau_n} |X_t^i - Y_t^i|^{p-2} \left( X_t^i - Y_t^i \right) \cdot \left[ \int_0^t L\Huge(t, s, X_t^i, \frac{1}{N} \sum_{j=1}^N \delta_{X_s^j} \Huge) u_t(\mathrm{d} s) - \int_0^t L\Huge(t, s, Y_t^i, \frac{1}{N} \sum_{j=1}^N \delta_{Y_s^j} \Huge) u_t(\mathrm{d} s) \right] \diff t.
\end{align*} 
\begin{remark}
For $p = 1$, we would need a classical regularization trick that we will not perform here. Refer to~\cite{FournierJourdain} for an example of derivation of this trick.
\end{remark}
Let us treat the first term later, and start by the second term. By using Cauchy-Schwarz inequality and Assumption~\eqref{memory}, we get
\begin{align*}
    \frac{1}{N} &\sum_{i = 1}^N |X_t^i - Y_t^i|^{p-2} \left( X_t^i - Y_t^i \right) \cdot \left[ \int_0^t L\Huge(t, s, X_t^i, \frac{1}{N} \sum_{j=1}^N \delta_{X_s^j} \Huge) u_t(\mathrm{d} s) - \int_0^t L\Huge(t, s, Y_t^i, \frac{1}{N} \sum_{j=1}^N \delta_{Y_s^j} \Huge) u_t(\mathrm{d} s) \right] \\
    \leq& \frac{1}{N} \sum_{i = 1}^N |X_t^i - Y_t^i|^{p-1} \int_0^t \left| L\Huge(t, s, X_t^i, \frac{1}{N} \sum_{j=1}^N \delta_{X_s^j} \Huge) - L\Huge(t, s, Y_t^i, \frac{1}{N} \sum_{j=1}^N \delta_{Y_s^j} \Huge) \right| u_t(\mathrm{d} s) \\
    \leq& \frac{1}{N} \sum_{i = 1}^N |X_t^i - Y_t^i|^{p-1} \int_0^t \left[ h^1(t, s) |X_t^i - Y_t^i| + h^2(t, s) W_p\left( \frac{1}{N} \sum_{j=1}^N \delta_{X_s^j} \, , \, \frac{1}{N} \sum_{j=1}^N \delta_{Y_s^j} \right) \right] u_t(\mathrm{d} s).
\end{align*}
Then, by using Young inequality with $\frac{p-1}{p} + \frac{1}{p} = 1$, and Assumption~\ref{assumparticles} and the the notation introduced in~\eqref{D_T}, we get
\begin{align}
\label{eq}
\nonumber
    \frac{1}{N} &\sum_{i = 1}^N |X_t^i - Y_t^i|^{p-2} \left( X_t^i - Y_t^i \right) \cdot \left[ \int_0^t L\Huge(t, s, X_t^i, \frac{1}{N} \sum_{j=1}^N \delta_{X_s^j} \Huge) u_t(\mathrm{d} s) - \int_0^t L\Huge(t, s, Y_t^i, \frac{1}{N} \sum_{j=1}^N \delta_{Y_s^j} \Huge) u_t(\mathrm{d} s) \right] \\
    \leq& \frac{D_T}{N} \sum_{i = 1}^N |X_t^i - Y_t^i|^p + \frac{1}{N} \sum_{i = 1}^N \int_0^t h^2(t, s) \left[ \frac{p-1}{p}|X_t^i - Y_t^i|^p + \frac{1}{p} W_p^p\left( \frac{1}{N} \sum_{j=1}^N \delta_{X_s^j} \, , \, \frac{1}{N} \sum_{j=1}^N \delta_{Y_s^j} \right) \right] u_t(\mathrm{d} s).
\end{align}
By taking the expectation, plugging in~\eqref{eq}, using~\eqref{pot}, and the fact that 
$\e \left[ W_p^p\left(\frac{1}{N} \sum_{j = 1}^N \delta_{X_s^j} \, , \, \frac{1}{N} \sum_{j = 1}^N \delta_{Y_s^j}\right) \right] \leq \frac{1}{N} \sum_{i=1}^N \e(|X_s^i - Y_s^i|^p)$, we get that for all $0 \leq t \leq T$
\begin{align*}
    \frac{1}{N} \sum_{i = 1}^N \frac{\e|X_{T \wedge \tau_n}^i - Y_{T \wedge \tau_n}^i|^p}{p} &\leq \left( |\kappa| + D_T\left(1 + \frac{p-1}{p}\right)\right) \frac{1}{N} \sum_{i = 1}^N \e\left[ \int_0^{T \wedge \tau_n} |X_t^i - Y_t^i|^p \diff t \right] \\ 
    &+ \frac{1}{pN} \sum_{i = 1}^N \e \left[ \int_0^{T \wedge \tau_n} \int_0^t h^2(t, s) \, |X_s^i - Y_s^i|^p u_t(\mathrm{d} s) \diff t \right] \\
    \leq& \left( |\kappa| + D_T\left(1 + \frac{p-1}{p}\right)\right) \frac{1}{N} \sum_{i = 1}^N \int_0^T \e|X_{t \wedge \tau_n}^i - Y_{t \wedge \tau_n}^i|^p \diff t \\ 
    &+ \frac{1}{pN} \sum_{i = 1}^N \int_0^T \int_0^t h^2(t, s) \, \e |X_{s \wedge \tau_n}^i - Y_{s \wedge \tau_n}^i|^p u_t(\mathrm{d} s) \diff t.
\end{align*} 
Define 
\begin{equation*}
    A(t) := \sup_{0 \leq s \leq t} \frac{1}{N} \sum_{i = 1}^N \e|X_{s \wedge \tau_n}^i - Y_{s \wedge \tau_n}^i|^p.
\end{equation*}
Then for all $T \leq \Tilde{T} \leq T'$
\begin{align*}
    \frac{1}{N} \sum_{i = 1}^N \frac{\e|X_{T \wedge \tau_n}^i - Y_{T \wedge \tau_n}^i|^p}{p} &\leq \left( |\kappa| + 2 D_T \right) \int_0^T A(t) \diff t \\
    &\leq \left( |\kappa| + 2 D_{T'} \right) \int_0^{\Tilde{T}} A(t) \diff t.
\end{align*}
By taking the supremum in $T \in [0, \Tilde{T}]$ in the left-hand side, we get that
\begin{align*}
    A(\Tilde{T}) \leq \left( |\kappa| + 2 D_{T'} \right) \int_0^{\Tilde{T}} A(t) \diff t.
\end{align*} 
Gronwall's lemma concludes that $X_{\cdot \wedge \tau_n}$ and $Y_{\cdot \wedge \tau_n}$ are indistinguishable, because we are working with a.s. continuous processes. Almost surely, both processes considered are continuous, so they are bounded on every compact $[0, T]$. Therefore, a.s. $\tau_n > T$, for $n$ large enough. This is true for $T > 0$ arbitrarily large, so that $\tau_n \xrightarrow[n \to + \infty]{\p} + \infty$. As a consequence, $\tau_n$ tends to $+ \infty$ a.s, up to a subsequence, and this concludes for the global strong uniqueness. Note that no condition on the parameters is needed.
\bigbreak

For the existence part, the classical route using Girsanov Theorem for weak existence and then Yamada-Watanabe principle to deduce strong existence would not work, since $\sigma$ can be a singular matrix. Let us proceed by using Banach-Picard fixed-point theorem. There are two difficulties to focus on. First, the memory term in the drift makes it impossible to use classical theorems, so we will need to manipulate a moment of order $p$. Plus, the external force $F$ is only locally Lipschitz, so we will need to localize.

\paragraph{Step 2. Prove the existence of a solution. To this end, let us first define truncated equations for which we need to prove the well-posedness.}
Define a truncated version of the force
\begin{equation}
\label{truncatedforce}
      F_n(x) = \left\{ \begin{array}{rl}
         &   F(x), \quad \text{ if $|x| \leq n$}\\
         &   F\left( \frac{nx}{|x|} \right), \quad \text{ if $|x| > n$}.
    \end{array} \right.
\end{equation} 
Then $F_n$ is globally Lipschitz, let us say with Lipschitz constant $L_n$. We define for every $n \in \mathbb{N}^*$
\begin{equation}
\label{EDSn}
    \diff X_t^{n, i} = \sigma \diff B_t^i - F_n(X_t^{n, i}) \diff t + \int_0^t  L\left(t, s, X_t^{n, i}, \frac{1}{N} \sum_{j = 1}^N \delta_{X_s^{n, j}}\right) u_t(\mathrm{d} s) \diff t, \quad (X_0^{n, 1}, \dots, X_0^{n, N}) \sim \mu.
\end{equation}
\begin{remark}
    Theorem $11.2$ of~\cite{Rogers_Williams_2000} wouldn't be applicable in our case to prove that ~\eqref{EDSn} is well posed because the drift term is $\int_0^t b(s, x_\cdot) \diff s$, which is time-homogeneous.
\end{remark}

In order to use Picard iteration argument, let us define on $\mathbb{L}^p \left( \Omega, \mathcal{C}\left( [0, T],\mathbb{R}^{d N} \right) \right)$ a functional $\Phi_n$, by using the weak formulation
\begin{equation*}
    \Phi_n(X)_t^i := \xi^i + \sigma B_t^i - \int_0^t F_n(X_s^i) \diff s + \int_0^t \int_0^s L(s, v, X_s^i, \frac{1}{N} \sum_{i=1}^N \delta_{X_v^j}) \, u_s(\mathrm{d} v) \diff s,
\end{equation*}
for all $0 \leq t \leq T$, if $\xi^i$ are $\mathcal{F}_0$-measurable distributed with law $\mu_i \in \mathcal{P}_m\left( \mathbb{R}^d \right)$. This expression is well defined for adapted processes. Let us fix $n$. Let $X$ and $Y \in \mathbb{L}^p \left( \Omega, \mathcal{C}\left( [0, T],\mathbb{R}^{d N} \right) \right)$ be two adapted processes. Then for every $i = 1, \dots, N$
\begin{align*}
    |\Phi_n(X)_T^i - \Phi_n(Y)_T^i|^p \leq& 3^{p-1} \left| \int_0^T (  F_n(X_t^i) -   F_n(Y_t^i)) \diff t \right|^p \\
    &+ 3^{p-1} \left| \int_0^T \int_0^t \left(L(t, s, X_t^i, \frac{1}{N} \sum_{j=1}^N \delta_{X_s^j}) - L(t, s, Y_t^i, \frac{1}{N} \sum_{j=1}^N \delta_{X_s^j})\right) \, u_t(\mathrm{d} s) \diff t \right|^p \\
    &+ 3^{p-1} \left| \int_0^T \int_0^t \left(L(t, s, Y_t^i, \frac{1}{N} \sum_{j=1}^N \delta_{X_s^j}) - L(t, s, Y_t^i, \frac{1}{N} \sum_{j=1}^N \delta_{Y_s^j})\right) \, u_t(\mathrm{d} s) \diff t \right|^p \\
    \leq& 3^{p-1} T^{p-1} \int_0^T \left|  F_n(X_t^i) -   F_n(Y_t^i)\right|^p \diff t \\
    &+ 3^{p-1} T^{p-1} \int_0^T \left( \int_0^t \left|L(t, s, X_t^i, \frac{1}{N} \sum_{j=1}^N \delta_{X_s^j}) - L(t, s, Y_t^i, \frac{1}{N} \sum_{j=1}^N \delta_{X_s^j})\right| \, u_t(\mathrm{d} s) \right)^p \diff t \\
    &+ 3^{p-1} T^{p - 1} \int_0^T \left( \int_0^t \left|L(t, s, Y_t^i, \frac{1}{N} \sum_{j=1}^N \delta_{X_s^j}) - L(t, s, Y_t^i, \frac{1}{N} \sum_{j=1}^N \delta_{Y_s^{n j}})\right| \, u_t(\mathrm{d} s) \right)^p \diff t,
\end{align*}
Therefore, by denoting $M_n$ the Lipschitz constant of $F$ on $\{ x \in \mathbb{R}^d, \, |x| \leq n \}$, we get by Jensen inequality that for all $i = 1, \dots, N$
\begin{align*}
    |\Phi_n(X)_t^i - \Phi_n(Y)_t^i|^p \leq& 3^{p-1} T^p M_n^p \int_0^T \e|X_t^i - Y_t^i|^p \diff t \\
    &+ 3^{p-1} T^{p - 1} \int_0^T \left( \int_0^t h^1(t, s) \left|X_t^i - Y_t^i\right| \, u_t(\mathrm{d} s) \right)^p \diff t \\
    &+ 3^{p-1} T^{p - 1} D_T^p \int_0^T \left( \int_0^t \frac{h^2(t, s)}{\int_0^t h^2(t, s) \, u_t(\mathrm{d} s)} W_p(\frac{1}{N} \sum_{j=1}^N \delta_{X_s^j} \, , \, \frac{1}{N} \sum_{j=1}^N \delta_{Y_s^j}) \, u_t(\mathrm{d} s) \right)^p \diff t \\
    \leq& 3^{p-1} T^p M_n^p \int_0^T |X_t^i - Y_t^i|^p \diff t \\
    &+ 3^{p-1} T^{p - 1} D_T^p \int_0^T \left|X_t^i - Y_t^i\right|^p \, u_t(\mathrm{d} s) \diff t \\
    &+ 3^{p-1} T^{p - 1} D_T^{p - 1} \int_0^T \int_0^t h^2(t, s) W_p^p(\frac{1}{N} \sum_{j=1}^N \delta_{X_s^j} \, , \, \frac{1}{N} \sum_{j=1}^N \delta_{Y_s^j}) \, u_t(\mathrm{d} s) \diff t.
\end{align*}
By using the fact that a.s, $W_p^p(\frac{1}{N} \sum_{j=1}^N \delta_{X_s^j} \, , \, \frac{1}{N} \sum_{j=1}^N \delta_{Y_s^j}) \leq \frac{1}{N} \sum_{i = 1}^N |X_s^i - Y_s^i|^p$, by bounding above the infimum in the definition of the Wasserstein distance by a term with the coupling $\pi := \frac{1}{N} \sum_{i = 1}^N \delta_{(X_s^i, Y_s^i)}$, and taking the supremum in time in the right-hand term, we get that for all $0 \leq t \leq T$
\begin{align*}
    |\Phi_n(X)_t^i - \Phi_n(Y)_t^i|^p \leq& 3^{p-1} T^p M_n^p \int_0^T \sup_{0 \leq s \leq t} |X_s^i - Y_s^i|^p \diff t \\
    &+ 3^{p-1} T^{p - 1} D_T^p \int_0^T \sup_{0 \leq s \leq t}  \left|X_s^i - Y_s^i\right|^p \diff t \\
    &+ 3^{p-1} T^{p - 1} D_T^p \frac{1}{N} \sum_{i = 1}^N \int_0^T \sup_{0 \leq s \leq t} |X_s^i - Y_s^i|^p \diff t.
\end{align*}
We have such an inequality for each $i = 1, \dots, N$, so by summing over $i$ and taking the supremum in the left-hand side and immediately after the expectation, we get 
\begin{equation*}
    \frac{1}{N} \sum_{i=1}^N \e\left[ \sup_{0 \leq t \leq T} |\Phi_n(X)_t^i - \Phi_n(Y)_t^i|^p \right] \leq K_T \frac{1}{N} \sum_{i=1}^N \int_0^T \e\left[ \sup_{0 \leq s \leq t} |X_s^i - Y_s^i|^p \right] \diff t,
\end{equation*}
with $K_T := 3^{p-1} T^{p - 1} \left( T M_n^p + 2 D_T^p \right)$.
Start from an a.s. continuous adapted process $X$ to define a sequence of approximations of the solution to~\eqref{EDSn}. All of the iterations are well defined by induction because each one is adapted. The iterations are also a.s. continuous, because of the regularity assumptions on the external force $F$ and of $t \mapsto \int_0^t L(t, s, x, \mu) \, u_t(\mathrm{d} s)$ (see Assumption~\ref{assumparticles}). By a classical iteration argument, we get that 
\begin{equation*}
    \frac{1}{N} \e\left[ \sum_{i=1}^N \sup_{0 \leq t \leq T} |\Phi_n^{k+1}(X)_t^i - \Phi_n^k(X)_t^i|^p \right] \leq \frac{K_T^k T^k}{k !} \frac{1}{N} \sum_{i=1}^N \e\left[ \sup_{0 \leq t \leq T} |\Phi_n(X)_t^i - X_t^i|^p \right],
\end{equation*}
so that $(\Phi_n^k(X))_k$ is a Cauchy sequence that converges in the complete metric space $\mathbb{L}^p \left( \Omega, \mathcal{C}\left( [0, T],\mathbb{R}^{d N} \right) \right)$, therefore it converges. The limit $X$ is a.s. continuous, and adapted to the completed filtration, because $X_t$ is an a.s. limit (up to a subsequence) of $\mathcal{F}_t$-measurable random variables. The process $X$ is also solution to~\eqref{particles}, since it satisfies its integral form, by dominated convergence.

The limit is therefore continuous and adapted, so it is by definition a solution to~\eqref{EDSn}. Note that the strong uniqueness is immediate from the Picard argument, since the $\mathbb{L}^p$ norm controls the a.s. convergence.

\paragraph{Step 3. Define a candidate for the global solution.}
From the previous step, we get the existence of $X^n$ solution to~\eqref{EDSn}. Set $\tau_n := \inf\{ t, \, \exists i, \, |X_t^{n, i}| \geq n\}$. 
For $k \leq n$, the process $X^n$ stopped at the time $\inf\{ t, \, \exists i, \, |X_t^{n, i}| \geq k\}$ is solution to~\eqref{EDSn} with $k$, so by uniqueness proved in the previous step, it coincides with $X^k$ until $\tau_k$. Define $X_t^i := X_t^{n , i}$ for every $t \in [0, \tau_n]$. We then need to check that $\tau_n \xrightarrow[n \to +\infty]{a.s.} +\infty$.

\paragraph{Step 4. Verify that the patching of local solutions recovers a global solution with uniform (in localization parameter $n$) bound of moments.}
In other words, our goal is to prove that $\tau_n \xrightarrow[n \to +\infty]{a.s.} +\infty$. For this, we need an estimate on $\sup_{0 \leq t \leq T} \e|X_{t \wedge \tau_n}^{n, i}|^q$, for a suitable $q$. Let us apply Itô formula with $x \mapsto |x|^m \in \mathcal{C}^2$ (because $m \geq 2$). We have
\begin{align*}
|X_{T \wedge \tau_n}^{n, i}|^m =& |\xi^i|^m - \int_0^{T \wedge \tau_n} m |X_t^{n, i}|^{m-2} \ X_t^{n, i} \cdot \left( F(X_t^{n, i}) -   F(0) \right) \diff t - \int_0^{T \wedge \tau_n} m |X_t^{n, i}|^{m-2} \ X_t^{n, i} \cdot F(0) \ \diff t \\
&+ \int_0^{T \wedge \tau_n} m |X_t^{n, i}|^{m-2} \ X_t^{n, i} \cdot \int_0^t \left[ L(t, s, X_t^{n, i}, \frac{1}{N} \sum_{j=1}^N \delta_{X_s^{n, j}}) - L(t, s, X_t^{n, i}, \delta_0) \right] \, u_t(\mathrm{d} s) \diff t \\
&+ \int_0^{T \wedge \tau_n} m |X_t^{n, i}|^{m-2} \ X_t^{n, i} \cdot \int_0^t \left[ L(t, s, X_t^{n, i}, \delta_0) - L(t, s, 0, \delta_0) \right] \, u_t(\mathrm{d} s) \diff t \\
&+ \int_0^{T \wedge \tau_n} m |X_t^{n, i}|^{m-2} \ X_t^{n, i} \cdot \int_0^t L(t, s, 0, \delta_0) \, u_t(\mathrm{d} s) \diff t \\
&+ \int_0^{T \wedge \tau_n} m |X_t^{n, i}|^{m-2} \ X_t^{n, i} \cdot \sigma \diff B_t^i \\
&+ \frac{1}{2} \int_0^{T \wedge \tau_n} \sum_{k, l = 1}^d \left[ m(m-2) |X_t^{n, i}|^{m-4} (X_t^{n, i})_k (X_t^{n, i})_j + m |X_t^{n, i}|^{m-2} \delta_{k, l}\right] \diff [(X_t^{n, i})_k, (X_t^{n, i})_l] \\
\leq& |\xi^i|^m - \kappa \int_0^T m |X_{t \wedge \tau_n}^{n, i}|^m \diff t + \int_0^T m |X_{t \wedge \tau_n}^{n, i}|^{m-1} |F(0)| \diff t \\
&+ \int_0^T m |X_{t \wedge \tau_n}^{n, i}|^{m-1} \int_0^t h^2(t, s) W_p(\delta_0, \frac{1}{N} \sum_{i=1}^N \delta_{X_{s \wedge \tau_n}^{n, j}}) u_{t \wedge \tau_n}(\mathrm{d} s) \diff t \\
&+ \int_0^T m |X_{t \wedge \tau_n}^{n, i}|^{m-1} \int_0^t \left| L(t, s, X_{t \wedge \tau_n}^{n, i}, \delta_0) - L(t, s, 0, \delta_0) \right| u_t(\mathrm{d} s) \diff t \\
&+ \int_0^T m |X_{t \wedge \tau_n}^{n, i}|^{m-1} \left| \int_0^t L(t, s, 0, \delta_0) \, u_t(\mathrm{d} s) \right| \diff t \\
&+ \int_0^T m |X_{t \wedge \tau_n}^{n, i}|^{m-2} \ X_{t \wedge \tau_n}^{n, i} \cdot \sigma \diff B_{t \wedge \tau_n}^i \\
&+ \frac{1}{2} \int_0^T \left[ m(m-2) |X_{t \wedge \tau_n}^{n, i}|^{m-4} |^t \sigma X_{t \wedge \tau_n}^{n, i}|^2 + m |X_{t \wedge \tau_n}^{n, i}|^{m-2} \mbox{tr}(\sigma ^t \sigma ) \right] \diff t,
\end{align*}
because for $0 \leq t \leq T \wedge \tau_n$, $t \wedge \tau_n = t$, and for $0 \leq s \leq t$, $s \wedge \tau_n = s$ as well.
Let us take the expectation, using that the term $\displaystyle \int_0^t |X_{s \wedge \tau_n}^{n, i}|^{m-2} \ X_{s \wedge \tau_n}^{n, i} \cdot \sigma \diff B_s^i$ is an $\mathbb{L}^2$-martingale, because the process $(X_{t \wedge \tau_n}^{n, i})_t$ is bounded by $n$, use the inequality $W_p\left( \delta_0, \frac{1}{N} \sum_{i = 1}^N \delta_{X_{s \wedge \tau_n}^{n, j}} \right) \leq \left( \frac{1}{N} \sum_{j=1}^N |X_{s \wedge \tau_n}^{n, j}|^m \right)^{1/m}$ that uses $p \leq m$, and apply Young inequality with $\frac{m - 1}{m} + \frac{1}{m} = 1$. We get
\begin{align*}
    \e|X_{T \wedge \tau_n}^{n, i}|^m \leq& \e|\xi^i|^m + \int_0^T m | \kappa | \e|X_{t \wedge \tau_n}^{n, i}|^m \diff t + \int_0^T (m - 1) \e|X_{t \wedge \tau_n}^{n, i}|^m \diff t + T | F(0)|^m \\
    &+ \e \int_0^T m |X_{t \wedge \tau_n}^{n, i}|^{m-1} \int_0^t h^2(t, s) \left( \frac{1}{N} \sum_{j=1}^N |X_{s \wedge \tau_n}^{n, j}|^m \right)^{1/m} u_t(\mathrm{d} s) \diff t \\
    &+ \int_0^T m H_t^1 \e|X_{t \wedge \tau_n}^{n, i}|^m \diff t \\
    &+ (m - 1) \int_0^T \e|X_{t \wedge \tau_n}^{n, i}|^m \diff t + \int_0^T \left( \int_0^t L(t, s, 0, \delta_0) \, u_t(\mathrm{d} s) \right)^m \diff t \\
    &+ \frac{1}{2} \int_0^T \e \left[ m(m-2) |X_{t \wedge \tau_n}^{n, i}|^{m-4} \cdot \|^t \sigma \|^2 | X_{t \wedge \tau_n}^{n, i}|^2 + m \mbox{tr}(\sigma ^t \sigma ) |X_{t \wedge \tau_n}^{n, i}|^{m-2} \right] \diff t \\
    \leq& \e|\xi^i|^m + (m | \kappa | + m - 1) \int_0^T \e|X_{t \wedge \tau_n}^{n, i}|^m \diff t + T |F(0)|^m \\
    &+ (m - 1) D_T \int_0^T \e|X_{t \wedge \tau_n}^{n, i}|^m \diff t + \frac{1}{N} \sum_{j=1}^N \int_0^T \int_0^t h^2(t, s) \e|X_{s \wedge \tau_n}^{n, j}|^m \, u_t(\mathrm{d} s) \diff t\\
    &+ m D_T \int_0^T \e|X_{t \wedge \tau_n}^{n, i}|^m \diff t \\
    &+ (m - 1) \int_0^T \e|X_{t \wedge \tau_n}^{n, i}|^m \diff t + \int_0^T \left( \int_0^t L(t, s, 0, \delta_0) \, u_t(\mathrm{d} s) \right)^m \diff t \\
    &+ \frac{1}{2} \int_0^T \left[ m(m-2) \e|X_{t \wedge \tau_n}^{n, i}|^{m-2} \cdot \|^t \sigma \|^2 + m \, \mbox{tr}(\sigma ^t \sigma ) \e|X_{t \wedge \tau_n}^{n, i}|^{m-2} \right] \diff t.
\end{align*}
Let us sum the previous inequality from $i = 1, \dots, N$. By using Young inequality for $\frac{m - 2}{m} + \frac{2}{m} = 1$ and then taking the supremum in time in the right-hand side, we have 
\begin{align*}
    \sum_{i=1}^N \e|X_{T \wedge \tau_n}^{n, i}|^q \leq& \sum_{i=1}^N \e|\xi^i|^m + N T | F(0)|^m + N \int_0^T \left( \int_0^t L(t, s, 0, \delta_0) \, u_t(\mathrm{d} s) \right)^m \diff t + N T (m - 2) \| ^t \sigma \|^2 + N T \tr(\sigma ^t \sigma) \\
    &+ \left(m | \kappa | + (m - 1)(2 + D_T) + m D_T + \frac{(m - 2)^2 \| ^t \sigma \|^2}{2} + \frac{(m - 2) \, \mbox{tr}(\sigma ^t \sigma)}{2}\right) \sum_{i=1}^N \int_0^T \e|X_{t \wedge \tau_n}^{n, i}|^m \diff t  \\
    &+ \sum_{i=1}^N \int_0^T \int_0^t h^2(t, s) \e|X_{s \wedge \tau_n}^{n, i}|^m \, u_t(\mathrm{d} s) \diff t \\
    \leq& \sum_{i=1}^N \e|\xi^i|^m + N T | F(0)|^m + N \int_0^T \left( \int_0^t L(t, s, 0, \delta_0) \, u_t(\mathrm{d} s) \right)^m \diff t + N T (m - 2) \| ^t \sigma \|^2 + N T \tr(\sigma ^t \sigma) \\
    &+ \left(m | \kappa | + (m - 1)(2 + D_T) + (m + 1) D_T + \frac{(m - 2)^2 \| ^t \sigma \|^2}{2} + \frac{(m - 2) \, \mbox{tr}(\sigma ^t \sigma)}{2} \right) \\
    & \quad \quad \quad \quad \quad \quad \quad \quad \quad \quad \quad \quad \quad \quad \quad \quad \quad \quad \quad \quad \quad \quad \quad \quad \quad \quad \quad \quad \times \sum_{i=1}^N \int_0^T \sup_{0 \leq s \leq t} \e|X_{s \wedge \tau_n}^{n, i}|^m \diff t.
\end{align*}
Therefore, for all $0 \leq t \leq T \leq \Tilde{T}$
\begin{align*}
    \sum_{i=1}^N \e|X_{t \wedge \tau_n}^{n, i}|^m \leq& \sum_{i=1}^N \e|\xi^i|^m + N \Tilde{T} |F(0)|^m + N \int_0^{\Tilde{T}} \left( \int_0^t L(t, s, 0, \delta_0) \, u_t(\mathrm{d} s) \right)^m \diff t + N \Tilde{T} (m - 2) \| ^t \sigma \|^2 + N \Tilde{T} \tr(\sigma ^t \sigma) \\
    &+ \left(m | \kappa | + (m - 1)(2 + D_{\Tilde{T}}) + (m + 1) D_{\Tilde{T}} + \frac{(m - 2)^2 \| ^t \sigma \|^2}{2} + \frac{(m - 2) \, \mbox{tr}(\sigma ^t \sigma)}{2} \right) \\
    & \quad \quad \quad \quad \quad \quad \quad \quad \quad \quad \quad \quad \quad \quad \quad \quad \quad \quad \quad \quad \quad \quad \quad \quad \quad \quad \quad \quad \times \sum_{i=1}^N \int_0^T \sup_{0 \leq s \leq t} \e|X_{s \wedge \tau_n}^{n, i}|^m \diff t,
\end{align*}
so that by taking the supremum for the left-hand side, we get for all $T \leq \Tilde{T}$
\begin{align*}
    \sum_{i=1}^N \sup_{0 \leq t \leq T} \e|X_{t \wedge \tau_n}^{n, i}|^m \leq& \sum_{i=1}^N \e|\xi^i|^m + N \Tilde{T} |F(0)|^m + N \int_0^{\Tilde{T}} \left| \int_0^t L(t, s, 0, \delta_0) \, u_t(\mathrm{d} s) \right|^m \diff t \\
    &+ N \Tilde{T} (m - 2) \| ^t \sigma \|^2 + N \Tilde{T} \tr(\sigma ^t \sigma) \\
    &+ \left(m | \kappa | + (m - 1)(2 + D_{\Tilde{T}}) + (m + 1) D_{\Tilde{T}} + \frac{(m - 2)^2 \| ^t \sigma \|^2}{2} + \frac{(m - 2) \, \mbox{tr}(\sigma ^t \sigma)}{2} \right) \\
    & \quad \quad \quad \quad \quad \quad \quad \quad \quad \quad \quad \quad \quad \quad \quad \quad \quad \quad \quad \quad \quad \quad \quad \quad \times \sum_{i=1}^N \int_0^T \sup_{0 \leq s \leq t} \e|X_{s \wedge \tau_n}^{n, i}|^m \diff t.
\end{align*}
Note that the term $\int_0^T \left| \int_0^t L(t, s, 0, \delta_0) \, u_t(\mathrm{d} s) \right|^m \diff t$ is finite thanks to Assumption~\ref{assumparticles} on the continuity of the integrand. Then by Gronwall's lemma, we have
\begin{equation}
\label{bound}
    \sum_{i=1}^N \sup_{0 \leq t \leq T} \e|X_{t \wedge \tau_n}^{n, i}|^m \leq C(m, T, N),
\end{equation}
with $C(m, T, N)$ a constant depending on $m$, $T$ and $N$.

Since we have, for all $0 \leq t \leq T$ 
\begin{equation*}
    \e\left(|X_{t \wedge \tau_n}^{n, i}|^m \right) \geq \e\left( |X_{t \wedge \tau_n}^{n, i}|^m \un_{\{ \tau_n \leq T \}} \right),
\end{equation*}
then by taking the supremum over $t \in [0, T]$, and summing from $i  = 1, \dots, N$, we get
\begin{align*}
    \sum_{i = 1}^N \sup_{0 \leq t \leq T} \e |X_{t \wedge \tau_n}^{n, i}|^m &\geq \sum_{i = 1}^N \sup_{0 \leq t \leq T} \e \left(|X_{t \wedge \tau_n}^{n, i}|^m \un_{\{ \tau_n \leq T \}}\right) \\
    &\geq n^m \p\left( \tau_n \leq T \right).
\end{align*}
By boundedness of the left hand side term uniformly in $n$, and because $m > 1$, we know that $\sum_n \p\left[ \tau_n \leq T \right] < \infty$. By Borel-Cantelli theorem, we conclude that for all $T > 0$, $\p\left( \underline{\lim} \{ \tau_n \leq T \} \right) = 0$, so that $\tau_n \xrightarrow[n \to + \infty]{a.s.} +\infty$. This concludes the proof of the existence.
\begin{remark}
    Remark that $X_{t \wedge \tau_n}^{n, i} = X_{t \wedge \tau_n} \xrightarrow[n \to + \infty]{a.s.} X_t$ by definition of $X_t$. Therefore by application of Fatou's lemma in~\eqref{bound}, we get that $\e|X_t|^m \leq C(m, T, N)$ for all $0 \leq t \leq T$, and by taking the supremum in time, we finally get
    \begin{equation*}
        \sum_{i=1}^N \sup_{0 \leq t \leq T} \e|X_t|^m \leq C(m, T, N).
    \end{equation*}
\end{remark}

\subsection{Proof of well-posedness of the McKean-Vlasov limit equation (Proposition~\ref{wellposedness})}
In~\cite{theseMilica}, Chapter $2$, section $2.2$, Milica Tomašević shows the existence and uniqueness for strong solutions of systems similar to~\eqref{particles} and~\eqref{limit}, with a slight change since we integrate the memory kernel with respect to the measure $u_t(\mathrm{d} s)$ rather than $\diff s$. We follow the sketch of her proof.
\begin{proof}
    Let $T > 0$. Denote $\mathcal{C}_T = \mathcal{C}([0, T], \mathcal{P}_p\left( \mathbb{R}^d \right))$.
    
    Let us proceed by fixed-point argument. Define the map $\Phi : \mathcal{C}_T \to \mathcal{C}_T$ that associates $m$ to the law of the solution to the following SDE 
    \begin{equation}
    \label{SDEMcKeanVlasov}
        \diff \bar{X}_t = \sigma \diff B_t - F(\bar{X}_t) \diff t + \int_0^t L\left(t, s, \bar{X}_t, m(s) \right) u_t(\mathrm{d} s) \ \diff t, \quad 0 \leq t \leq T,
    \end{equation} 
    with initial condition $\bar{X}_0 \sim p_0 \in \mathcal{P}_q\left( \mathbb{R}^{d} \right) \subset \mathcal{P}_p\left( \mathbb{R}^{d} \right)$ because $q \geq p$. Let us first prove that the map $\Phi$ is well defined.
    
    Let us fix $m \in \mathcal{C}_T$. The SDE~\eqref{SDEMcKeanVlasov} has a unique solution because by Assumption $1$, the drift
    \begin{equation*}
        b(t, x) := - F(x) + \int_0^t L(t, s, x, m(s)) \, u_t(\mathrm{d} s)
    \end{equation*} 
    is continuous in $t$, locally Lipschitz in $x$, and satisfies the local monotonicity condition required in Theorem~$1.1$ of~\cite{LanWusufficientconditionsSDE}.
    
    Let us then prove a contraction property and conclude.
    Let us take $m_1$, $m_2 \in  \mathcal{C}_T$ and define the two processes \[ \bar{X}_t^1 = \bar{X}_0 + \sigma B_t - \int_0^t   F(\bar{X}_s^1) \diff s + \int_0^t \int_0^s L\Huge(s, u, \bar{X}_s^1, m_1(u) \Huge) u_t(\mathrm{d} u) \, \diff t, \quad 0 \leq t \leq T\] and \[\bar{X}_t^2 = \bar{X}_0 + \sigma B_t - \int_0^t   F(\bar{X}_s^2) \diff s + \int_0^t \int_0^s L\Huge(s, u, \bar{X}_s^2, m_2(u)\Huge) u_t(\mathrm{d} s) \ \diff t, \quad 0 \leq  t \leq T\] respectively. We have $Law(\bar{X}_\cdot^1) = \Phi(m_1)$ and $Law(\bar{X}_\cdot^2) = \Phi(m_2)$. Let us prove that $\Phi$ is a contraction with respect to the distance 
    \[
    D(m_1, m_2) := \left( \underset{0 \leq t \leq T}{\sup} W_p^p(m_1(t), m_2(t)).
   \right)^{1/p} 
   \]
   Let us differentiate the power of the gap between $\bar{X}_\cdot^1$ and $\bar{X}_\cdot^2$ and use the Lipschitz hypotheses
    \begin{align*}
        \frac{\diff}{\diff t} \left[ \frac{|\bar{X}_t^1 - \bar{X}_t^2|^p}{p}\right] =& - |\bar{X}_t^1 - \bar{X}_t^2|^{p-2} \left( \bar{X}_t^1 - \bar{X}_t^2 \right) \cdot \left( F(\bar{X}_t^1) -   F(\bar{X}_t^2) \right) \\
        &+ |\bar{X}_t^1 - \bar{X}_t^2|^{p-2} \left( \bar{X}_t^1 - \bar{X}_t^2 \right) \cdot \left[ \int_0^t L(t, s, \bar{X}_t^1, m_1(s)) u_t(\mathrm{d} s) - \int_0^t L(t, s, \bar{X}_t^2, m_2(s)) u_t(\mathrm{d} s) \right] \\
        \leq& |\kappa| \left| \bar{X}_t^1 - \bar{X}_t^2 \right| \\
        &+ |\bar{X}_t^1 - \bar{X}_t^2|^{p-1} \left| \int_0^t L(t, s, \bar{X}_t^1, m_1(s)) u_t(\mathrm{d} s) - \int_0^t L(t, s, \bar{X}_t^2, m_1(s)) u_t(\mathrm{d} s) \right| \\
        &+ |\bar{X}_t^1 - \bar{X}_t^2|^{p-1} \left| \int_0^t L(t, s, \bar{X}_t^2, m_1(s)) u_t(\mathrm{d} s) - \int_0^t L(t, s, \bar{X}_t^2, m_2(s)) u_t(\mathrm{d} s) \right| \\
        \leq& (|\kappa| + 1) |\bar{X}_t^1 - \bar{X}_t^2|^p + |\bar{X}_t^1 - \bar{X}_t^2|^{p-1} D_T \int_0^t \frac{h^2(t, s)}{\int_0^t h^2(t, v) u_t(\mathrm{d} v)} W_p(m_1(s) , m_2(s)) u_t(\mathrm{d} s).
    \end{align*}
    Note that by parallel coupling, $t \mapsto \bar{X}_t^1 - \bar{X}_t^2$ satisfies an ODE because the Brownian motions cancel out. So we only need $x \mapsto |x|^p$ to be $\mathcal{C}^1$ to differentiate. The case $p = 1$ needs to be treated aside by a regularization trick we will not perform here.
    
    By taking the expectation of the previous inequality and applying Young inequality with $\frac{1}{p} + \frac{p-1}{p} = 1$, we get
    \begin{align*}
        \frac{\diff}{\diff t} \left[ \frac{\e|\bar{X}_t^1 - \bar{X}_t^2|^p}{p}\right] &\leq (|\kappa| + 1 + \frac{p-1}{p}) \, \e|\bar{X}_t^1 - \bar{X}_t^2|^p + \frac{1}{p} \left(\int_0^t D_T \frac{h^2(t, s)}{\int_0^t h^2(t, v) u_t(\mathrm{d} v)} W_p(m_1(s) , m_2(s)) u_t(\mathrm{d} s) \right)^p \\
        &\leq (p |\kappa| + 2p - 1) \frac{\e|\bar{X}_t^1 - \bar{X}_t^2|^p}{p} + \frac{D_T^p}{p} \int_0^t \frac{h^2(t, s)}{\int_0^t h^2(t, v) u_t(\mathrm{d} v)}  W_p^p(m_1(s) , m_2(s)) u_t(\mathrm{d} s),
    \end{align*} by using Jensen inequality for the last inequality. Then by taking the supremum in time for the Wasserstein distance, we get \begin{equation*}
        \frac{\diff}{\diff t} \left[ \frac{\e|\bar{X}_t^1 - \bar{X}_t^2|^p}{p}\right] \leq (p |\kappa| + 2p - 1) \frac{\e|\bar{X}_t^1 - \bar{X}_t^2|^p}{p} + \frac{D_T^p}{p} \underset{0 \leq s \leq t}{\sup} W_p^p(m_1(s) , m_2(s)).
    \end{equation*} 
    This way, by Gronwall's lemma, we have
    \begin{equation}
    \label{distancecontrol}
    \frac{\e|\bar{X}_t^1 - \bar{X}_t^2|^p}{p} \leq C_T \underset{0 \leq s \leq t}{\sup} W_p^p(m_1(s) , m_2(s)),\end{equation} 
    with $C_T := \frac{D_T^p}{p} e^{(p |\kappa| + 2p - 1)T}$. By using that $W_p^p\left(Law(\bar{X}_s^1), Law(\bar{X}_s^2)\right) \leq \e|\bar{X}_s^1 - \bar{X}_s^2|^p$, then taking the supremum in time, we get \begin{equation*}
    \underset{0 \leq t \leq T}{\sup} W_p^p\left(Law(\bar{X}_t^1), Law(\bar{X}_t^2)\right) \leq p T C_T \underset{0 \leq t \leq T}{\sup} W_p^p(m_1(u) , m_2(u)) \diff s,
    \end{equation*} 
    so that \[D(\Phi(m_1), \Phi(m_2)) \leq (p T C_T)^{1/p} D(m_1, m_2).\] For $T$ small enough, it means $\Phi : \mathcal{C}_T \to \mathcal{C}_T$ is contracting, so that it admits a unique fixed point.
    
    Start from a given $m \in \mathcal{C}_T$. Iterating the previous integration argument transforms \[
    D(\Phi^{k+1}(m), \Phi^k(m_2)) \leq (p C_T T)^{1/p} D(\Phi^k(m), \Phi^{k-1}(m))
    \] 
    into 
    \[ 
    D(\Phi^{k+1}(m), \Phi^k(m)) \leq p^{1/p} \frac{C_T^{k/p} T^{k/p}}{k !} D(\Phi(m), m),
    \] 
    for all $k \geq 1$. So $(\Phi^k(m))_k$ is a Cauchy sequence that converges because $\mathcal{C}_T$ is complete with respect to the metric $D$. The limit of this sequence is therefore the unique fixed point of $\Phi$. 
    This gives the weak existence and weak uniqueness of the solution to~\eqref{limit} on $[0, T]$.
    
    Let $m^\ast$ be the fixed point of $\phi$. The frozen SDE associated with $m^\ast$ admits a unique strong solution $\bar{X}$. Since $Law(\bar{X}_t) = \phi(m^\ast)(t) = m^\ast(t)$, $\bar{X}$ is a strong solution of~\eqref{limit}.
    
    Then, taking the supremum on time of the inequality~\eqref{distancecontrol}, we have 
    \begin{equation*}
        \underset{0 \leq s \leq t}{\sup} \e|\bar{X}_s^1 - \bar{X}_s^2|^p \leq p T C_T \underset{0 \leq s \leq t}{\sup} W_p^p(m_1(s), m_2(s)),
    \end{equation*} which allows to deduce the strong uniqueness because if $m_1 \equiv m_2$ then $\bar{X}_t^1 = \bar{X}_t^2$ a.s., and the processes $\bar{X}^1$ and $\bar{X}^2$ are indistinguishable because they are a.s. continuous.
    \bigbreak
    
    Let $q \geq p$. The memory kernel $L$ is supposed $W_p$-Lipschitz, so it is also $W_q$-Lipschitz. If we suppose $Law(\bar{X}_0) \in \mathcal{P}_q\left( \mathbb{R}^d \right)$, then a similar reasoning gives the existence and uniqueness of a solution in $\mathcal{C}([0, T], \mathcal{P}_q\left( \mathbb{R}^d \right))$. Since $\mathcal{C}([0, T], \mathcal{P}_q\left( \mathbb{R}^d \right)) \subset \mathcal{C}([0, T], \mathcal{P}_p\left( \mathbb{R}^d \right))$, the uniqueness of solution in $\mathcal{C}([0, T], \mathcal{P}_p\left( \mathbb{R}^d \right))$ gives that the solution $\bar{X}_t$ found in $\mathcal{C}([0, T], \mathcal{P}_p\left( \mathbb{R}^d \right))$ is indeed in $\mathcal{C}([0, T], \mathcal{P}_q\left( \mathbb{R}^d \right))$.
    \end{proof}
    
\paragraph{Acknowledgement.} The author would like to thank Pierre Monmarché and Milica Tomašević for their valuable guidance on the subject, the questions to cover, and for the technical discussions.

\bibliographystyle{alpha}
\bibliography{bibli}

\end{document}